\documentclass[12pt]{amsart} 
\pdfoutput=1 

\usepackage[margin=1in]{geometry}

\usepackage{ifpdf}

\usepackage{amsmath}
\usepackage{amsthm}
\usepackage{amsfonts}
\usepackage{amssymb}

\ifpdf
    \usepackage[colorlinks,citecolor=blue,linkcolor=blue]{hyperref}
\fi
\usepackage{amscd}	
\usepackage{cancel}	
\usepackage{microtype} 

\usepackage{graphicx}

\usepackage{color}
\usepackage[initials]{amsrefs}

\graphicspath{}
\ifpdf
  \newcommand{\fig}[2]{
    \IfFileExists{#1.pdf_tex}{
      \def\svgwidth{#2}\input{#1.pdf_tex}
    }{
      \frame{Missing figure ``#1.pdf\_tex''}
      \message{LaTeX Warning: Missing figure ``#1.pdf\_tex'' on input line \the\inputlineno}
    }
  }
\else
    \IfFileExists{#1.eps_tex}{
      \def\svgwidth{#2}\input{#1.eps_tex}
    }{
      \frame{Missing figure ``#1.eps\_tex''}
      \message{LaTeX Warning: Missing figure ``#1.eps\_tex'' on input line \the\inputlineno}
    }
  }
\fi

\newcommand{\dimension}{n}

\newcommand{\dist}{\operatorname{dist}}

\newcommand{\supp}{\operatorname{supp}}
\newcommand{\interior}{\operatorname{int}}
\newcommand{\trace}{\operatorname{tr}}

\newcommand{\esssup}{\operatorname*{ess\,sup}}
\newcommand{\essinf}{\operatorname*{ess\,inf}}

\newcommand{\argmax}{\operatorname*{arg\,max}}

\newcommand{\divo}{\operatorname{div}}

\newcommand{\dx}{\;dx}

\newcommand{\abs}[1]{\left|#1\right|}
\newcommand{\pth}[1]{\left(#1\right)}
\newcommand{\bra}[1]{\left[#1\right]}
\newcommand{\set}[1]{{\left\{#1\right\}}}

\newcommand{\at}[2]{{{\left.{#1}\right|}_{#2}}}

\newcommand{\norm}[1]{\left\|#1\right\|}

\newcommand{\cl}[1]{\overline{#1}}	

\newcommand{\al}{\ensuremath{\alpha}}
\newcommand{\de}{\ensuremath{\delta}}
\newcommand{\e}{\ensuremath{\varepsilon}}
\newcommand{\vp}{\ensuremath{\varphi}}

\newcommand{\ta}{\ensuremath{\theta}}

\newcommand{\R}{\ensuremath{\mathbb{R}}}

\newcommand{\Rd}{\ensuremath{{\mathbb{R}^{\dimension}}}}
\newcommand{\Rn}{\Rd}
\newcommand{\Z}{\ensuremath{\mathbb{Z}}}
\newcommand{\N}{\ensuremath{\mathbb{N}}}

\newcommand{\halflimsup}{\operatorname*{\star-limsup}}
\newcommand{\halfliminf}{\operatorname*{\star-liminf}}

\newcommand{\Lip}{{\rm Lip}}

\definecolor{grey}{rgb}{0.6,0.6,0.6}

\numberwithin{equation}{section}

\newtheorem{theorem}{Theorem}[section]
\newtheorem{lemma}[theorem]{Lemma}
\newtheorem{proposition}[theorem]{Proposition}
\newtheorem{corollary}[theorem]{Corollary}

\newtheorem{definition}[theorem]{Definition}

\theoremstyle{definition}

\newtheorem{remark}[theorem]{Remark}
\newtheorem{example}[theorem]{Example}

\usepackage{paralist}

\newcommand{\Wulff}{\operatorname{Wulff}}
\newcommand{\nbd}[1]{{\mathcal U}^{#1}}
\newcommand{\graph}{\operatorname{graph}}
\newcommand{\sign}{\operatorname{sign}}

\newcommand{\Span}{\operatorname{span}}

\newcommand{\domain}{\operatorname{\mathcal{D}}}

\newcommand{\SE}{E^{\rm sl}}
\newcommand{\SW}{W^{\rm sl}}
\newcommand{\Pair}{\operatorname{Pair}}
\newcommand{\aff}{\operatorname{aff}}
\newcommand{\ri}{\operatorname{ri}}
\newcommand{\epi}{\operatorname{epi}}

\newcommand{\loc}{{\rm loc}}
\newcommand{\CH}{{CH}}
\newcommand{\SlCH}{CH^{\rm sl}}
\newcommand{\facet}[1]{#1_-^c \cap #1_+^c}
\newcommand{\TT}{{\mathcal T}}

\title{A level set crystalline mean curvature flow of surfaces}

\author[Y. Giga]{Yoshikazu Giga}
\address[Y. Giga]{Graduate School of Mathematical Sciences, University of Tokyo, 3-8-1 Komaba Meguro-ku,
Tokyo 153-8914, Japan.}
\author[N. Po\v{z}\'{a}r]{Norbert Po\v{z}\'{a}r}
\address[N. Po\v{z}\'{a}r]{Faculty of Mathematics and Physics, Institute of Science and Engineering, Kanazawa University,
Kakuma town, Kanazawa, Ishikawa 920-1192, Japan.}
\email{npozar@se.kanazawa-u.ac.jp}

\date{\today}


\begin{document}

\begin{abstract}
We introduce a new notion of viscosity solutions for the level set formulation of the motion by
crystalline mean curvature in three dimensions.
The solutions satisfy the comparison principle, stability with respect to an approximation by
regularized problems, and we also show the uniqueness and existence of a level set
flow for bounded crystals.
\end{abstract}

\maketitle

\tableofcontents
\section{Introduction}

A crystalline mean curvature flow is a typical example of an anisotropic mean curvature flow, which
can be regarded as a mean curvature flow under a Minkowski or Finsler metric \cite{BP96}.
A crystalline mean curvature flow was proposed by S.~B.~Angenent and M.~E.~Gurtin \cite{AG89} and
independently by J.~Taylor \cite{T91} to describe the motion of an anisotropic antiphase boundary in materials science.
There is a large amount of literature devoted to the study of the motion by crystalline mean curvature.
However, even local-in-time unique solvability of its initial value problem has been a
long-standing open problem except in the case of planar motion or convex initial data.
The main reason is that the surface energy density is not smooth and hence the speed of evolution is determined by a nonlocal quantity.

Our goal in this paper is to solve this long-standing open problem for purely crystalline mean
curvature flow in $\R^3$.
In fact, we shall introduce a new notion of solutions which corresponds to a generalization of a
level set flow for the mean curvature flow equation and establish its unique existence.

To motivate the problem, let us explain an example of anisotropic mean curvature flow equation and
its level set formulation; see e.g.\ \cite{CGG,GG92,G06}.
Let $\gamma:S^2 \to (0,\infty)$ be a given interfacial energy density on the unit sphere $S^2$.
For a given closed surface $\Gamma$ we define the interfacial energy
\[
I_\gamma(\Gamma) = \int_\Gamma \gamma(\mathbf{n})\;d\mathcal{H}^2,
\]
and call $I_\gamma$ the interfacial energy of $\Gamma$ with density $\gamma$.
Here $\mathbf{n}$ denotes the unit exterior normal of $\Gamma$ and $d\mathcal{H}^2$ denotes the area element.
The anisotropic mean curvature $\kappa_\gamma$ is the first variation of $I_\gamma$ with respect to change of volume enclosed by $\Gamma$.
Its explicit form is
\[
\kappa_\gamma = -\operatorname{div}_\Gamma \left(\nabla_p \gamma(\mathbf{n})\right)
\]
where $\gamma$ is $1$-homogeneously extended as
$\gamma(p)=|p|\gamma\left(p/|p|\right)$ for $p\in \R^3\backslash\{0\}$ and $\gamma(0)=0$; $\operatorname{div}_\Gamma$ denotes the surface divergence \cite{Si83,G06}.
If $\gamma(p)=|p|$, $I_\gamma$ is the surface area and $\kappa_\gamma=-\operatorname{div}_\Gamma
\mathbf{n}$, which is nothing but (two times) the classical mean curvature.
When the interfacial energy density $\gamma$ is not a constant function on $S^2$, we say $\kappa_\gamma$ is an anisotropic mean curvature.
Let $\{\Gamma_t\}_{t>0}$ be a smooth family of closed surfaces in $\R^3$ and let $V$ be its normal velocity in the direction of $\mathbf{n}$.
The equation for $\{\Gamma_t\}$ of the form
\[
V = \kappa_\gamma \quad \text{on} \quad \Gamma_t
\]
is a simple example of an anisotropic mean curvature flow equation.
Of course, if $\gamma(p)=|p|$, then this equation is nothing but the standard mean curvature flow equation $V=\kappa$.
A typical feature of this equation is that even if one starts with a smooth surface $\Gamma_0$, the solution $\Gamma_t$ may pinch in finite time, for example a dumbbell with thin neck \cite{Gr89}.
So a weak formulation is necessary to track the evolution after the formation of singularities.
There are two standard approaches for the (isotropic) mean curvature flow equation.
One is a variational way like a varifold solution initiated by K.~Brakke \cite{B78} and developed
further by T.~Ilmanen \cite{Il93} and K.~Takasao and Y.~Tonegawa \cite{TT}.
Another approach is a level set method based on a comparison principle introduced by \cite{CGG,ES}.
As already noted in \cite{CGG} the level set method is very flexible and it applies to anisotropic
curvature flow equation \cite{GG92} while a varifold solution is still limited to the isotropic mean curvature flow equation.

Let us explain the idea of the level set formulation.
We introduce an auxiliary function $u:\R^3\times [0,\infty)\to\R$ so that its zero level set agrees with $\Gamma_t$.
To fix the idea we assume that $u>0$ in a region $D_t$ enclosed by $\Gamma_t$ and $u<0$ outside of $D_t \cup \Gamma_t$.
Then the equation $V=\kappa_\gamma$ is represented as
\[
\frac{u_t}{|\nabla u|}=-\operatorname{div}\left(\nabla_p \gamma\left(-\frac{\nabla u}{|\nabla u|}\right)\right) \quad \text{on} \quad \Gamma_t
\]
since $V=u_t/|\nabla u|$, $\mathbf{n}=-\nabla u/|\nabla u|$.
The idea of the level set method is to consider this equation not only on $\Gamma_t$ but also in
$\R^3$, i.e.\ each level set of $u$ is required to move by $V=\kappa_\gamma$.
In other words, we consider
\begin{equation}
\label{level set mcf}
u_t-|\nabla u|\left(-\operatorname{div}\left(\nabla_p \gamma \left(-\nabla u/|\nabla u|\right)\right)\right) = 0
\quad \text{in} \quad \R^3 \times (0,\infty)
\end{equation}
with initial condition
\begin{equation}
\label{level set mcf initial}
u(x,0)=u_0(x), \ x \in \R^3.
\end{equation}
Here $u_0$ is taken so that $\Gamma_0$ is its zero level set.
In the case $\gamma(p)=|p|$, \eqref{level set mcf} is nothing but the famous level set mean curvature flow equation
\[
u_t - |\nabla u|\operatorname{div} \left(\nabla u/|\nabla u|\right)=0.
\]
The level set equation \eqref{level set mcf} is degenerate even if $\gamma$ is convex.
It is unexpected that the problem can be solved even locally-in-time in classical sense even if $u_0$ is smooth.

Fortunately, if $\gamma$ is $C^2$ on $\R^3 \setminus \set0$ and convex, the notion of viscosity
solutions \cite{CIL} is adjustable to solve \eqref{level set mcf}--\eqref{level set mcf initial} uniquely and globally-in-time for any uniformly continuous initial data \cite{CGG,G06}.
One shall notice that there is a large freedom to choose $u_0$ for given $\Gamma_0$.
However, it is known \cite{CGG,G06} that the zero level set is uniquely determined by $\Gamma_0$
(independently of the choice of $u_0$).
Although the zero level set of $u$ may fatten, it is often called a level set flow (solution) of $V=\kappa_\gamma$ with initial data $\Gamma_0$.
The theory is based on a comparison principle for viscosity solutions and it applies when $\gamma$ is not necessarily $C^2$ but the singularity is weak.
For example, in planar motion even if the second derivative of $\gamma\in C^1\left(\R^2 \backslash\{0\}\right)$ is allowed to jump at finitely many point in $S^1$, the result of \cite{CGG} is extendable \cite{OS93,GSS}; see \cite{I96} for higher dimensional problem.
However, if the singularity of $\gamma$ is strong, such that the first derivative of $\gamma$ may have jumps, then the situation is completely different.
The equation becomes very singular in the sense that the speed becomes a nonlocal quantity and
establishing the level set method becomes totally non-trivial even if only a planar motion is considered, although it has been established in \cite{GG01}.
However, it has been a long-standing open problem for surface evolution even if $\gamma$ is
(purely) crystalline, i.e.\ $\gamma$ is piecewise linear and convex in $\R^3$.
Such functions are often in convex analysis referred to as \emph{polyhedral} \cite{Rockafellar}.

Our purpose is to establish a level set method for a crystalline mean curvature flow, whose typical example includes $V=\kappa_\gamma$ for crystalline $\gamma$.
Our theory can apply to more general equations such as $V=\kappa_\gamma+1$.
We shall introduce a new notion of viscosity solutions so that the following well-posedness result holds.

\begin{theorem}[Unique existence]
\label{th:unique existence}
Let $\gamma$ be crystalline in $\R^3$.
Assume that $f=f(m,\lambda)$ is continuous on $S^2\times\R^3$ and $\lambda 	\mapsto f(m,\lambda)$ is non-decreasing.
Assume that $\left|f(m,\lambda)\right|/\left(|\lambda|+1\right)$ is bounded in $S^2\times\R$.
Let $D_0$ be a bounded open set in $\R^3$ with the boundary $\Gamma_0=\partial D_0$.
Then there exists a global unique level set flow $\{\Gamma_t\}_{t\geq 0}$ with
\begin{equation}
\label{general mcf}
V=f(\mathbf{n},\kappa_\gamma) \quad \text{on} \quad \Gamma_t
\end{equation}
and initial data $\Gamma_0$.
\end{theorem}

The assumption of the linear growth for $f$ in $\lambda$ is just for simplicity.
One can remove it by introducing a special class of test functions \cite{IS,G06} or by a flattening argument \cite{Go}.

To prove the uniqueness part a key step is to establish a comparison principle for the level set
equation of \eqref{general mcf} which is of the form
\begin{equation}
\label{P}
u_t + F \left( \nabla u, \operatorname{div}\partial W(\nabla u) \right) = 0,
\end{equation}
where
\begin{align}
\label{geometric F}
F(p,\Lambda) = -|p|f \left(-p/|p|,\Lambda \right),\ W(p)=\gamma(-p).
\end{align}
Here we rather use the subdifferential notion $\partial W$ instead of $\nabla W$ since $W$ is
piecewise linear and so not everywhere differentiable.
To prove the existence part, one cannot unfortunately apply Perron's method since the nonlocal
quantity ``$\operatorname{div}\partial W(\nabla u)$'' is not constant in a flat part of the
solution (which is different from planar case.) We thus construct a solution by smoothing $W$.
Here we need to establish a stability of our viscosity solutions.
The basic idea of proofs is an elaboration on the idea for establishing uniqueness based on the
comparison principle and stability for the total variation flow of non-divergence type \cite{GGP13JMPA,GGP13AMSA}.
We shall establish comparison principle for a more general nonlinearity $F$ than \eqref{geometric
F}, see Remark~\ref{rem:general F} below.

The bibliography of \cite{GGP13AMSA} includes many references on unique solvability. We take this
opportunity to mention related results for evolution of closed surfaces by crystalline or
more general singular interface energy.
In three dimensions and higher, the crystalline mean curvature $\kappa_\gamma$ is not only a
nonlocal quantity as mentioned above, but it might be non-constant on facets of the crystal \cite{BNP99}. In fact, it might be
discontinuous, and in general it is known to be only a function of bounded variation
\cite{BNP01a,BNP01b}.
Therefore facet breaking and bending might occur and we cannot restrict the solutions only to
surfaces with facets parallel to those of the Wulff shape corresponding to the crystalline energy
density $\gamma$. A more general notion of solutions is necessary. The variational approach have
led to a significant progress by understanding the properties of $\kappa_\gamma$. A notion of solutions via an approximation by reaction-diffusion equations for $V
= \gamma \kappa_\gamma$ was established in \cite{BGN00,BN00}. An approximation via
minimizing movements was used in \cite{CasellesChambolle06,BCCN06,BCCN09}. However, all these
results only provide existence for \emph{convex} initial data.

We also establish a convergence result which is useful to discuss approximation by an Allen-Cahn type equation.

\begin{theorem}[Convergence]
\label{th:convergence}
Under the assumption of Theorem~\ref{th:unique existence}, let $u$ be a viscosity solution of \eqref{P} with initial data
$u_0 \in C(\R^3)$ such that $u_0(x)=-c$ for $|x|\geq R$ with some $R$ and $c>0$.
Assume that $\gamma_\varepsilon$ is smooth in $\R^3\setminus\set0$, convex and $1$-homogeneous and
$\gamma_\varepsilon\to\gamma$ uniformly on $S^2$.
Let $u^\varepsilon$ be a viscosity solution of \eqref{P} with
$W=W_\varepsilon(p)=\gamma_\varepsilon(-p)$, with initial data $u_0^\varepsilon$ such that
$u_0^\varepsilon(x)=-c$ for $|x|\geq R$.
Assume that $u_0^\varepsilon\to u_0$ uniformly.
Then  $u^\varepsilon$ converges locally uniformly to $u$ in $\R^3\times [0,\infty)$.
\end{theorem}

This gives a convergence of diffuse interface model to the sharp interface model even if $\gamma$ is crystalline; see \cite{GOS,TC98}.

After this work had been completed, the authors learned of a recent work by
A. Chambolle, M. Morini and M. Ponsiglione \cite{CMP}, where they established a
unique global solvability (up to fattening) for $V=\gamma \kappa_\gamma$
for
any convex $\gamma$ by introducing a new notion of a solution related to
the
anisotropic distance function. Their approach applies to all dimension
and all
initial data not necessarily bounded. However, their approach requires a
special
form of the equation so that the mobility is proportional to the
interfacial
energy density $\gamma$ and it does not apply to $V=\kappa_\gamma$ or
$V=\kappa_\gamma +1$. Our approach applies to all $V=f(\mathbf{n},\kappa_\gamma)$
including these equations but the dimension $n$ is limited as $n \leq 3$
and
$\gamma$ is limited to crystalline. It is not yet clear whether or not our
solution agrees with theirs in the case when both approaches are available
although it is very likely.

\begin{remark}
\label{rem:general F}
In full generality, we will assume that $F \in C(\Rn \times \R)$, $n \geq 1$, and that it is nonincreasing in the second variable, that is,
\begin{align}
\label{F ellipticity}
F(p, \xi) \leq F(p, \eta) \qquad \text{for all } p \in \Rn,\ \xi \geq \eta.
\end{align}
For simplicity, we shall also assume that
\begin{align*}
F(0,0) = 0.
\end{align*}
In particular, constants are solutions of \eqref{P}.
\end{remark}

\subsection*{Viscosity solutions and the contribution of this paper}

We extend the notion of viscosity solutions to the problem \eqref{P} with crystalline $W$. The main
strength of the viscosity solution approach is that it can handle general problems that are not of
divergence form by exploiting their comparison principle structure \cite{CIL,G06}.

The main difficulty in defining a solution of \eqref{P} is the singular, nonlocal operator $\divo
\partial W(\nabla \cdot)$. We interpret this operator as the minimal section (also known as the canonical
restriction) of the subdifferential of the anisotropic total variation energy in the Hilbert space
$L^2(\Omega)$,
\begin{align*}
E(\psi) :=
\begin{cases}
\int_\Omega W(D \psi) \dx, & \psi \in L^2(\Omega) \cap BV(\Omega),\\
+\infty, & \text{otherwise},
\end{cases}
\end{align*}
where $\Omega$ is the flat torus $\R^n / L\Z^n$ for some $L > 0$, $n \geq 1$, and $BV(\Omega)$ is the space of
functions of bounded variation. That is, we only consider this
energy for periodic functions $\psi$ to avoid issues with handling the boundary of $\Omega$.
Since $D\psi$ is in general only a Radon measure, the functional $E$ is understood as the lower
semi-continuous envelope (closure) of the functional defined for Sobolev functions
$W^{1,1}(\Omega)$.

It is well-known that the subdifferential of $E$ defined above is the set of divergences of
certain vector fields, often called \emph{Cahn-Hoffman vector fields} \cite{Moll}. More precisely, if $\psi$ is a Lipschitz function on $\Omega$, then
\begin{align*}
\partial E(\psi) = \set{-\divo z: z(x) \in \partial W(\nabla \psi(x)) \text{ for a.e. $x$, } \divo
z \in L^2(\Omega)}.
\end{align*}
The subdifferential $\partial E(\psi)$ is a closed convex, possibly empty subset of the Hilbert
space $L^2(\Omega)$. If it is nonempty, we say that $\psi \in \domain(\partial E)$ and the unique
element of the subdifferential with the minimal $L^2$-norm is called the
minimal section of $\partial E(\psi)$ and is denoted as $\partial^0 E(\psi)$. In such a case we
will interpret $\divo \partial W(\nabla \psi)$ as $- \partial^0 E(\psi)$.

This interpretation is consistent with the classical theory of monotone operators for the
solvability of problems of the form
\begin{align*}
u'(t) \in - \partial E(u(t)).
\end{align*}
Indeed, it is known that a solution is right-differentiable and the right derivative $d^+u/dt (t) = -\partial^0
E(u(t))$. As we noted above, the mean curvature flow can be viewed as the gradient flow of the
surface energy functional.

The viscosity solutions are defined via a comparison with a suitable class of test functions. It is
therefore necessary to identify a sufficiently large class of functions for which we can define
$\divo \partial W(\nabla \cdot)$ so that they can serve as test functions in the definition of
viscosity solutions. In particular, it must be possible to prove both uniqueness (via a comparison principle)
and existence (via a stability property of solutions).

Since the energy density $W$ is crystalline, that is, piecewise linear, the domain of the
subdifferential of $E$ can be understood as functions that have flat parts with gradients that fall
into the set where $W$ is not differentiable. These flat parts then correspond to the features of
the crystal---facets and edges---depending on the dimension of the subdifferential $\partial
W(\nabla \psi)$ on the given flat part of $\psi$.
This then leads to an idea of energy stratification with respect to the subdifferential dimension.
It turns out that the value of $\divo \partial W(\nabla \psi)$ at a point $x$ depends only on the shape of $\psi$
in the directions parallel to $\partial W(\nabla \psi(x))$, and it is basically independent of the
shape in the orthogonal direction.

Because of the simple structure of $W$, the local behavior of $W$ (and $\partial W$) in a
neighborhood of a given gradient $p$ can be
completely captured by a one-homogeneous function that is linear in directions orthogonal to the
subspace spanned by the directions in $\partial W(p)$, Proposition~\ref{pr:direction-decomposition}.
We therefore for a given slope $p$ define a sliced energy $\SE_p$ to capture the interesting
behavior, and reduce the analysis to a space $\R^k$, where $k$ is the dimension of $\partial W(p)$.
Then we consider \emph{stratified faceted functions} by separating the variables into the directions
parallel to $\partial W(p)$, in which we assume that the function has a ``nice'' facet, and the
orthogonal directions where the function can be of any form (as long as it is differentiable),
Definition~\ref{def:strat-faceted-test-function}.

It can be easily seen that $\divo \partial W(\nabla \psi)(x) = 0$ whenever $\psi$ is twice
continuously differentiable in a neighborhood of $x$ and $W$ is differentiable at
$\nabla \psi(x)$.
We therefore have to identify the value of this operator at points where $\partial W(\nabla \psi)$
is not a singleton, that is, on the flat parts of the stratified faceted functions. These flat
parts can be thought of as $k$-dimensional facets, and they can be described by a pair of open sets $(A_-, A_+)$, which specify where the function is
below ($A_-$) or above $(A_+)$ the flat part. It turns out that $\divo \partial W(\nabla \psi)$ is
independent of the particular choice of $\psi$, Corollary~\ref{co:lambda support func indep}, but only depends on the sets $(A_-, A_+)$ and the
slope $p = \nabla \psi$ of the flat part. We call this value $\Lambda_p(\psi)$ to emphasize this
dependence on $p$, and connect this to the previous results \cite{GGP13JMPA,GGP13AMSA}, see
Section~\ref{sec:crystalline curvature}.
While $\Lambda_p(\psi)$ might be discontinuous on the flat parts, it satisfies a comparison
principle property with respect to a natural ordering of the $k$-dimensional facets.

We use the stratified faceted functions as the test functions for the definition of viscosity
solutions. Heuristically speaking, a continuous function $u$ is a viscosity solution of \eqref{P}
if it satisfies a comparison principle with all stratified faceted functions that are local
solutions of \eqref{P}.

To show that this definition of viscosity solutions is reasonable, we have to establish a general
comparison principle and stability of solutions (with respect to approximation by regularized
problems). For the comparison principle, we need a sufficiently large class of stratified faceted
test functions. In particular, for any given gradient $p$ such that $\partial W(p)$ is not a
singleton and a pair of smooth disjoint open sets $(A_-, A_+)$ in
$\R^k$, $k = \dim \partial W(p)$, we need
to be able to construct a $k$-dimensional facet arbitrarily close to the facet given by $(A_-, A_+)$
such that there exists a stratified faceted function with this facet, and for which
$\Lambda_p(\psi)$ is well-defined. See Corollary~\ref{co:approximate pair sliced} for details. This unfortunately seems to be quite nontrivial, and we
currently know how to do this construction in one and two dimensions. This allows us to prove the
comparison principle for \eqref{P} in three dimensions. However, if this approximated admissible
facet construction in Corollary~\ref{co:approximate pair sliced} can be extended to
higher dimensions, our results Theorem~\ref{th:unique existence} and Theorem~\ref{th:convergence} will automatically apply to the higher dimensions as well.

The proof of the comparison principle Theorem~\ref{th:comparison principle} follows the standard doubling-of-variables argument with an
additional parameter as in \cite{GGP13JMPA,GGP13AMSA}. This is substantially extended to
handle the stratified energy and the stratified faceted test functions.
We consider two solutions $u$, $v$ of \eqref{P} that are ordered as $u \leq v$ at $t = 0$ and consider the
function
\begin{align*}
\Phi_{\zeta,\e}(x,t,y,s) := u(x,t) - v(y,s) - \frac{\abs{x-y-\zeta}^2}{2\e}-
    S_\e(t,s),
\end{align*}
on $(x, t, y, s) \in \R^n \times (0, T) \times \R^n \times (0, T)$, where $S_\e$ is defined in
\eqref{Se}, and $T, \e > 0$ are fixed. We then analyze the maxima of $\Phi_{\zeta, \e}$ for $\zeta
\in \R^n$ small. This extra parameter $\zeta$ allows us to recover additional information about the
behavior of $u$ and $v$ near the maximum of $\Phi_{\zeta, \e}$. We then argue by contradiction: if
$u > v$ at some point, we can construct stratified faceted test functions for $u$ and $v$ near the
maximum of $\Phi_{\zeta, \e}$. These test functions have ordered facets, which then together with
the comparison principle for $\Lambda_p$ yields a contradiction.

The stability of solutions with respect to approximation of \eqref{P} by regularized problems then
follows from an extension of the argument developed in \cite{GGP13JMPA}. We have to again overcome the
discrepancy between the test functions of the regularized problem, which are only smooth functions,
and the stratified faceted functions for the limit problem \eqref{P}. This is related to the fact
that we are approximating a singular, nonlocal operator by local operators. The idea is to perturb the
test function by solving the resolvent problem for the energy $E$ and the regularized (elliptic) energy $E_m$
with a small parameter $a > 0$:
\begin{align*}
\psi_a = (I + a \partial E)^{-1} \psi, \qquad \psi_{a,m} = (I + a \partial E_m)^{-1} \psi,
\end{align*}
which amounts to solving one step of the implicit Euler discretization of the gradient flow of
those energies. This transfers the nonlocal information onto the perturbed test function and
allows passing in the limit, Theorem~\ref{th:stability quadratic}. The main extension in this paper is the handling of the sliced
energy. An elaboration on this argument yields also stability with respect to an approximation by
one-homogeneous energies $E_m$, Theorem~\ref{th:linear growth stability}.

Combining the above results we obtain the existence of a unique solution of \eqref{P}. Since the
level set of the solution does not depend on the choice of the initial level set function, we
have uniqueness of the level set flow.

\subsection*{Outline}

We open with a review of the theory for convex functionals with linear growth in
Section~\ref{sec:convex func lin growth}. This will allow us to introduce the idea of energy
stratification and the slicing of the energy density $W$ according to its features,
Section~\ref{sec:energy-stratification}. We then define the crystalline mean curvature $\Lambda$ on various features
of the evolving surface such as edges and facets, Section~\ref{sec:crystalline curvature}, and
establish its properties, including a comparison principle.
At this point we introduce the notion of viscosity solutions, Section~\ref{sec:viscosity
solutions}, and construct faceted test functions in Section~\ref{sec:faceted functions}.
The comparison principle for viscosity solutions is established in Section~\ref{sec:comparison principle}, followed by the stability results, Section~\ref{se:stability}.
Finally, the main result on the well-posedness of \eqref{P} is presented in
Section~\ref{sec:well-posedness}.

\section{Convex functionals with linear growth}
\label{sec:convex func lin growth}

There are a considerable number of publications on the topic of convex functionals with linear
growth, see \cite{ACM} for a list of references.
In this section we review the rather standard notation and results that we will use
throughout the paper, and prove two important lemmas that will allow us to better
understand the crystalline mean curvature later.

Suppose that $W: \R^d \to \R$, $d \geq 1$, is a convex function that satisfies the growth condition
\begin{align}
\label{growth-condition}
\abs{W(p)} \leq M (1 + \abs{p}), \qquad p \in \R^d,
\end{align}
for some $M > 0$.
Note that it is usually also assumed that $W(p) \geq c\abs{p}$ for some $c > 0$, or that $W(p) = W(-p)$,
but we make no such assumption since they are unnecessary for our purposes,
and in fact we need the generality.

Let $\Omega$ be either
an bounded open subset of $\R^d$ or the $d$-dimensional flat torus $\R^d / L \Z^d$.
We are interested in the functional $E_W(\cdot; \Omega): L^2(\Omega) \to \R$ defined as
\begin{align}
\label{EW}
E_W(\psi; \Omega) &=
\begin{cases}
\int_\Omega W(D\psi) & \psi \in L^2(\Omega) \cap BV(\Omega),\\
+\infty & \text{otherwise}
\end{cases}
\end{align}
that is understood as the \emph{relaxation}
(also the \emph{closure} or the \emph{lower semi-continuous envelope}) of the functional
\begin{align}
\label{functional-w11}
\psi \mapsto
\begin{cases}
\int_\Omega W(\nabla\psi) & \psi \in L^2(\Omega) \cap W^{1,1}(\Omega),\\
+\infty & \text{otherwise}.
\end{cases}
\end{align}

The relaxed functional $E_W$ can be expressed more explicitly following \cite{GiaquintaModicaSoucek,BouchitteDalMaso}.
Indeed, we introduce the recession function of $W$,
\begin{align*}
W0^+(p) = \lim_{\lambda\to 0+} \lambda W(\lambda^{-1} p),
\end{align*}
which is a positively one-homogeneous convex function on $\R^d$
due to the growth condition \eqref{growth-condition}.
If $W$ is one-homogeneous itself, we have $W0^+ = W$.

For $\psi \in BV(\Omega)$,
$\nabla\psi$ will denote the Radon-Nikod\'{y}m derivative of the absolutely continuous part of $D\psi$
with respect to the Lebesgue measure
$L^d \lfloor \Omega$ and $D^s \psi$ will be the singular part.
Then we have
\begin{align*}
D \psi = \nabla \psi L^d \lfloor \Omega + D^s \psi,
\end{align*}
and we can write $E_W$ as
\begin{align}
\label{ew-decomp}
E_W(\psi; \Omega) =
\int_\Omega W(\nabla \psi) \dx + \int_\Omega W0^+\pth{\frac{D^s \psi}{\abs{D^s \psi}}} \;d\abs{D^s\psi},
\end{align}
where $\frac{D^s \psi}{\abs{D^s \psi}}$ is the Radon-Nikod\'{y}m derivative of $D^s\psi$ with respect to $\abs{D^s\psi}$.
We note that if $\psi \in L^2(\Omega) \cap W^{1,1}(\Omega)$, or even $\psi \in \Lip(\Omega)$, then
this formula simplifies to \eqref{functional-w11} since $D^s \psi = 0$.

\subsection{Subdifferentials}

Since $E_W(\cdot; \Omega)$ is a proper closed (that is, lower semi-continuous) convex functional on $L^2(\Omega)$,
its subdifferential
\begin{align*}
\partial E_W(\psi; \Omega) = \set{
v \in L^2(\Omega):
E_W(\psi + h; \Omega) - E_W(\psi; \Omega) \geq (h, v)
\text{ for all $h \in L^2(\Omega)$}}
\end{align*}
is a closed convex, possibly empty subset of the Hilbert space $L^2(\Omega)$ equipped with the
inner product $(h, v) := \int_\Omega h v \dx$.
If $\partial E_W(\psi; \Omega)$ is nonempty,
we say that $\psi \in \domain(\partial E_W(\cdot; \Omega))$,
the \emph{domain} of the subdifferential,
and we define
the \emph{minimal section} (also known as the \emph{canonical restriction})
$\partial^0 E_W(\psi; \Omega)$ of the subdifferential as the unique element
of $\partial E_W(\psi; \Omega)$ with the minimal norm in $L^2(\Omega)$.

The characterization of the subdifferential of $E_W$ is well-known when
$W$ is a positively one-homogeneous function, that is, when
\begin{align*}
W(tp) = t W(p) \qquad t \geq 0.
\end{align*}
We will need this characterization for Lipschitz functions only,
and we therefore present it in this simplified settings.
Let $\Omega$ be an open subset of $\R^d$ or a $d$-dimensional torus $\R^d / L \Z^d$ for some $L > 0$.
Following \cite{Anzellotti}, let us introduce the space of vector fields with $L^2$ divergence,
\begin{align*}
X_2(\Omega) = \set{z \in L^\infty(\Omega; \R^d): \divo z \in L^2(\Omega)}.
\end{align*}
For given $\psi \in \Lip(\Omega)$, we define the set of \emph{Cahn-Hoffman vector fields}
on $\psi$ as
\begin{align}
\label{cahn-hoffman}
\CH_W(\psi; \Omega) := \set{z \in X_2(\Omega): z(x) \in \partial W(\nabla \psi(x)) \text{ a.e. $x \in \Omega$}}.
\end{align}
Note that the set
\begin{align}
\label{div-cahn-hoffman}
\divo \CH_W(\psi; \Omega) := \set{\divo z: z \in \CH_W(\psi; \Omega)}
\end{align}
is a closed convex, possibly empty subset of $L^2(\Omega)$.
We have the well-known characterization of the subdifferential of $E_W$ in the periodic case, see
\cite[Section~1.3]{ACM} or \cite{Moll}.
\begin{proposition}
\label{pr:subdiff-char-periodic}
Let $\Omega = \R^d / L \Z^d$ for some $d \in \N$ and $L > 0$, and assume that
$W$ is a positively one-homogeneous convex function on $\R^d$.
If $\psi \in \Lip(\Omega)$ then
\begin{align*}
\partial E_W(\psi; \Omega) = \set{-\divo z: z \in \CH_W(\psi; \Omega)} =
-\divo \CH_W(\psi; \Omega).
\end{align*}
\end{proposition}

\begin{remark}
If $\Omega$ is a bounded open subset of $\R^d$ with a Lipschitz boundary,
then the subdifferential is given by the vector fields
$z \in \CH_W(\psi; \Omega)$ such that $[z \cdot \nu] = 0$ on $\partial\Omega$;
see \cite{ACM} for details.
We will work on periodic domains to not have to deal with this technicality.
We will see later (Lemma~\ref{le:cahn-hoffman-patch} and Properties~\ref{pr:lambda-well-defined}) that this does not change the value of the crystalline curvature on the facet.
\end{remark}

Let us also mention one trivial result concerning the subdifferential of one-homogeneous
convex functions on $\R^d$.

\begin{lemma}
\label{le:one-homogeneous-subdiff}
Suppose that $W$ is positively one-homogeneous convex function on $\R^d$.
Then $\partial W(p) \subset \partial W(0)$ for any $p \in \R^d$.
We also have $(x - y) \perp p$ for any $x, y \in \partial W(p)$ and any $p \in \R^d$.
\end{lemma}

\subsection{The resolvent problem and the approximation by regularized functionals}
\label{sec:resolvent-approximation}

Let $W$ be a convex function satisfying the growth condition \eqref{growth-condition}.
For some flat torus $\Gamma = \R^d / L \Z^d$, $d \geq 1$,
we want to approximate $E_W(\cdot; \Gamma)$ defined in \eqref{EW} by
certain regularized functionals.

Suppose therefore that $\set{W_m}_{m\in \N}$ is a sequence of convex functions on $\R^d$
that satisfies the following:
\begin{enumerate}
\item $\set{W_m}_{m\in\N}$ is a decreasing sequence,
\item $W_m \in C^2(\R^d)$,
\item $W_m \searrow W$ as $m\to\infty$ locally uniformly on $\R^d$,
\item there exist positive numbers $a_m$ such that
$a_m^{-1} I \leq \nabla_p^2 W_m(p) \leq a_m I$ for all $p \in \R^d$, $m\in\N$,
where $I$ is the $d\times d$ identity matrix.
\end{enumerate}

We introduce the regularized functionals
\begin{align*}
E_m(\psi; \Gamma) := \begin{cases}
\int_\Rn W_m(\nabla \psi) \dx & \psi \in H^1(\Gamma),\\
+\infty & \psi \in L^2(\Gamma) \setminus H^1(\Gamma),
\end{cases}
\end{align*}
where $H^k(\Gamma) := W^{k,2}(\Gamma)$ is the standard Sobolev space of $L\Z^d$-periodic functions.

Let us give an example of a regularized $W_m$ first.

\begin{example}
\label{ex:wm-example}
Let $\eta_m$ be the standard mollifier
with support of radius $1/m$.
Define the smoothing
\begin{align*}
W_m(p) = (W * \eta_m)(p) + \frac 1{2m} \abs{p}^2 \qquad p \in \Rn.
\end{align*}
By convexity we have $W_m \geq W$ and $W_m$ convex,
$W_m \in C^\infty(\R^d)$,
$\nabla^2 W_m \geq \frac 1m I$
and $W_m \searrow W$ as $m\to0$ locally uniformly.
The uniform upper bound on $\nabla^2 W_m$ follows immediately from $\partial_{p_i p_j}(W * \eta_m)
= \partial_{p_i} W * \partial_{p_j} \eta_m$, and the right-hand side is bounded since $\nabla
W$ is bounded.
\end{example}

We need the following result similar to \cite[Proposition~5.1]{GGP13JMPA}.

\begin{proposition}
\label{pr:energy-convergence}
\begin{enumerate}
\item $E_m(\cdot; \Gamma)$ form
a decreasing sequence of proper closed convex functionals on $L^2(\Gamma)$.
\item The subdifferential $\partial E_m$ is a singleton for all
\begin{align*}
\psi \in \domain(\partial E_m) = H^2(\Gamma)
\end{align*}
containing the unique element
\begin{align*}
-\trace \bra{\pth{\nabla_p^2 W_m} \pth{\nabla \psi} \nabla^2 \psi} \qquad \text{a.e.}
\end{align*}
\item
$(\inf_m E_m(\cdot; \Gamma))_* = E_W(\cdot; \Gamma)$, the lower semi-continuous envelope of $\inf_m E_m$ in $L^2(\Gamma)$.
\end{enumerate}
\end{proposition}

\begin{proof}
For (a) and (b) see \cite[Section~9.6.3]{Evans}.

(c): $\set{E_m(\cdot; \Gamma)}$ is decreasing since $\set{W_m}$ is decreasing.
Therefore $E_m(\psi; \Gamma) \to E_W(\psi; \Gamma)$
for any $\psi \in H^1(\Gamma)$ by the Dominated convergence theorem,
since $E_W$ is of the form \eqref{functional-w11} in this case.
If $\psi \notin H^1(\Gamma)$, $E_m(\psi; \Gamma) = \infty$ by definition
and therefore $E_W(\cdot; \Gamma) \leq \inf_m E_m(\cdot; \Gamma)$,
with equality on $H^1(\Gamma)$.
Let us now denote $F(\psi) = \inf_m E_m(\psi; \Gamma)$.
By a standard approximation result,
for any $\psi \in BV(\Gamma)$
there exists a sequence $\set{\psi_k} \subset C^\infty(\Gamma) \cap BV(\Gamma) \subset H^1(\Gamma)$
such that
$\psi_k \to \psi$ in $L^2(\Omega)$
and $\int_\Gamma \abs{D\psi_k} \to \int_\Gamma \abs{D\psi}$, which yields
$E_W(\psi_k; \Gamma) \to E_W(\psi; \Gamma)$ due to \cite{Resetnjak};
see \cite{GiaquintaModicaSoucek}.
In particular,
\begin{align*}
F_*(\psi) \leq \liminf_{k\to\infty} F(\psi_k) = \liminf_{k\to\infty} E_W(\psi_k; \Gamma)
= E_W(\psi; \Gamma).
\end{align*}
Hence $F_* = E_W$ by the lower semi-continuity of $E_W$.
\end{proof}

We will need the following approximation and convergence result
for the resolvent problems.

\begin{proposition}
\label{pr:resolvent-problems}
For $\psi \in \Lip(\Gamma)$,
and $m \in \N$, $a > 0$, the resolvent
problems
\begin{align*}
\psi_a + a \partial E_W(\psi_a; \Gamma) \ni \psi,\\
\psi_{a,m} + a \partial E_m(\psi_{a,m}; \Gamma) \ni \psi,
\end{align*}
admit unique solutions $\psi_a$ and $\psi_{a,m}$ in $L^2(\Gamma)$, respectively.
Moreover, $\psi_a$ and $\psi_{a,m}$ are Lipschitz continuous
and
\begin{align*}
\norm{\nabla \psi_a}_\infty, \norm{\nabla \psi_{a,m}}_\infty \leq
\norm{\nabla \psi}_\infty.
\end{align*}
Finally, $\psi_{a,m} \in C^{2,\al}(\Gamma)$ for some $\al = \alpha_m > 0$.

We also introduce the functions
\begin{align*}
h_a := \frac{\psi_a - \psi}{a}, &&&
h_{a,m} := \frac{\psi_{a,m} - \psi}{a} = -\trace\bra{(\nabla_p^2 W_m)(\nabla \psi_{a,m}) \nabla^2 \psi_{a,m}}.
\end{align*}
Then, for fixed $a>0$,
\begin{align*}
\psi_{a,m} &\rightrightarrows \psi_a && \text{uniformly as $m \to \infty$, and},\\
h_{a,m} &\rightrightarrows h_a && \text{uniformly as $m \to \infty$}.
\intertext{Moreover,}
\psi_a &\rightrightarrows \psi && \text{uniformly as $a \to 0$.}
\intertext{If furthermore $\psi \in \domain\pth{\partial E_W(\cdot; \Gamma)}$ then
also}
h_a &\to -\partial^0 E(\psi; \Gamma)&& \text{in $L^2(\Gamma)$ as $a \to 0$.}
\end{align*}
\end{proposition}

\begin{proof}
We follow the proof of \cite[Proposition~5.3]{GGP13JMPA}.
Due to Proposition~\ref{pr:energy-convergence}(a), \cite[Theorem~3.20]{Attouch} implies
the \emph{Mosco convergence} of $E_m$ to $E$.
This yields the resolvent convergence \cite[Theorem~3.26]{Attouch},
namely, for fixed $a > 0$ we have
\begin{align}
\label{res-l2-conv}
\psi_{a,m} \to \psi_a  \quad \text{in $L^2(\Gamma)$.}
\end{align}

The $C^{2,\alpha}$ regularity of $\psi_{a,m}$ is standard from the elliptic theory,
as $I + a \partial^0 E_m(\cdot; \Gamma)$ is a quasilinear uniformly elliptic operator
as noted in Proposition~\ref{pr:energy-convergence}.

Since the $E_m$-resolvent problem is translation invariant and has a maximum principle,
we find that $\psi_{a,m}$ is Lipschitz since $\psi$ is Lipschitz, and
\begin{align*}
\norm{\nabla \psi_{a,m}}_\infty \leq \norm{\nabla \psi}_\infty.
\end{align*}
Therefore the Arzel\'{a}-Ascoli theorem and \eqref{res-l2-conv} yield
the uniform convergence of
$\psi_{a,m} \to \psi_a$ and $h_{a,m} \to h_a$ as $m\to\infty$ for fixed $a > 0$,
and hence also the Lipschitz bound $\norm{\nabla \psi_a}_\infty \leq \norm{\nabla \psi}_\infty$.
Moreover, since the $E_m$-resolvent problem has a maximum principle,
the $E_W$-resolvent problem has a maximum principle as well.

Finally, a standard result implies that $\psi_a \to \psi$ in $L_2(\Gamma)$ as $a\to0$
\cite[Theorem~3.24]{Attouch},
therefore with Arzel\'{a}-Ascoli and the uniform Lipschitz bound we conclude that
$\psi_a \to \psi$ uniformly.
If furthermore $\psi \in \domain(\partial E_W(\cdot; \Gamma))$,
also $h_a \to -\partial^0 E_W(\psi; \Gamma)$ \cite[Proposition~3.56]{Attouch}.
\end{proof}

We give a lemma on the Mosco convergence of functionals with linear growth.

\begin{lemma}
\label{le:lingrowthapproximation}
Suppose that $W_m$ are convex positively one-homogeneous functions such that $W_m \rightrightarrows
W$ uniformly on the unit ball. Then $E_m(\psi) = \int_\Gamma W_m(\nabla \psi)$ Mosco-converges to
$E(\psi) = \int_\Gamma W(D\psi)$ as $m \to \infty$.
\end{lemma}

\begin{proof}
By \cite[Proposition~3.19]{Attouch}, we need to show that for every $\psi$, $\psi_m
\stackrel{w}{\to} \psi$ weakly in $L^2(\Gamma)$ we have
$E(\psi) \leq \liminf_m E_m(\psi_m)$ and that for every $\psi \in L^2(\Gamma)$ there exists a sequence $\psi_m \to
\psi$ strongly in $L^2(\Gamma)$ such that $E(\psi) = \lim_m E_m(\psi_m)$.

If $\psi_m \stackrel{w}\to \psi$ weakly in $L^2(\Gamma)$, we can deduce $E(\psi) \leq \liminf_m
E_m(\psi_m)$ from the formula \cite{AB}
\begin{align*}
E(\psi) := \sup \Big\{\int_\Gamma \psi \divo \varphi: &\varphi \in
C^1(\Gamma), \norm{\varphi}_\infty \leq 1,\\ &\varphi(x) \cdot p \leq 1 \text{ whenever } W(p) \leq
1, x \in \Gamma\Big\}.
\end{align*}

By a standard approximation result,
for any $\psi \in BV(\Gamma)$
there exists a sequence $\set{\psi_k} \subset C^\infty(\Gamma) \cap BV(\Gamma) \subset W^{1, 1}(\Gamma)$
such that
$\psi_k \to \psi$ in $L^2(\Omega)$
and $\int_\Gamma \abs{D\psi_k} \to \int_\Gamma \abs{D\psi}$, which yields
$E(\psi_m) \to E(\psi)$ by the theorem of Re\v{s}etnjak \cite{Resetnjak}.
On the other hand, by the uniform convergence of $W_m$ to $W$ on the unit ball we have for any
$\xi \in W^{1,1}(\Gamma)$
\begin{align*}
\abs{\int_\Gamma W_m(\nabla\xi) - W(\nabla \xi) \dx} \leq &\int_\Gamma |\nabla \xi|
\abs{W_m\pth{\frac{\nabla\xi}{|\nabla \xi|}} - W\pth{\frac{\nabla\xi}{|\nabla \xi|}}} \dx
\\
&\leq \int_\Gamma |\nabla \xi| \dx\norm{W_m - W}_{L^\infty(B_1(0))}.
\end{align*}
Therefore $E(\psi) = \lim_m E_m(\psi_m)$.
\end{proof}

\subsection{Cahn-Hoffman vector field patching}

We shall use the minimal section $\partial^0 E_W(\psi; \Omega)$ of the subdifferential of $E_W$ to define the crystalline
curvature for a given Lipschitz function $\psi$ on $\Omega$.
However, the minimal section is a solution of a variational problem and therefore
its value might depend strongly on the set $\Omega$, and nonlocally on the values of $\psi$.
Fortunately, the situation is not as dire as it might appear at first,
and in fact, the minimal section is nonlocal only on flat parts (facets) of $\psi$.
This restriction of nonlocality is expressed by the following lemma.
Intuitively, we can patch the Cahn-Hoffman vector fields as much as we please as long
as we do it across the level sets of $\psi$.

\begin{lemma}
\label{le:cahn-hoffman-patch}
Let $W: \R^d \to \R$ be a positively one-homogeneous convex function, $d \geq 1$.
Suppose that $\psi_1 \in \Lip(\Omega_1)$ and $\psi_2 \in \Lip(\Omega_2)$ are two
Lipschitz functions on two open subsets $\Omega_1, \Omega_2$ of $\R^d$.
Let $G = \set{x \in \Omega_1: a < \psi_1(x) < b}$ for some $a < b$
such that $\cl G \subset \Omega_1 \cap \Omega_2$
and $\psi_1 = \psi_2$ on $G$.
If $z_i \in \CH_W(\psi_i; \Omega_i)$ are two Cahn-Hoffman vector fields,
then
\begin{align}
\label{z-patch}
z(x) =
\begin{cases}
z_2(x) & x \in G,\\
z_1(x) & x \in \Omega_1 \setminus G,
\end{cases}
\end{align}
is also a Cahn-Hoffman vector field $z \in \CH_W(\psi_1; \Omega_1)$,
and
\begin{align}
\label{divz-patch}
\divo z(x) =
\begin{cases}
\divo z_2(x) & \text{a.e. } x \in G,\\
\divo z_1(x) & \text{a.e. } x \in \Omega_1 \setminus G.
\end{cases}
\end{align}
\end{lemma}

\begin{proof}
Since adding the same constant to both $\psi_1$ and $\psi_2$ does not change anything,
we can assume that $a = -\delta$ and $b = \delta$ for some $\delta > 0$.
For given $\e \in (0, \delta)$ we introduce the Lipschitz function
\begin{align*}
\zeta_\e(x) = 1 + \max \pth{ -1, \min \pth{0, \frac{\abs{\psi_1(x)} - \delta}\e}}.
\end{align*}
Note that $\zeta_\e = 0$ on $\set{\abs{\psi_1} \leq \delta - \e}$
and $\zeta_\e = 1$ on $\set{\abs{\psi_1} \geq \delta}$.
Furthermore,
\begin{align}
\label{zeta-deriv}
\nabla \zeta_\e(x) =
\begin{cases}
\sign \psi_1(x)\frac{\nabla \psi_1(x)}\e & \delta - \e < \psi_1(x) < \delta,\\
0 & \text{otherwise}
\end{cases}
\end{align}
for a.e. $x$.
Finally, $\zeta_\e \searrow \chi_{\Omega_1 \setminus G}$ monotonically pointwise as $\e \to 0$.

Now for $\rho > 0$ we define $z_i^\rho = z_i * \eta_\rho$, where $\eta_\rho$ is the standard mollifier with radius $\rho$, and we extend $z_i$ as $0$ to $\Omega_i^c$.
We have $z_i^\rho \to z_i$ in $L^\infty(\Omega_i)$-weak$^*$ and strongly in $L^p_{\rm loc}(\Omega_i)$
for any $1 \leq p < \infty$ as well as $\divo z_i^\rho \to \divo z_i$ strongly in $L^2_\loc(\Omega_i)$
as $\rho \to 0$, $i = 1,2$.
Define
\begin{align*}
z_\e^\rho = z_1^\rho \zeta_\e + z_2^\rho (1 - \zeta_\e).
\end{align*}
This function is clearly Lipschitz.

On $G$ we have $z_i(x) \in \partial W(\nabla \psi_1(x)) = \partial W(\nabla \psi_2(x))$
for a.e. $x$.
Therefore $(z_1(x) - z_2(x)) \cdot \nabla \psi_1(x) = 0$ for a.e. $x \in G$ by Lemma~\ref{le:one-homogeneous-subdiff}, which together with \eqref{zeta-deriv} implies
\begin{align}
\label{zetae-ortho}
\nabla \zeta_\e \cdot (z_1 - z_2) = 0 \qquad \text{a.e.}
\end{align}
Thus we have for any $\varphi \in C^\infty_c(\Omega_1)$
\begin{align*}
\int z_\e^\rho \cdot \nabla \varphi
&= -\int \varphi \divo z^\rho\\
&= -\int \varphi \bra{\zeta_\e \divo z_1^\rho + (1- \zeta_\e) \divo z_2^\rho + \nabla \zeta_\e \cdot (z_1^\rho - z_2^\rho)}.
\end{align*}
Now we send $\rho \to 0$ and obtain
\begin{align*}
\int z_\e \cdot \nabla \varphi &=
-\int \varphi \bra{\zeta_\e \divo z_1 + (1- \zeta_\e) \divo z_2 + \nabla \zeta_\e \cdot (z_1 - z_2)}\\
&= -\int \varphi \bra{\zeta_\e \divo z_1 + (1- \zeta_\e) \divo z_2},
\end{align*}
where we used \eqref{zetae-ortho}.
Finally we send $\e \to 0$ and use the Dominated convergence theorem to conclude that
\begin{align*}
\int z \cdot \nabla \varphi = - \int \varphi \bra{\chi_{\Omega_1 \setminus G} \divo z_1 + \chi_G \divo z_2}.
\end{align*}
Since this holds for any test function, we see that $\divo z \in L^2(\Omega_1)$ and it can be expressed as in \eqref{divz-patch}.
\end{proof}

\begin{remark}
\label{arb-convex-patch}
We can take an arbitrary convex combination of $z_1$ and $z_2$ on $G$ in \eqref{z-patch}.
Indeed, take $z$ as in $\eqref{z-patch}$.
Then $\lambda z_1 + (1-\lambda) z = (\lambda z_1 + (1 - \lambda) z_2) \chi_G + z_1 \chi_{G_1 \setminus G} \in \CH_W(\psi_1; G_1)$ by convexity.
\end{remark}

\begin{remark}
\label{re:patching-on-facet-boundaries}
In the proof of \cite[Proposition~2.10]{GGP13JMPA} in the case of $W$ with a smooth $1$-level set
we used the fact that Cahn-Hoffman vector fields can be patched across the boundary of a facet
arbitrarily, as a consequence of \cite[Proposition~2.8]{GGP13JMPA}.
This is stronger than Lemma~\ref{le:cahn-hoffman-patch} above where we can patch the Cahn-Hoffman
vector field only if the support functions coincide on a neighborhood of the facet. We believe that
this requirement can be removed as in \cite{GGP13JMPA}, but we do not pursue this matter further in
the current paper.
\end{remark}

Finally, let us briefly consider the characterization of the subdifferential of $E_W$
in the case when $W$ is not positively one-homogeneous.
Proposition~\ref{pr:subdiff-char-periodic} does not apply in such a case.
However, if $W$ is equal to a positively one-homogeneous function $W'$ in the neighborhood of
the origin, the subdifferentials of $E_W$ and $E_{W'}$ coincide at least for functions
with small Lipschitz constant.

\begin{lemma}
\label{le:subdiff-homog-relation}
Suppose that $W$ is a convex function and $W'$ is a positively one-homogeneous convex function
on $\R^d$, $d \geq 1$, and there exists $\e > 0$ such that $W(p) = W'(p)$ for $\abs p < \e$.
Suppose that $\Omega$ is a bounded open subset of $\Rd$ or the torus $\R^d / L \Z^d$ for some $L > 0$.
If $\psi \in \Lip(\Omega)$ and $\norm{\nabla \psi}_\infty < \e$, then
\begin{align*}
\partial E_W(\psi; \Omega) = \partial E_{W'}(\psi; \Omega).
\end{align*}
\end{lemma}

\begin{proof}
We shall denote the functionals as $E$ and $E'$ for short.
Fix $\psi \in \Lip(\Omega)$ with $\norm{\nabla \psi}_\infty < \e$.
By definition of the functionals and our assumption on the equality of $W$ and $W'$, we have
\begin{align}
\label{en-equality}
E(\psi + h) = E'(\psi + h) \qquad h \in \Lip(\Omega),\ \norm{\nabla h}_\infty < \delta = \e - \norm{\nabla \psi}_\infty.
\end{align}

The convexity of $W$, $W'$, and one-homogeneity of $W'$ imply for $p \in \R^d$ and $\lambda \in (0,1)$ such that $\lambda \norm p < \e$
\begin{align*}
\lambda W(p) \geq W(\lambda p) - (1 -\lambda) W(0) = W'(\lambda p) = \lambda W'(p).
\end{align*}
In particular, $W(p) \geq W'(p)$ on $\R^d$.
Therefore $E(\psi + h) - E(\psi) \geq E'(\psi + h) - E'(\psi)$ for all $h \in L^2(\Omega)$
since $E(\psi) = E'(\psi)$.
We conclude that $\partial E'(\psi) \subset \partial E(\psi)$.

To prove the opposite inclusion, take $v \in \partial E(\psi)$, if such an element exists.
We want to prove
\begin{align}
\label{subdiff-def}
E'(\psi + h) - E'(\psi) \geq (h, v) \qquad \text{for all $h \in L^2(\Omega)$.}
\end{align}
If $h \notin BV(\Omega)$, $E'(\psi + h) = \infty$ by definition.
Thus we can assume that $h \in BV(\Omega)$.
By a standard approximation result, there exists a sequence $\set{h_m} \subset C^\infty(\Omega) \cap BV(\Omega)$ such that
$h_m \to h$ in $L^2(\Omega)$ and $Dh_m \to Dh$ weakly$^*$ as measures, which yields
$E'(\psi + h_m) \to E'(\psi + h)$ due to \cite{Resetnjak}; see \cite{GiaquintaModicaSoucek}.
But we can choose $\lambda_m \in (0,1)$ such that $\lambda_m \norm{\nabla h_m}_\infty < \delta$.

Then \eqref{en-equality} implies
\begin{align*}
E'(\psi + \lambda_m h_m) - E'(\psi) =
E(\psi + \lambda_m h_m) - E(\psi) \geq \lambda_m (h_m, v).
\end{align*}
By convexity, we have
\begin{align*}
E'(\psi + h_m) - E'(\psi) \geq (h_m, v).
\end{align*}
Indeed,
\begin{align*}
\lambda_m E'(\psi + h_m) + (1-\lambda_m) E'(\psi)
&\geq E'(\lambda_m(\psi + h_m) + (1-\lambda_m)\psi)\\
&=E'(\psi +\lambda_m h_m) \geq E'(\psi) +\lambda_m(h_m, v).
\end{align*}
Sending $m \to \infty$ yields \eqref{subdiff-def}.
\end{proof}

\section{Energy stratification}
\label{sec:energy-stratification}

In this section we shall assume that $W$ is a convex \textbf{polyhedral} function on $\Rn$.
Since $W$ is polyhedral, it can be locally viewed as a positively one-homogeneous convex function; we will give a detailed explanation in this section.
The features of $W$ correspond to the dual features of the crystal such as facets, edges and vertices, depending on the dimension.
For each gradient, we will decompose the space into orthogonal subspaces of interesting directions, corresponding to the given feature of the crystal, and the directions in which $W$ is linear and therefore its behavior simple.

\subsection{Slicing of $W$}
\label{sec:slicing of W}

To perform the decomposition, we need a few standard concepts from convex analysis (see for example \cite{Rockafellar}).
For a given convex set $C$ let $\aff C$ denote the affine hull of $C$, that is, the smallest affine space containing $C$.
The dimension of the convex set is defined as the dimension of its affine hull, $\dim C := \dim \aff C$.
Let $\ri C$ be the relative interior of $C$ with respect to $\aff C$.
A convex set is said to be relatively open if $C = \ri C$.
We know that $\ri C \neq \emptyset$ if $C \neq \emptyset$ (\cite[Theorem~6.2]{Rockafellar}).
We say that $\aff C$ is parallel to a subspace $V \subset \Rn$ if $\aff C = p + V$ for some $p \in \Rn$.

We can decompose $\Rn$ based on the features of the crystal, which correspond to the value of $\partial W$.

\begin{proposition}[Feature decomposition]
\label{pr:feature-decomposition}
For given $W$ polyhedral with $W < \infty$ on $\Rn$
there exist a finite number of mutually disjoint maximal sets $\Xi_i$, $i \in \mathcal N$, such that $\Rn = \bigcup_{i \in \mathcal N} \Xi_i$
and $\partial W$ is constant on each $\Xi_i$.
Furthermore, each $\Xi_i$ is a relatively open convex set and $\aff \Xi_i \perp \aff \partial W(p)$ for $p \in \Xi_i$ in the sense that whenever $p,
q \in \Xi_i$ and $\xi, \zeta \in \partial W(p)$ then $p - q \perp \xi - \zeta$.
\end{proposition}

\begin{proof}
We use the projections of relative interiors of the non-empty faces of the epigraph $\epi W :=
\set{(p, \lambda): \lambda \geq W(p), p \in \Rn}$, other than $\epi W$
itself, onto $\Rn$.
For the definition of a face of a convex set see \cite[Section~18]{Rockafellar}.
By \cite[Corollary~18.1.3]{Rockafellar}, all faces of $\epi W$ other than $\epi W$ itself must lie
in the relative boundary of $\epi W$. The relative boundary of $\epi W$, the set $\epi W \setminus
\ri \epi W$, is just the regular boundary and therefore it is the graph of $W$,
$\graph W := \set{(p, W(p)): p \in \Rn} \subset \R^{n+1}$.
By \cite[Theorem~18.2]{Rockafellar}, the relative interiors $\hat \Xi_i$ of the faces of $\epi W$ other than
$\epi W$ itself form a partition of $\graph W$. By projecting these relative interiors $\hat \Xi_i$ onto $\Rn$
we obtain sets $\Xi_i$, which form a partition of $\Rn$ and are again relatively open by
\cite[Theorem~6.6]{Rockafellar}.

Let us now prove that $\partial W$ is constant on $\Xi_i$. Fix two points $p, q \in \Xi_i$. Since
$\hat \Xi_i$ are relatively open, there exists $\mu > 1$ such that $W(\mu p + (1 - \mu) q) = \mu
W(p) + (1- \mu) W(q)$. Let $\xi \in \partial W(p)$. By definition of the subdifferential, we have
$W(\mu p + (1 - \mu) q) \geq W(p) + (\mu - 1) \xi \cdot (p - q)$ and $W(q) \geq W(p) + \xi \cdot (q
- p)$. Using the equality in the first inequality and dividing by $\mu - 1$ we obtain $W(q) \leq
W(p) + \xi \cdot (q - p)$. Therefore $W(q) - W(p) = \xi \cdot (q - p)$ and we deduce that $\xi \in
\partial W(q)$. Finally, if $\zeta \in \partial W(p)$ as well, we have $(\zeta - \xi)
\cdot (q - p) = 0$.
Maximality, that is, that $\partial W(p) \neq \partial W(q)$ for $p \in \Xi_i$, $q \in \Xi_j$, $i
\neq j$, follows from the definition of convex faces.
\end{proof}

\begin{lemma}
\label{le:aff Xi origin}
Suppose that $\Xi_i$ are as in Proposition~\ref{pr:feature-decomposition} and suppose that $W$ is
also positively one-homogeneous. Then $0 \in \aff \Xi_i$ for every $i$.
\end{lemma}

\begin{proof}
This follows immediately from one-homogeneity since $\partial W(p) = \partial W(tp)$ for any $p \in
\R^n$, $t > 0$.
\end{proof}

Since $W$ is finite everywhere, $\partial W(p)$ is a nonempty closed convex set for any $p \in \Rn$.
For given $p_0 \in \Rn$ we introduce the one-sided directional derivative of $W$ at $p_0$ with respect to a vector $p \in \Rn$ as (\cite[Section~23]{Rockafellar})
\begin{align*}
W_{p_0}'(p) := \lim_{\lambda \to 0+} \frac{W(p_0 + \lambda p) - W(p_0)}{\lambda}.
\end{align*}
Then $W'(p_0; \cdot)$ is a positively one-homogeneous convex function, and (\cite[Theorem~23.4]{Rockafellar})
\begin{align}
\label{W'-convex-conjugate}
W_{p_0}'(p) \equiv \delta^*(p \mid \partial W(p_0)) := \sup \set{p \cdot \xi: \xi \in \partial W(p_0)}.
\end{align}
In particular, $W_{p_0}'$ is the convex conjugate of the indicator function of $\partial W(p_0)$.
Therefore by \cite[Theorem~13.4]{Rockafellar} the lineality space of $W_{p_0}'$ (the subspace of directions in which $W_{p_0}'$ is affine) is the orthogonal complement of the subspace parallel to $\aff \partial W(p_0)$.
This provides the orthogonal decomposition of $\Rn$ for a given gradient.

\begin{proposition}[Direction decomposition]
\label{pr:direction-decomposition}
Let $W$ be a polyhedral convex function on $\Rn$ finite everywhere and let $p_0 \in \Rn$.
Let $V$ be the subspace of $\Rn$ parallel to $\aff \partial W(p_0)$ and set $U = V^\perp$.
Then $W_{p_0}'$ is linear on $U$ and
\begin{align}
\label{W'-linear}
W_{p_0}'(p) = W_{p_0}'(P_V p) + \xi \cdot P_U p \qquad \text{for any $p \in \Rn$, $\xi \in \aff \partial W(p_0)$},
\end{align}
where $P_U$ and $P_V$ are the orthogonal projections onto $U$ and $V$, respectively.

Moreover, there exists $\delta > 0$ such that
\begin{align*}
W(p) = W_{p_0}'(p - p_0) + W(p_0) \qquad \text{for all $\abs{p - p_0} < \delta$.}
\end{align*}
\end{proposition}

\begin{proof}
\eqref{W'-linear} follows from \eqref{W'-convex-conjugate} and from the orthogonality of $U$ and $V$.

The existence of $\delta > 0$ can be proved by contradiction: suppose that there exists a sequence $\set{p_k}$, $p_k \to p_0$ such that $W(p_k) - W(p_0) > W_{p_0}'(p_k - p_0)$
(it is clear that $W_{p_0}'(p - p_0) \leq W(p) - W(p_0)$ by convexity).
Since $W$ is polyhedral, it is given as the maximum of a finite number of affine functions, and therefore by taking a subsequence we can assume that $W(p_k) = \xi \cdot p_k + c$ for some fixed $\xi \in \Rn$, $c \in \R$.
By continuity we have $W(p_k) - W(p_0) = \xi \cdot (p_k - p_0)$.
Therefore $\xi \in \partial W(p_0)$.
But this yields a contradiction since then \eqref{W'-convex-conjugate} implies $W_{p_0}'(p_k - p_0) \geq \xi \cdot (p_k - p_0)$.
\end{proof}

The previous proposition tells us that the behavior of $W$ is interesting only in the directions parallel to $\aff \partial W$.
That motivates the following notation.
For given $W: \Rn \to \R$ convex polyhedral and $p \in \Rn$
let $V$ be the subspace of $\Rn$ parallel to $\aff \partial W(p)$, $U = V^\perp$, $k  = \dim V$, and
we fix an arbitrary rotation
\begin{align}
\label{rotation}
\TT: \Rn \to \Rn
\end{align}
that maps $\R^k \times \set 0$ onto $V$ and $\set 0 \times \R^{n-k}$ onto $U$.
For given $x \in \Rn$, we define the unique $x' \in \R^k$ and $x'' \in \R^{n-k}$
such that
\begin{align}
\label{rotation'}
\TT(x', x'') = x.
\end{align}
We set $\TT_V: \R^k \to V$ and $\TT_U: \R^{n-k} \to U$ by
\begin{align}
\label{rotationUV}
\TT_V x' = \TT(x', 0), \qquad \TT_U x'' = \TT(0, x'').
\end{align}
In the above we also allow for $k = 0$ and $k = n$, in which case
terms containing $x'$ respectively $x''$ simply do not appear in the formulas,
and $\TT_V$ respectively $\TT_U$ are trivial maps.
Note that
\begin{align*}
\left( \TT_V z \right)' = z, \quad \left( \TT_U w \right)'' = w, \qquad z \in \R^k, w \in
\R^{n-k},
\end{align*}
and
\begin{align*}
\TT_V x' = P_V x, \quad \TT_U x'' = P_U x, \qquad x \in \R^n,
\end{align*}
where $P_V$ and $P_U$ are respectively the orthogonal projections on $V$ and $U$.
Since $\TT$ is a linear isometry, it preserves the inner product
\begin{align*}
z_1 \cdot z_2 = \TT_V z_1 \cdot \TT_V z_2, \qquad z_1, z_2 \in \R^k,
\end{align*}
and similarly for $\TT_U$.

\begin{remark}
We are free to choose any such $\TT$, as long as we keep this choice consistent throughout
the paper for given $W$ and $p$.
We can in fact choose the same $\TT$ for all $p \in \Xi_i$
from Proposition~\ref{pr:feature-decomposition}.
\end{remark}

We will introduce the sliced energy density $\SW$ that locally captures behavior
of $W$ in the directions $V$.

\begin{definition}
\label{def:sliced-W}
We define the \emph{sliced density} $\SW_p: \R^k \to \R$ as
\begin{align}
\label{W-red}
\begin{aligned}
\SW_p &:= W_p' \circ \TT_V.
\end{aligned}
\end{align}
\end{definition}

\begin{lemma}[Decomposition]
\label{le:decomposition-subdiff-W}
For any fixed $p_0 \in \Rn$ we have
\begin{align*}
\partial W_{p_0}'(p) =
\set{\TT (\zeta', \xi''): \zeta' \in \partial \SW_{p_0}(p')}
\qquad \text{for all $p \in \Rn$, $\xi \in \partial W_{p_0}'(0)$}.
\end{align*}
\end{lemma}

The following lemma states that the behavior of $W$
in the neighborhood of some $p$ is completely captured
by the sliced density $\SW_p$.

\begin{lemma}
\label{le:W-decomposition}
For every $p_0 \in \R^n$
there exists $\e > 0$ such that
\begin{align}
\label{W-decomposition}
W(p) = \SW_{p_0}(p'-p_0')
+ P_U \xi \cdot (p - p_0) + W(p_0)
\end{align}
for any $p \in \Rn$,
$\abs{p - p_0} <\e$, $\xi \in \partial W(p)$.
Since $\TT$ is an isometry, we have $P_U \xi \cdot (p - p_0) = \xi'' \cdot (p'' - p_0'')$.
\end{lemma}

\begin{proof}
The claim follows from Definition~\ref{def:sliced-W} and Proposition~\ref{pr:direction-decomposition}.
\end{proof}

\begin{lemma}
\label{le:linear growth SW}
Suppose that $p_0 \in \Rn$ and $\xi_0 \in \ri \partial W(p_0)$. Then there exists $\delta > 0$ such
that
\begin{align*}
\SW_{p_0}(z) - \xi_0' \cdot z \geq \delta |z|, \qquad z \in \R^k,
\end{align*}
where $k = \dim \aff \partial W(p_0)$.
\end{lemma}

\begin{proof}
Let again $V$ be the subspace parallel to $\aff \partial W(p_0)$. Then $\aff \partial W(p_0)
= \xi_0 + V$. Since $\xi_0 \in \ri \partial W(p_0)$, there exists $\delta > 0$ with $\xi \in
\partial W(p_0)$ for all $|\xi - \xi_0| \leq \delta$, $\xi \in \xi_0 + V$.
Take $z \in \R^k$ and set $\zeta = \xi_0 + \delta \frac{\TT_V z}{|z|} \in \partial W(p_0)$. From
the definition of $\SW_{p_0}$ we have from \eqref{def:sliced-W} and \eqref{W'-convex-conjugate}
\begin{align*}
\SW_{p_0}(z) = W_{p_0}'(\TT_V z) =\sup \set{\TT_V z \cdot \xi : \xi \in \partial W(p_0)} \geq \TT_V
z \cdot \zeta = \xi_0'\cdot z + \delta |z|.
\end{align*}
This yields the lower bound.
\end{proof}

\subsection{Sliced energy}

Suppose now that $p \in \Rn$ such that $k = \dim \partial W(p) > 0$
and recall the definition of $\TT$ in \eqref{rotation}.
We shall consider the rotated flat torus $\Gamma = \R^n / L \TT \Z^n$ for some $L > 0$.
We can write $\Gamma = \TT(\Gamma' \times \Gamma'')$,
where $\Gamma' = \R^k / L \Z^k$ and $\Gamma'' = \R^{n - k} / L\Z^{n-k}$,
and $x \in \Gamma$ is given as $x = \TT(x', x'')$
for $x' \in \Gamma'$, $x'' \in \Gamma''$.

We define the functionals
\begin{align*}
E_p(\psi) &:= E_{W(\cdot + p) - W(p)}(\psi; \Gamma), && \psi \in L^2(\Gamma),\\
E'_p(\psi) &:= E_{W_p'}(\psi; \Gamma), && \psi \in L^2(\Gamma),\\
\SE_p(\psi) &:= E_{\SW_p}(\psi; \Gamma'), && \psi \in L^2(\Gamma').
\end{align*}
All three functionals are proper closed convex functions on $L^2(\Gamma)$ resp. $L^2(\Gamma')$.

Since $W_p'$ and $\SW_p$ are positively one-homogeneous,
the characterization of the subdifferential in Proposition~\ref{pr:subdiff-char-periodic}
applies.

The function $q \mapsto W(q + p) - W(p)$ is not one-homogeneous in general, however,
and therefore the same
characterization does not apply for the subdifferential of $E_p$.
Nevertheless, it coincides with the subdifferential of $E_p'$ at $\psi \in \Lip(\Gamma)$
when $\norm{\nabla \psi}_\infty$ is small
by Lemma~\ref{le:subdiff-homog-relation}.
This observation allows us to use the simpler, positively one-homogeneous energy $E_p'$
when defining the crystalline curvature of a facet.

What follows is the main justification of the energy stratification.
We show that since $W_p'$ is linear on the subspace $U$,
we need to only consider the directions in $V = U^\perp$
when computing the crystalline curvature of a \textbf{stratified} function.

\begin{lemma}
\label{le:subdiff-slicing}
Let $p$ be as above.
Suppose that $\bar\psi \in \Lip(\Gamma')$ and $f \in C^1(\Gamma'')$
are given functions and let $\psi(x) = \bar\psi(x') + f(x'')$.
Let $\psi_a$ and $\bar \psi_a$ be the unique solutions of the resolvent problems
\begin{align*}
\psi_a + a \partial E_p'(\psi_a) &\ni \psi,\\
\bar\psi_a + a \partial \SE_p(\bar\psi_a) &\ni \bar\psi,
\end{align*}
for given $a > 0$.
Then
\begin{align*}
\psi_a(x)  = \bar \psi_a(x') + f(x''), \qquad \text{$x = \TT(x', x'') \in \Gamma$}.
\end{align*}
or, equivalently,
\begin{align*}
(I +a\partial E_p')^{-1}(\psi)(x) = (I +a \partial \SE_p)^{-1}(\bar\psi)(x') + f(x'').
\end{align*}
If moreover $\bar\psi \in \domain(\partial \SE_p)$,
then $\psi \in \domain(\partial E_p')$, $\partial^0 E_p'(\psi)$ is independent of $x''$ and
\begin{align}
\label{minsectionslicing}
\partial^0 E_p'(\psi)(x) = \partial^0 \SE_p(\bar \psi)(x') \qquad \text{a.e. $x = \TT(x', x'') \in \Gamma$}.
\end{align}
\end{lemma}

\begin{proof}
Suppose that $\psi(x) = \bar \psi(x') + f(x'')$ for some $\bar \psi \in \Lip(\Gamma')$
and $f \in C^1(\Gamma'')$.
By the characterization of the subdifferentials in Proposition~\ref{pr:subdiff-char-periodic},
we have
\begin{align*}
\partial E'_p(\psi) = - \divo \CH_{W_p'}(\psi; \Gamma), \qquad
\partial \SE_p(\bar \psi) = - \divo \CH_{\SW_p}(\bar\psi; \Gamma').
\end{align*}
The decomposition lemma~\ref{le:decomposition-subdiff-W} implies
\begin{align}
\label{subdiffslicing}
\partial W'_p(\nabla \psi(x)) = \set{\TT(\xi', \xi''):
\xi' \in \partial \SW_p(\nabla \bar\psi(x'))}
\end{align}
for some fixed $\xi'' \in \R^{n-k}$
since
\begin{align*}
\nabla \psi(x) = \TT (\nabla \bar \psi(x'), \nabla f(x'')).
\end{align*}

By Proposition~\ref{pr:resolvent-problems}, both $\psi_a$ and $\bar\psi_a$ are Lipschitz.
As $\bar\psi_a$ is the unique solution of the resolvent problem,
the characterization of the subdifferential of $\SE_p$ above yields that
there exists $\bar z_a \in \CH_{\SW_p}(\bar\psi_a; \Gamma')$
such that $\bar\psi_a - \bar\psi = a \divo \bar z_a$.
Set $z_a(x) = \TT(\bar z_a(x'), \xi'')$ for some fixed $\xi''$ as above
and $\zeta_a(x) = \bar\psi_a(x') + f(x'')$.
Note that $z_a \in \CH_{W_p'}(\zeta_a; \Gamma)$
by \eqref{subdiffslicing}.
Moreover $\divo_x z_a(x) = \divo_{x'} \bar z_a(x')$.
Therefore
\begin{align}
\label{hgyfhbs}
\begin{aligned}
\zeta_a(x) - \psi(x) &= \bar\psi_a(x') +f(x'') - \bar\psi(x') + f(x'') =
\bar \psi_a(x') - \bar\psi(x')\\
&= a \divo_{x'} \bar z_a(x') =
a \divo_x z_a(x).
\end{aligned}
\end{align}
The characterization of the subdifferential of $E_p'$ above implies that
$\zeta_a - \psi \in  - \partial E_p'(\zeta_a; \Gamma)$
and therefore $\zeta_a$ is a solution of the resolvent problem.
However, the solution is unique and therefore $\psi_a = \zeta_a$ almost everywhere.

Now we suppose that $\bar \psi \in \domain(\partial \SE_p)$, that is,
that $\partial \SE_p(\bar\psi)$ is nonempty.
And so there exists $\bar z\in \CH_{\SW_p}(\bar\psi; \Gamma')$.
But then $z(x) = \TT(\bar z(\bar x), \xi'')  \in \CH_{W_p'}(\psi; \Gamma)$ as we just observed.
In particular, $\psi \in \domain(\partial E_p')$.

Let us set $h_a = (\psi_a - \psi)/a$ and $\bar h_a = (\bar \psi_a - \bar\psi)/a$
as in Proposition~\ref{pr:resolvent-problems}.
Observe that due to \eqref{hgyfhbs}
\begin{align*}
h_a(x) = \bar h_a(x').
\end{align*}
Since $-h_a \to \partial^0 E_p'(\psi; \Gamma)$ in $L^2(\Gamma)$
and $-\bar h_a \to \partial^0 \SE_p(\bar \psi; \Gamma')$ in $L^2(\Gamma')$ as $a \to 0$,
we conclude \eqref{minsectionslicing}.
\end{proof}

\section{Crystalline curvature}
\label{sec:crystalline curvature}

We introduce an operator $\Lambda_p$ that
assigns the crystalline curvature to a facet with slope $p$ given by a faceted function,
as long as the faceted function is admissible in a certain sense.

\subsection{Facets}
To describe facets,
let us recall the notation for pairs that was introduced
in \cite{GGP13AMSA}.
Since we need to construct facets of various dimensions, depending
on the dimension of $\partial W(p)$,
$\mathcal P^k$ will denote the set of pairs on $\R^k$:

\begin{definition}[cf. \cite{GGP13AMSA}]
For any $k \in \N$ we will denote by $\mathcal P^k$ the set of all
ordered pairs $(A_-, A_+)$ of disjoint sets $A_\pm \subset \R^k$,
$A_- \cap A_+ = \emptyset$.

We will introduce a partial ordering $(\mathcal P^k, \preceq)$
by
\begin{align*}
(A_-, A_+) \preceq (B_-, B_+)
\qquad \Leftrightarrow \qquad
A_+ \subset B_+ \text{ and } B_- \subset A_-
\end{align*}
for $(A_-, A_+), (B_-, B_+) \in \mathcal P^k$,
as well as the \emph{reversal}
\begin{align*}
-(A_-, A_+) := (A_+, A_-).
\end{align*}
Clearly, if $(A_-, A_+) \preceq (B_-, B_+)$
then $-(B_-, B_+) \preceq -(A_-, A_+)$.

A pair $(A_-, A_+) \in \mathcal P^k$ is said to be \emph{open}
if both $A_-$ and $A_+$ are open.

A \emph{smooth} pair is then an open pair $(A_-, A_+) \in \mathcal P^k$
for which we also have
\begin{enumerate}[(i)]
\item
$\dist(A_-, A_+) > 0$, where we use the convention
$\dist(\emptyset, E) = + \infty$ for any $E$, and
\item
$\partial A_- \in C^\infty$ and $\partial A_+ \in C^\infty$.
\end{enumerate}

We will refer to the set
\begin{align*}
\R^k \setminus \pth{A_- \cup A_+} = A_-^c \cap A_+^c
\end{align*}
as the \emph{facet} of the pair $(A_-, A_+) \in \mathcal P^k$.
\end{definition}

\begin{remark}
We will drop $k$ if the dimension is understood from the context
or is irrelevant.
\end{remark}

We will add the notion of a bounded pair.
\begin{definition}
We say that a pair $(A_-, A_+) \in \mathcal P^k$ is \emph{bounded}
if either $A_-^c$ or $A_+^c$ is bounded.
\end{definition}

\begin{remark}
Note that if $(A_-, A_+)$ is a bounded pair, then the facet $A_-^c \cap A_+^c$ is bounded.
If $(A_-, A_+)$ is an open pair, the reverse implication also applies.
\end{remark}

Let us also recall the useful notion of a support function.

\begin{definition}[cf. \cite{GGP13AMSA}]
A Lipschitz function $\psi \in \Lip(\R^k)$ is called
a \emph{support function} of an open pair $(A_-, A_+) \in \mathcal P^k$
if
\begin{align*}
\psi(x)
\begin{cases}
> 0 & x \in A_+,\\
= 0 & x \in A_-^c \cap A_+^c,\\
< 0 & x \in A_-.
\end{cases}
\end{align*}

On the other hand, for any function $\psi$ on $\R^k$
we define the pair
\begin{align*}
\Pair(\psi) := \pth{\set{x \in \R^k : \psi(x) <0},
\set{x \in \R^k: \psi(x) >0}}.
\end{align*}
\end{definition}

\begin{example}
\label{ex:trivial-support-function}
For any open pair $(A_-, A_+) \in \mathcal P^k$ the function
\begin{align*}
\psi(x) := \dist(x, A_+^c) - \dist(x, A_-^c)
\end{align*}
is a support function of the pair $(A_-, A_+)$.
\end{example}

Finally, let us recall the notion of a generalized neighborhood of
a subset of $\R^k$.

\begin{definition}[cf. \cite{GGP13AMSA}]
For any set $E \subset \R^k$ and $\rho \in \R$ the
\emph{generalized neighborhood} is defined as
\begin{align*}
\nbd\rho(E) :=
\begin{cases}
E + \cl B_\rho(0) & \rho > 0,\\
E & \rho = 0,\\
\set{x \in E : \cl B_{\abs\rho}(x) \subset E} & \rho < 0.
\end{cases}
\end{align*}

For a pair $(A_-, A_+) \in \mathcal P^k$ we introduce the
generalized neighborhood
\begin{align*}
\nbd\rho(A_-,A_+) := \pth{\nbd{-\rho}(A_-),\nbd{\rho}(A_+)}.
\end{align*}
\end{definition}

A part of the following proposition was stated in \cite{GGP13AMSA} for the $n$-dimensional torus, but it can be easily restated for $\Rn$.
The proof is straighforward.

\begin{proposition}
\label{pr:nbd-properties}
\begin{enumerate}
\item $\mathcal U^{-\rho}(A) \subset A \subset \mathcal U^\rho(A)$
for $\rho > 0$.
\item (complement)
\begin{align}
\label{compl-nbd}
\pth{\nbd\rho(A)}^c = \nbd{-\rho}(A^c)
\qquad \text{for any set $A \subset \Rn$ and $\rho \in \R$}
\end{align}
\item (monotonicity)
\begin{align*}
\nbd\rho(A_1) \subset \nbd\rho(A_2)\qquad
\text{for $A_1 \subset A_2 \subset \Rn$ and $\rho \in \R$.}
\end{align*}
\item
$\nbd\rho(A_1 \cap A_2) \subset \nbd\rho(A_1) \cap \nbd\rho(A_2)$
 for all $\rho \in \R$, with equality for $\rho \leq 0$.
\item
$\nbd{r}(\nbd\rho(A)) \subset \nbd{r+\rho}(A)$
for $r \geq 0$ and $\rho \in \R$; equality holds if $\rho \geq 0$.
\item
For any $\rho \in \R$, we have
$\nbd\rho(A_1) \subset A_2$ if and only if $A_1 \subset \nbd{-\rho} (A_2)$.
\item (interior and closure)
\begin{align*}
\bigcup_{\rho > 0} \nbd{-\rho}(A) = \interior A \subset A \subset \cl A = \bigcap_{\rho>0} \nbd\rho(A) \qquad \text{for any set $A \subset \Rn$.}
\end{align*}
\item (distance)
\begin{align*}
\dist(A_1, A_2) = \sup \set{\rho \geq 0: \nbd\rho(A_1) \subset A_2^c} \qquad \text{for all $A_1, A_2 \subset \Rn$.}
\end{align*}
\end{enumerate}
\end{proposition}

\subsection{Definition of crystalline curvature}
We assume for the rest of the paper that $W$ is a convex polyhedral function on $\Rn$.
Let $p \in \Rn$ such that $k = \dim \partial W(p) > 0$.
Let $\psi$ be a support function of a bounded open pair $(A_-, A_+) \in \mathcal P^k$.
We say that $\psi$ is an \emph{$p$-admissible support function} if there exists
an open set $G \supset A_-^c \cap A_+^c$ such that
the set of Cahn-Hoffman vector fields
\begin{align*}
\SlCH_p(\psi; G) := \CH_{\SW_p}(\psi; G)
\end{align*}
is nonempty.
We denote this for short as $\psi \in \domain(\Lambda_p)$.
If for a given bounded open pair $(A_-, A_+)$ there exists at least one
$p$-admissible support function, we say that $(A_-, A_+)$ is a \emph{$p$-admissible pair}.
If $p$ is understood from the context, we refer to them as an admissible support function and
an admissible pair.

Let $\psi \in \domain(\Lambda_p)$ be an admissible support function of an admissible pair
$(A_-, A_+)$.
We define the function $\Lambda_p[\psi] \in L^2(A_-^c \cap A_+^c)$ on the facet as
\begin{align}
\label{crystalline-curvature}
\Lambda_p[\psi](x) = \divo z_{\rm min}(x), \qquad x \in A_-^c \cap A_+^c,
\end{align}
where $z_{\rm min}$ is an element of $\SlCH_p(\psi; G)$ that minimizes
$\norm{\divo z}_{L^2(G)}$.
We call $\Lambda_p$ the \emph{crystalline curvature}.

\begin{remark}
As we shall see later in Corollary~\ref{co:lambda support func indep} at the end of this section,
the crystalline curvature satisfies a comparison principle and therefore its value on a facet of a
given admissible pair is independent of the choice of an admissible support function of this pair.
\end{remark}

We first prove that the crystalline curvature is well-defined $\Lambda_p$.

\begin{proposition}
\label{pr:lambda-well-defined}
The quantity $\Lambda_p[\psi]$ is well-defined in the sense that the value is unique a.e.
and it does not depend on $G$ nor on the value of $\psi$ away from the facet.
More precisely, if $\psi_1$ and $\psi_2$ are two support functions of a bounded open pair $(A_-, A_+) \in \mathcal P^k$
with $\psi_i \in \domain(\Lambda_p)$ for some $p \in \Rn$ with $k = \dim \partial W(p) > 0$
such that $\psi_1 = \psi_2$ on a neighborhood of the facet $A_-^c \cap A_+^c$,
then $\Lambda_p[\psi_1] = \Lambda_p[\psi_2]$ a.e. on $A_-^c \cap A_+^c$.
\end{proposition}

\begin{proof}
Let $\psi_i \in \domain(\Lambda_p)$, $i = 1,2$, be two support functions that satisfy the hypothesis.
Then there are open sets $G_i \supset A_-^c \cap A_+^c$ and
associated Cahn-Hoffman vector fields $z_i \in \SlCH_p(\psi_i; G_i)$
that minimize $\norm{\divo z_i}_{L^2(G_i)}$ over $\SlCH_p(\psi_i; G_i)$.
Since the facet $A_-^c \cap A_+^c$ is assumed to be bounded, we can find a bounded open set
$H \supset A_-^c \cap A_+^c$ with $\psi_1 = \psi_2$ on $H$ and $H \subset G_1 \cap G_2$.
Let us take $0 < \delta < \min_{\partial H} \abs{\psi_1}$
and set $G = \set{x \in H: \abs{\psi_1} < \delta} \subset\subset H$.

Set $z = z_1 \chi_{G_1 \setminus G} + z_2 \chi_G$.
By Lemma~\ref{le:cahn-hoffman-patch} we have that $z \in \SlCH_p(\psi_1; G_1)$
and therefore
$\norm{\divo z}_{L^2(G_1)} \geq \norm{\divo z_1}_{L^2(G_1)}$,
which with \eqref{divz-patch} implies
\begin{align*}
\norm{\divo z_2}_{L^2(G)} \geq \norm{\divo z_1}_{L^2(G)}.
\end{align*}
Reversing the roles of $\psi_1$ and $\psi_2$, and $G_1$ and $G_2$, we get the opposite inequality.

Therefore the strict convexity of the $L^2$-norm implies
that $\divo z_1 = \divo z_2$ a.e. on $L$.
Indeed, if they are not equal, we can decrease the norm by taking the vector field
$z = \frac 12 \pth{z_1  + z_2}$ on $G$
which is still admissible due to Remark~\ref{arb-convex-patch}.
\end{proof}

The following crucial result will allow us to express the crystalline curvature
as the minimal section of the subdifferential of the sliced energy on a periodic domain.

\begin{proposition}
\label{pr:curvature-as-min-section}
Let $p \in \R^k$ be such that $k = \dim \partial W(p) > 0$.
Suppose that $\psi \in \domain(\Lambda_p)$, that is, $\psi$ is an admissible
support function of a bounded open pair $(A_-, A_+)$.
Let $L > 0$ be such that $A_-^c \cap A_+^c \subset B_{L/4}(0)$.
Denote $\Gamma = \R^k / L \Z^k$.

There exists an $L$-periodic Lipschitz function $\psi_2 \in \Lip(\Gamma)$
such that $\psi_2$ is a support function of the open pair $(A_- + L \Z^k, A_+ + L \Z^k)$
and $\SlCH_p(\psi_2; \Gamma)$ is nonempty,
and for some open set $H$, $A_-^c \cap A_+^c \subset H \subset B_{L/4}(0)$
we have $\psi_1 = \psi_2$ on $H$.
Moreover,
\begin{align}
\label{lambda-per-expr}
\Lambda[\psi_1](x) = -\partial^0 \SE_p(\psi_2; \Gamma)(x) \qquad \text{a.e. $x \in A_-^c \cap A_+^c$.}
\end{align}
\end{proposition}

\begin{proof}
Let us first show \eqref{lambda-per-expr} if we have function $\psi_2$
with the properties stated in the proposition.
We use the characterization of the differential in Proposition~\ref{pr:subdiff-char-periodic}.
Let therefore $z_2 \in \SlCH_p(\psi_2; \Gamma)$ be a Cahn-Hoffman vector field
that minimizes $\norm{\divo z_2}_2$ in this set.
Note that we have $\partial^0 \SE_p(\psi_2; \Gamma)(x) = - \divo z_2$ by the
characterization of the subdifferential in Proposition~\ref{pr:subdiff-char-periodic}.

Now we can proceed as in the proof of Proposition~\ref{pr:lambda-well-defined}.
Let $z_1$ minimize $\norm{\divo z_1}_2$ in $\SlCH_p(\psi_1; G_1)$
for some open set $G_1 \supset A_-^c \cap A_+^c$.
We can assume that $G_1 \subset H$.
We proceed as follows: given that $G \subset\subset H$ with $G$ as defined in that proof,
Lemma~\ref{le:cahn-hoffman-patch} can be applied to $\psi_1$ on $G_1$ and $\psi_2$ on $G_2 = H$,
and since we are only modifying the vector fields away from the boundary of $H$,
replacing $z_2$ by $z_1$ on the set $G$ yields again a vector field in $\SlCH_p(\psi_2; \Gamma)$.
We again deduce that $\divo z_1 = \divo z_2$ on $A_-^c \cap A_+^c$,

We shall now construct $\psi_2$.
Since $\psi_1$ is admissible there are an open set $G \supset A_- \cap A_+$,
$G \subset B_{L/4}(0)$,
and a vector field $z \in \SlCH_p(\psi_1; G)$.
Let us choose a positive $\delta$ such that $\delta < \min_{\partial G} \abs{\psi_1}$.
This is possible since $\psi_1$ is continuous and
$\partial G \subset A_- \cup A_+ = \set{\psi_1 \neq 0}$.
We set
\begin{align*}
\xi(x) =
\begin{cases}
-\delta & x \in A_- \setminus G,\\
\max(-\delta, \min(\delta, \psi_1(x))) & x \in G\\
\delta & x \in A_+ \setminus G.
\end{cases}
\end{align*}
Note that $\xi$ is Lipschitz on $\R^k$,
$\nabla \xi(x) = \nabla \psi_1$ whenever $\abs{\xi(x)} < \delta$ and $\nabla \xi(x) = 0$
if $\abs{\xi(x)} = \delta$, almost everywhere.
Moreover, we see that
\begin{align}
\label{xi-H-eq}
\xi = \psi_1 \qquad \text{on } H := \set{x \in G: \abs{\psi_1} < \delta}.
\end{align}
Since the complement of $B_{L/4}(0)$ is connected and $A_-$, $A_+$ are open disjoint sets,
we must have either $A_- \subset B_{L/4}(0)$ or $A_+ \subset B_{L/4}(0)$.
In any case, $\xi$ is constant outside of $B_{L/4}(0)$.

Let $\phi \in C^\infty(\R^k)$ be such that $0 \leq \phi \leq 1$, $\supp \phi \subset G$ and $\phi = 1$ on
$\set{x \in G: \abs\psi \leq \delta}$.
Define
\begin{align*}
w(x) =
\begin{cases}
z(x) \phi(x) & x \in G,\\
0 & \text{otherwise}.
\end{cases}
\end{align*}
Clearly $w \in L^\infty(\R^k; \R^k)$, $\divo w \in L^2(\R^k)$, $\supp w \subset G$.
Moreover, since $\partial \SW_p(q) \subset \partial \SW_p(0) \ni 0$ for any $q \in \R^k$
by Lemma~\ref{le:one-homogeneous-subdiff}
and $\partial \SW_p(0)$ is convex, we have
$w(x) \in \partial \SW_p(\nabla \xi(x))$ for a.e. $x \in \R^k$.
Therefore $\xi$ is an admissible support function of the pair $(A_-, A_+)$.

Since $\xi = \psi$ in a neighborhood of $A_-^c \cap A_+^c$, we conclude that
$\Lambda_p[\xi] = \Lambda_p[\psi_1]$ a.e. on $A_-^c \cap A_+^c$ due to
Proposition~\ref{pr:lambda-well-defined}.

Now we $L$-periodically extend $\xi$ and $w$ from $[-L/2, L/2)^k$ to $\R^k$
and call them $\psi_2$ and $z_2$, respectively.
This gives a support function of an open pair $(A_- + L \Z^k, A_+ + L \Z^k)$
and clearly $\SlCH_p(\psi_2; \Gamma) \ni z_2$.

By construction, $\psi_2 = \psi_1$ on $H$ due to \eqref{xi-H-eq}.
\end{proof}

\subsection{Comparison principle for the crystalline curvature}

We can prove the following comparison theorem for the crystalline curvature of ordered facets, as
in \cite{GGP13JMPA}. This will imply that $\Lambda_p[\psi]$ on a given admissible pair is in fact
independent of the choice of an admissible support function $\psi$, Corollary~\ref{co:lambda support func indep} below.

\begin{proposition}[Comparison principle for $\Lambda_p$]
\label{pr:comparison Lambda}
Let $p \in \Rn$ such that $k = \dim \partial W(p) > 0$.
Suppose that $(A_{1,-}, A_{1,+})$ and $(A_{2,-}, A_{2,+})$ are two $p$-admissible pairs in
$\mathcal P^k$.
If the pairs are ordered in the sense of
\begin{align*}
(A_{1,-}, A_{1,+}) \prec (A_{2,-}, A_{2,+}),
\end{align*}
then for any two $p$-admissible support functions $\psi_1$ and $\psi_2$
of the respective pairs we have
\begin{align}
\label{curv-ordered-facet}
\Lambda_p[\psi_1](x) \leq \Lambda_p[\psi_2](x) \quad \text{a.e. $x \in A_{1,-}^c \cap A_{1,+}^c \cap
A_{2,-}^c \cap A_{2,+}^c$.}
\end{align}
\end{proposition}

Before proceeding with the proof, we first give a technical lemma, which is a variant of
\cite[Lemma 4.2.9]{G06}; such a result goes back to \cite{CGG,ES} to establish a uniqueness of a
level set flow.

\begin{lemma}
\label{le:Lipschitz ordering}
Suppose that $\psi$ and $\varphi$ are two nonnegative periodic Lipschitz functions on $\R^d$, $d \geq 1$, such that
$\set{\psi = 0} \subset \set{\varphi = 0}$.
Then there exists a Lipschitz continuous function $\theta: [0, \infty) \to [0, \infty)$ such that
$\theta(0) = 0$, $\theta(s) > 0$ for $s > 0$ and $\theta'(s) > 0$ for almost every $s > 0$
and we have
\begin{align*}
\theta \circ \varphi \leq \psi \qquad \text{on $\R^d$.}
\end{align*}
\end{lemma}

\begin{proof}
We may assume that $\set{\psi = 0} \neq \emptyset$, otherwise the statement is trivial.
We define
\begin{align*}
\eta(s) := \inf \set{\psi(x): \varphi(x) \geq s}.
\end{align*}
Clearly by compactness $\eta(0) = 0$, $\eta(s) > 0$ for $s > 0$. Furthermore, $\eta$ is nondecreasing since
$s \mapsto \set{\varphi \geq s}$ is nonincreasing. Finally, $\eta \circ \varphi \leq \psi$ as
\begin{align*}
\eta(\varphi(x)) = \inf \set{\psi(y): \varphi(y) \geq \varphi(x)} \leq \psi(x).
\end{align*}

As $\eta$ can have jumps or be infinite, we now consider
\begin{align*}
\sigma(s) := \inf \set{\eta(t) + |s - t|: 0 \leq t \leq s}.
\end{align*}
We immediately obtain $0\leq \sigma(s) \leq \eta(s)$ and $\sigma(0) = 0$. On the other hand,
$\eta(t) + |s - t| \geq \min \set{\frac s2, \eta(\frac s2)} > 0$ for $s > 0$, $t \in [0, s]$, and
so $\sigma(s) > 0$ for $s > 0$.
As for monotonicity, a simple estimate for $s \geq u \geq 0$ yields
\begin{align*}
\sigma(s) &= \min \set{\inf \set{\eta(t) + |s - t|: 0 \leq t \leq u}, \inf \set{\eta(t) + |s - t|: u
\leq t \leq s}}\\
&\geq \min \set{\sigma(u) + |s - u|, \eta(u)}\\
&\geq \sigma(u).
\end{align*}
We also show that $\sigma$ is Lipschitz. Take $0 \leq u \leq s$ and $\delta > 0$ and find $t \in
[0,u]$ such that $\sigma(u) > \eta(t) + |u - t| - \delta$. Then we have
\begin{align*}
\sigma(s) \leq \eta(t) + |s - t| = \eta(t) + |u - t| + |s - u| < \sigma(u) + |s - u| + \delta.
\end{align*}
Since $\delta$ was arbitrary, $\sigma$ is Lipschitz.

Finally, set
\begin{align*}
\theta(s) := (1 - e^{-s}) \sigma(s).
\end{align*}
Clearly $\theta(0) = 0$, $\theta(s) > 0$ for $s > 0$. The product rule yields $\theta'(s) > 0$ for
almost every $s > 0$.
By construction,
\begin{align*}
\theta \circ \varphi \leq \sigma \circ \varphi \leq \eta \circ \varphi \leq \psi.
\end{align*}
\end{proof}

Now we complete the proof of the comparison principle for the crystalline curvature $\Lambda_p$.

\begin{proof}[Proof of Theorem~\ref{pr:comparison Lambda}]
By Proposition~\ref{pr:curvature-as-min-section},
we can for a sufficiently large $L > 0$ find
$L$-periodic functions, called $\tilde\psi_1$ and $\tilde\psi_2$,
such that $\SlCH_p(\tilde\psi_i; \Gamma)$ is nonempty, $\Gamma = \R^k / L \Z^k$,
and $\tilde\psi_i$ coincides with the original $\psi_i$ on the neighborhood of
the facet $A_{i,-}^c \cap A_{i,+}^c$, $i = 1, 2$,
and that
\begin{align}
\label{l-is-cononrest}
\Lambda_p[\psi_i] = -\partial^0 \SE_p(\tilde\psi_i; \Gamma)
\qquad \text{a.e. on $A_{i,-}^c \cap A_{i,+}^c$,
$i =1,2$.}
\end{align}

Since the pairs are ordered, if we consider the sets $A_{i, \pm}$ as subsets of $\Gamma$ we have
\begin{align*}
\{\tilde\psi_{2,+} = 0\} = A_{2,+}^c \subset A_{1,+}^c = \{\tilde\psi_{1,+} = 0\},\\
\{\tilde\psi_{1,-} = 0\} = A_{1,-}^c \subset A_{2,-}^c = \{\tilde\psi_{2,-} = 0\},
\end{align*}
where $\tilde\psi_{i,\pm} := \max (\pm\tilde\psi_i, 0)$ denote the positive and negative parts.
By Lemma~\ref{le:Lipschitz ordering}, there exist Lipschitz functions $\theta^-$ and $\theta^+$
on $[0, \infty)$ such that $\theta^\pm(0) = 0$, $\theta^\pm(s) > 0$ for $s > 0$, and
$(\theta^\pm)'(s) > 0$ for almost all $s > 0$, such that $\theta^+ \circ \tilde\psi_{1,+} \leq
\tilde\psi_{2,+}$ and $\theta^- \circ \tilde\psi_{2,-} \leq \tilde\psi_{1,-}$.
We introduce
\begin{align*}
\theta_1(s) :=
\begin{cases}
s, & s <0,\\
\theta^+(s), & s \geq 0,
\end{cases}
\qquad
\theta_2(s) :=
\begin{cases}
-\theta^-(-s), & s <0,\\
s, & s \geq 0.
\end{cases}
\end{align*}
and
\begin{align*}
\xi_1 := \theta_1 \circ \tilde\psi_1, \qquad \xi_2 := \theta_2 \circ \tilde\psi_2.
\end{align*}
By construction we have that $\xi_i$ are Lipschitz on $\Gamma$,
\begin{align*}
\xi_1 \leq \xi_2,
\end{align*}
and the chain rule for Lipschitz functions yields
\begin{align*}
\nabla \xi_i(x) = \theta_i'(\xi_i(x)) \nabla \tilde\psi_i(x), \qquad \text{for almost every $x$},
\end{align*}
if we interpret the right-hand side to be equal to zero if $\nabla \tilde\psi_i(x)$ is
zero, no matter if $\theta_i'$ is differentiable at $\xi_i(x)$ or not.
Since $\theta_i'(s) > 0$ for almost every $s \in \R$, we have
by the positive one-homogeneity of $\SW_p$
\begin{align*}
\partial \SW_p(\nabla \xi_i(x)) = \partial \SW_p(\nabla \tilde\psi_i(x)) \qquad \text{for almost
every $x$},
\end{align*}
and therefore
\begin{align}
\label{subd-bent-same}
\SlCH_p(\xi_i; \Gamma) = \SlCH_p(\tilde\psi_i; \Gamma) \neq \emptyset.
\end{align}

The functional $\SE_p(\cdot; \Gamma)$
is proper closed convex and therefore the resolvent problems
\begin{align*}
\zeta_i + \lambda \partial \SE_p(\zeta_i; \Gamma) \ni \xi_i
\end{align*}
have unique solutions $\zeta_i \in L^2(\Gamma)$.

By approximation by smooth problems that have a comparison principle,
as in Proposition~\ref{pr:resolvent-problems} and its proof,
we can deduce that $\zeta_i$ are Lipschitz since $\xi_i$ are Lipschitz, and
\begin{align*}
\zeta_1 \leq \zeta_2.
\end{align*}
On the intersection of the facets $K = A_{1,-}^c \cap A_{1,+}^c \cap A_{2,-}^c \cap A_{2,+}$
we have $\xi_1 = \xi_2 = 0$ and therefore
\begin{align}
\label{resol-ordering}
\frac{\zeta_1 - \xi_1}\lambda \leq \frac{\zeta_2 - \xi_2}\lambda \qquad \text{on $K$.}
\end{align}
By \eqref{subd-bent-same} and the characterization
of the subdifferential
Proposition~\ref{pr:subdiff-char-periodic},
we know that $\xi_i \in \domain(\partial \SE_p(\cdot; \Gamma))$
and therefore the standard result \cite[Proposition~3.56]{Attouch} yields
\begin{align*}
\frac{\zeta_i - \xi_i}\lambda \to - \partial^0 \SE_p(\xi_i; \Gamma) \qquad \text{in $L^2(\Gamma)$
as $\lambda\to0$.}
\end{align*}
We can send $\lambda\to0$, and
then use
\eqref{l-is-cononrest}, \eqref{subd-bent-same} and the ordering \eqref{resol-ordering}
to conclude that
\begin{align*}
\Lambda_p[\psi_1] = -\partial^0 \SE_p(\xi_1; \Gamma)
\leq -\partial^0 \SE_p(\xi_2; \Gamma) = \Lambda_p[\psi_2] \qquad \text{a.e. on $K$.}
\end{align*}
This is the comparison principle for the $\Lambda_p$.
\end{proof}

The following result is an immediate consequence of Proposition~\ref{pr:comparison Lambda}.

\begin{corollary}
\label{co:lambda support func indep}
Suppose that $p \in \R^n$ with $k := \dim \partial W(p) > 0$ and let $(A_-, A_+) \in \mathcal
P^k$ be a $p$-admissible pair. Then the value of $\Lambda_p$ on the facet $A_-^c \cap A_+^c$ is
independent of the choice of a $p$-admissible support function, that is, for any two $p$-admissible
support functions $\psi, \xi$ of pair $(A_-, A_+)$ we have
\begin{align*}
\Lambda_p[\psi] = \Lambda_p[\xi] \qquad \text{a.e. on $A_-^c \cap A_+^c$.}
\end{align*}
\end{corollary}

\section{Viscosity solutions}
\label{sec:viscosity solutions}

In this section we introduce viscosity solutions of problem \eqref{P}.
For the definition of viscosity solutions we shall use \emph{stratified
faceted functions}
that rely on the concept of energy stratification that we have developed in
Section~\ref{sec:energy-stratification}.
Recall that for every $\hat p \in \Rn$
we have introduced the coordinate system $x = \TT(x', x'')$
using the rotation $\TT = \TT_{\hat p}$ from \eqref{rotation}.

\begin{definition}
\label{def:strat-faceted-test-function}
Let $(\hat x, \hat t) \in \Rn \times \R$ and $\hat p \in \Rn$, $V \subset \Rn$ be the subspace
parallel to $\aff \partial W(\hat p)$, $U = V^\perp$, $k = \dim V$.
We say that a function $\vp(x,t)$ is a
\emph{stratified faceted test function at $(\hat x, \hat t)$ with gradient
$\hat p$}
if
\begin{align*}
\vp(x,t) = \bar \psi\pth{x' - \hat x'}
    + f\pth{x'' - \hat x''}
    + \hat p \cdot x + g(t),
\end{align*}
where
\begin{itemize}
\item $\bar \psi: \R^k \to \R$ is a support function of a bounded facet
    $(A_-, A_+) \in \mathcal P^k$ with
    $0 \in \interior (A_-^c \cap A_+^c)$
    and $\bar \psi \in \domain(\Lambda_{\hat p})$,
\item $f \in C^2(\R^{n - k})$, $f(0) = 0$ and $\nabla f(0) = 0$,
\item $g \in C^1(\R)$.
\end{itemize}
\end{definition}

With this notion of test functions, we define viscosity solutions.

\begin{definition}
\label{def:visc-solution}
An upper semi-continuous function $u: \cl Q \to \R$ is a \emph{viscosity subsolution}
of \eqref{P} if the following hold:
\begin{enumerate}[(i)]
\item \emph{(faceted test)}
Let $\vp$ be a stratified faceted test function at $(\hat x, \hat t) \in Q$
with gradient $\hat p \in \Rn$ and pair $(A_-, A_+)$.
Then if there is $\rho > 0$ such that
\begin{align}
\label{general-position}
u(x + w,t) - \vp(x,t) \leq u(\hat x, \hat t) - \vp(\hat x, \hat t)
\end{align}
for all
\begin{align*}
\abs{w'} \leq \rho,\ w'' = 0, \quad \text{and }
x' - \hat x' \in \nbd\rho(\facet A),\ \abs{x'' - \hat x''} \leq \rho,\ \abs{t - \hat t} \leq \rho,
\end{align*}
then there exists $\de > 0$ such that $B_\delta(\hat x') \subset \interior(\facet A)$ and
\begin{align*}
\vp_t(\hat x, \hat t) + F(\hat p, \essinf_{B_\de(0)} \Lambda_{\hat p}[\bar \psi]) \leq 0.
\end{align*}
\item \emph{(off-facet test)} Let $\vp \in C^1(\mathcal U)$ where $\mathcal U$ is a neighborhood of some
point $(\hat x, \hat t) \in Q$ and suppose that
$\dim \partial W(\nabla \vp(\hat x, \hat t)) = 0$.
If $u - \vp$ has a local maximum at $(\hat x, \hat t)$ then
\begin{align*}
\vp_t(\hat x, \hat t) + F(\nabla \vp(\hat x, \hat t), 0) \leq 0.
\end{align*}
\end{enumerate}

Supersolutions are defined analogously.
\end{definition}

If for some $p$ the value of $F(p, \xi)$ does not depend on $\xi$
in the sense below,
we can replace the faceted test by a simpler test
that does not need an admissible faceted function.

\begin{definition}[Curvature-free type at $p_0$]
\label{def:level-set-type}
We say that $F$ is of \emph{curvature-free type} at $p_0 \in \Rn$ if
we have for any constant $C > 0$
\begin{align*}
\lim_{p \to p_0} \sup_{\abs\zeta \leq C} F(p, \zeta) = F(p_0, 0) =
\lim_{p \to p_0} \inf_{\abs\zeta \leq C} F(p, \zeta).
\end{align*}
\end{definition}

\begin{remark}
The function $F$ defined in \eqref{geometric F} is of curvature-free type at $p_0 = 0$.
\end{remark}

\begin{definition}[Faceted test at curvature-free gradients]
\label{def:level-set-test}
If $F$ is of curvature-free type at $p_0 = 0$, we replace
the faceted test (i) in Definition~\ref{def:visc-solution}
at $\hat{p} = p_0 = 0$
by the following test:
\begin{enumerate}
\item[(i-cf)] Let $g \in C^1(\R)$, $\vp(x, t) := g(t)$ and suppose that $u - \vp$ has a local maximum at $(\hat x, \hat t)$.
Then
\begin{align*}
g'(\hat t) + F(0, 0) = g'(\hat t) \leq 0.
\end{align*}
\end{enumerate}
\end{definition}

\section{Construction of faceted functions}
\label{sec:faceted functions}

To prove the uniqueness of viscosity solutions of \eqref{P},
we need to be able to construct a sufficiently wide class of test functions,
the \emph{faceted functions}.
In this section we will assume that $W$ is convex, positively one-homogeneous and
crystalline.
We shall also assume that there exists $\delta > 0$ such that $W(p) \geq \delta \abs p$.
The important case for us is $\SW_p$ from Definition~\ref{def:sliced-W}.

The polar function $W^\circ$ of $W$
is defined as
\begin{align}
\label{polar}
W^\circ(x) = \sup \set{x \cdot p: W(p) \leq 1}.
\end{align}
Clearly
\begin{align*}
(W^\circ)^\circ = W.
\end{align*}
We define the Wulff shape corresponding to $W$ as
\begin{align*}
\Wulff_W := \set{x \in \Rn: W^\circ(x) \leq 1}.
\end{align*}
Note that the Wulff shape of a one-homogeneous crystalline (polyhedral) $W$ with linear growth is a bounded polyhedron containing the
origin in its interior.

We want to establish a proposition similar to \cite[Proposition~2.12]{GGP13AMSA}, but for a crystalline energy:

\begin{proposition}
\label{pr:approximate-pair}
Let $k = 1$ or $2$, $(A_-, A_+) \in \mathcal P^k$ be a \emph{bounded} pair and let
$0 \leq \rho_1 < \rho_2$. Suppose that $W: \R^k \to \R$ is a convex, positively one-homogeneous
polyhedral function such that there exists $\delta > 0$ with $W(p) \geq \delta |p|$ for $p \in \R^k$. Then there
exists an \emph{admissible pair} $(G_-, G_+) \in \mathcal P^k$ such that
\begin{align}
\label{admiss pair approx}
\nbd{\rho_1}(A_-, A_+) \preceq (G_-, G_+) \preceq \nbd{\rho_2}(A_-, A_+),
\end{align}
that is, there exists a support function $\psi$ of pair $(G_-, G_+)$ such that $\CH_W(\psi; \R^k)$
is nonempty.
\end{proposition}

We shall use this result in the following form:

\begin{corollary}
\label{co:approximate pair sliced}
Let $W: \R^n \to \R$ be a polyhedral convex function finite everywhere. Suppose that $p_0 \in \R^n$
such that $\dim \partial W(p_0) = k$ for $k = 1$ or $2$. Then for any bounded pair $(A_-, A_+) \in
\mathcal P^k$ and any $0 \leq \rho_1 < \rho_2$ there exists a $p_0$-\emph{admissible pair} $(G_-,
G_+)$ satisfying \eqref{admiss pair approx}.
\end{corollary}

\begin{proof}
Let us take $\xi_0 \in \ri \partial W(p_0)$. The function $\hat W(p) := \SW_{p_0}(p) - \xi_0' \cdot p$ satisfies the assumptions of Proposition~\ref{pr:approximate-pair} by Lemma~\ref{le:linear growth
SW}. Therefore there exists a pair $(G_-, G_+) \in \mathcal P^k$, its support function $\psi$ and a
Cahn-Hoffman vector field $z \in \CH_{\hat W}(\psi; \R^k)$. It is easy to see that $z +
\xi_0 \in \SlCH_{p_0}(\psi; \R^k)$, and therefore $(G_-, G_+)$ is $p_0$-admissible.
\end{proof}

As of now we only know how to construct such admissible facets for dimensions $k = 1$ and $k = 2$.

For the construction of an admissible function we will basically use a signed-distance-like function
induced by $W$, and then define a possible Cahn-Hoffman vector field for this function.
For a given set $V \subset \R^k$ the signed-distance-like function $d_V$ is defined as
\begin{align}
\label{d_V-def}
d_V(x) := \inf_{y\in V} W^\circ(x - y) -\inf_{y\in V^c} W^\circ(y - x), \quad x \in \R^k,
\end{align}
where $W^\circ$ is the polar of $W$ given as \eqref{polar}.

\subsection{One-dimensional admissible facets}

We will give an explicit construction as a proof of Proposition~\ref{pr:approximate-pair} in the
one-dimensional case to illustrate the process and hopefully prepare the reader for the construction in the two-dimensional case.
\begin{figure}
\centering
\def\svgwidth{4.5in}
\begingroup%
  \makeatletter%
  \providecommand\color[2][]{%
    \errmessage{(Inkscape) Color is used for the text in Inkscape, but the package 'color.sty' is not loaded}%
    \renewcommand\color[2][]{}%
  }%
  \providecommand\transparent[1]{%
    \errmessage{(Inkscape) Transparency is used (non-zero) for the text in Inkscape, but the package 'transparent.sty' is not loaded}%
    \renewcommand\transparent[1]{}%
  }%
  \providecommand\rotatebox[2]{#2}%
  \ifx\svgwidth\undefined%
    \setlength{\unitlength}{324bp}%
    \ifx\svgscale\undefined%
      \relax%
    \else%
      \setlength{\unitlength}{\unitlength * \real{\svgscale}}%
    \fi%
  \else%
    \setlength{\unitlength}{\svgwidth}%
  \fi%
  \global\let\svgwidth\undefined%
  \global\let\svgscale\undefined%
  \makeatother%
  \begin{picture}(1,0.22222222)%
    \put(0,0){\includegraphics[width=\unitlength,page=1]{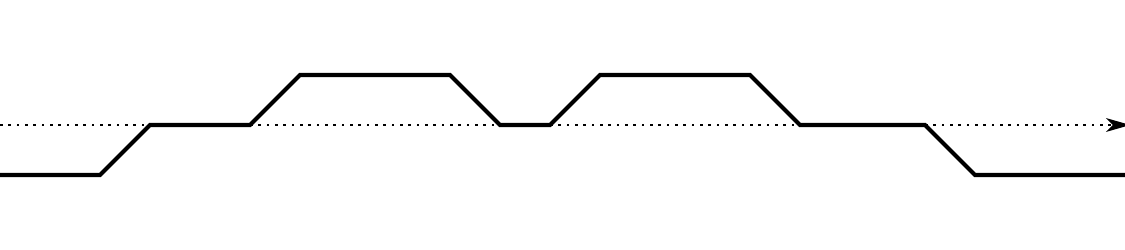}}%
    \put(0.42592593,0.02733674){\color[rgb]{0,0,0}\makebox(0,0)[lb]{\smash{$b_1$}}}%
    \put(0.11992946,0.02733674){\color[rgb]{0,0,0}\makebox(0,0)[lb]{\smash{$b_0$}}}%
    \put(0.21075837,0.02733674){\color[rgb]{0,0,0}\makebox(0,0)[lb]{\smash{$a_1$}}}%
    \put(0.80158729,0.02733674){\color[rgb]{0,0,0}\makebox(0,0)[lb]{\smash{$a_m$}}}%
    \put(0.68606703,0.02733674){\color[rgb]{0,0,0}\makebox(0,0)[lb]{\smash{$b_{m-1}$}}}%
    \put(0.56701939,0.07142853){\color[rgb]{0,0,0}\makebox(0,0)[lb]{\smash{$G_+$}}}%
    \put(0.28835979,0.07407399){\color[rgb]{0,0,0}\makebox(0,0)[lb]{\smash{$G_+$}}}%
    \put(0.00176366,0.14462069){\color[rgb]{0,0,0}\makebox(0,0)[lb]{\smash{$G_-$}}}%
    \put(0.85978838,0.14462069){\color[rgb]{0,0,0}\makebox(0,0)[lb]{\smash{$G_-$}}}%
    \put(0.9656085,0.08201052){\color[rgb]{0,0,0}\makebox(0,0)[lb]{\smash{$x$}}}%
    \put(0,0){\includegraphics[width=\unitlength,page=2]{admissible-pair-1d.pdf}}%
    \put(0.12433863,0.19047617){\color[rgb]{0,0,0}\makebox(0,0)[lb]{\smash{$\psi$}}}%
    \put(0.14827216,0.14550261){\color[rgb]{0,0,0}\makebox(0,0)[lb]{\smash{$\delta$}}}%
    \put(0.12380952,0.05820093){\color[rgb]{0,0,0}\makebox(0,0)[lb]{\smash{$-\delta$}}}%
    \put(0.23104056,0.17989403){\color[rgb]{0,0,0}\makebox(0,0)[lb]{\smash{$a_1+\delta$}}}%
    \put(0.36331572,0.17989403){\color[rgb]{0,0,0}\makebox(0,0)[lb]{\smash{$b_1-\delta$}}}%
    \put(0,0){\includegraphics[width=\unitlength,page=3]{admissible-pair-1d.pdf}}%
  \end{picture}%
\endgroup%
\caption{Construction of an one-dimensional admissible pair and its support function}
\label{fig:addissible-pair-1d}
\end{figure}
The situation is depicted in Figure~\ref{fig:addissible-pair-1d}.

Let $(A_-, A_+) \subset \mathcal P^1$ be a bounded pair in $\R$.
By making $\rho_1$ larger if necessary, we can assume that $0 < \rho_1 < \rho_2$.
Let us set $\e := \frac{\rho_2 - \rho_1}3$.

We define the open sets
\begin{align*}
G_- := \interior \cl{\nbd\e\pth{\nbd{-\rho_2}(A_-)}} \qquad \text{and} \qquad
G_+ := \interior \cl{\nbd{\rho_1 + \e}(A_+)}.
\end{align*}
Due to the properties of the set neighborhood, we have for all $\eta > 0$
\begin{align}
\label{G-pm-eta}
\nbd{-\rho_2} (A_-) \subset G_- \subset \nbd{-\rho_2 + \e+ \eta}(A_-), \qquad \nbd{\rho_1}(A_+) \subset G_+ \subset \nbd{\rho_1 + \e + \eta}(A_+).
\end{align}
In particular,
we take the interior of the closure in the definition of $G_\pm$ to regularize the boundary so that $G^c_\pm$ has no isolated points.

By definition $A_-\subset A_+^c$, and therefore Proposition~\ref{pr:nbd-properties} together with \eqref{G-pm-eta} imply that for any $\eta \in (0, 2\e)$
\begin{align*}
G_- &\subset \nbd{\e+\eta}\pth{\nbd{-\rho_2}(A_-)} \subset \nbd{-\rho_2 + \e + \eta}(A_-)\\
&\subset \nbd{-\rho_2 + \e + \eta}(A_+^c) \subset \nbd{-\e + 2\eta}\pth{\nbd{-\rho_1 - \e - \eta}(A_+^c)}\\
&= \nbd{-\e + 2\eta}\pth{\nbd{\rho_1 + \e + \eta}(A_+)^c}\\
&= \nbd{\e - 2\eta}\pth{\nbd{\rho_1 + \e + \eta}(A_+)}^c \subset \nbd{\e - 2\eta}\pth{G_+}^c
\end{align*}
We conclude that
\begin{align*}
\dist(G_-, G_+) = \e > 0.
\end{align*}
Therefore $(G_-, G_+)$ is an open pair, and due to \eqref{G-pm-eta}
\begin{align*}
\nbd{\rho_1}(A_-, A_+) \preceq (G_-, G_+) \preceq \nbd{\rho_2}(A_-, A_+),
\end{align*}
To prove that the pair $(G_-, G_+)$ is bounded, we recall that $(A_-, A_+)$ is a bounded pair therefore there exists $R > 0$ such that $B_R^c(0) \subset A_-$ or $B_R^c(0) \subset A_+$.
From \eqref{G-pm-eta} we have that $\nbd{-\rho_2}(B_R^c(0)) \subset G_-$ or $\nbd{\rho_1}(B_R^c(0)) \subset A_+$.
Therefore $B_{\tilde R}^c(0) \subset G_-$ or $B_{\tilde R}^c(0) \subset G_+$ for $\tilde R = R + \rho_2$, which implies that $(G_-, G_+)$ is bounded.

Since $G_\pm$ are open, we can write the union $G_- \cup G_+$ as at most a countable union of disjoint open intervals.
Since the facet $G_-^c \cap G_+^c$ is bounded, and the sets $G_\pm$ have the interior ball property with radius $\e$ by construction, the length of the intervals must be greater than or equal to $2\e$.
In particular, there must only be finitely many of them.
Since moreover $\dist (G_-, G_+) > 0$, we can find $m \in \N$ and $\set{a_i}_{i=0}^m$, $\set{b_i}_{i=0}^m$ such that
\begin{align*}
-\infty = a_0 < b_0 < a_1 < b_1 < \cdots < a_m < b_m = \infty
\end{align*}
and
\begin{align*}
G_- \cup G_+ = \bigcup_{i=0}^m (a_i, b_i).
\end{align*}
Finally, by construction,
\begin{align}
\label{def-of-delta}
\delta := \frac13 \min \set{\min_{0 \leq i \leq m} b_i - a_i, \min_{1 \leq i \leq m} a_i - b_{i-1}} > 0.
\end{align}
The facet $G_-^c \cap G_+^c$ is closed and
\begin{align*}
G_-^c \cap G_+^c = \bigcup_{i=1}^m [b_{i-1}, a_i].
\end{align*}

Let us now introduce the sign function
\begin{align*}
\sigma(x) :=
\begin{cases}
1 & x \in \cl{G_+},\\
-1 & x \in \cl{G_-},\\
0 & \text{otherwise}.
\end{cases}
\end{align*}
This allows us to define the function
\begin{align*}
\psi(x) := \min\set{\delta, \dist(x, G_+^c)} - \min \set{\delta, \dist(x, G_-^c)},
\end{align*}
as a clipped version of the function in Example~\ref{ex:trivial-support-function},
which is again clearly a support function of the pair $(G_-, G_+)$.
Moreover,
\begin{align*}
\psi(x) =
\begin{cases}
\delta \sigma(x) & x \in [a_i + \delta, b_i - \delta] \text{ for some $i$},\\
0 & x \in [b_{i-1}, a_i] \text{ for some $i$},\\
(x-a_i)\sigma(x) & x \in (a_i, a_i + \delta) \text{ for some $i$},\\
(b_i -x)\sigma(x) & x \in (b_i - \delta, b_i) \text{ for some $i$}.
\end{cases}
\end{align*}
Therefore the function $\psi$ is differentiable everywhere except at the points $a_i, b_i, a_i + \delta, b_i - \delta$ for $0 \leq i \leq m$.
We can evaluate the derivative at the other points as
\begin{align*}
\psi'(x) =
\begin{cases}
0 & x \in (a_i + \delta, b_i -\delta) \cup (b_{i-1}, a_i) \text{ for some $i$},\\
\sigma(x) & x \in (a_i, a_i + \delta) \text{ for some $i$},\\
-\sigma(x) & x \in (b_i - \delta, b_i) \text{ for some $i$}.
\end{cases}
\end{align*}

In one dimension, the subdifferential of one-homogeneous $W$ can be expressed as
\begin{align*}
\partial W(p) =
\begin{cases}
\set{w_-} & p < 0,\\
[w_-, w_+] & p = 0,\\
\set{w_+} & p < 0,
\end{cases}
\end{align*}
for $w_\pm = W'(\pm 1)$,
$w_- < 0 < w_+$.

Let us define the continuous Cahn-Hoffman vector field as
\begin{align*}
z(x) :=
\begin{cases}
W'(\sigma(x)) & x \in (a_i, a_i + \delta) \text{ for some $i$},\\
W'(-\sigma(x)) & x \in (b_i - \delta, b_i) \text{ for some $i$},\\
W'(\sigma(b_0)) & x \leq b_0 - \delta,\\
W'(\sigma(a_m)) & x \geq a_m + \delta,\\
\text{linear} & \text{otherwise},
\end{cases}
\end{align*}
One can easily see that the function $z$ is Lipschitz continuous on $\R$ and $\norm{\nabla z}_\infty \leq \frac{w_+ - w_-}\delta \leq \infty$ by the definition of $\delta$ in \eqref{def-of-delta}.
Therefore $\psi \in \domain(\partial E)$ and the facet $(G_-, G_+)$ is admissible, which finishes the proof of Proposition~\ref{pr:approximate-pair} in the case of $k = 1$.

\subsection{Two-dimensional admissible facets}

In this section we give a proof of Proposition~\ref{pr:approximate-pair} in the two-dimensional
case.
We can without loss suppose that $\rho_1 = 0$ and $\rho_2 = \rho > 0$.
Let us stress again that we do not assume that the Wulff shape of $W$ is symmetric with respect to the origin.

The proof of Proposition~\ref{pr:approximate-pair} for $k = 2$ uses a rather simple idea of an explicit construction that is unfortunately quite technical. It will be split in several
steps:
\begin{enumerate}[1.]
\item Approximate a general bounded facet by a smooth facet.
\item Rotate the smooth facet by a small angle so that the boundary has nonzero curvature at the
points where the normal is pointing in the direction of a corner of $W$.
\item Flatten the boundary locally at these points.
\item Use the Fenchel distance-like function induced by $W$ to construct a support function and a
Cahn-Hoffman vector field in the neighborhood of the boundary.
\end{enumerate}

We define the set of critical directions,
\begin{align*}
\mathcal N := \set{p \in S^1: \partial W(p) \text{ is not a singleton}} = \set{p \in S^1: W \text{
is not differentiable at $p$}},
\end{align*}
where $S^1 := \set{p \in \R^2: |p| = 1}$ is the unit circle.
Since $W$ is polyhedral, $\mathcal N$ is finite.
\begin{lemma}
\label{le:}
$\partial W: p \to 2^{\R^2}$ is constant on every connected component of $S^1 \setminus \mathcal
N$. Moreover, $\partial W(p)$ is a singleton for every such $p$.
\end{lemma}

\begin{proof}
This follows from the fact that $W$ is polyhedral.
\end{proof}

We will also use some basic results of the convex analysis. In particular, recall the definition of
the polar $W^\circ$ in \eqref{polar}.
We will for short denote the associate Wulff shape as
\begin{align*}
\mathcal W := \set{x: W^\circ(x) \leq 1}.
\end{align*}
This is a polygon in two dimensions, with a finite number of vertices, corresponding to the number
of critical directions $\mathcal N$.
We have the following basic result:

\begin{lemma}
\label{le:wulff-frank-rel}
If $p \neq 0$ and $x \in \partial W(p)$ then $W^\circ(x) = 1$ and $x \cdot p = W(p)$. Similarly, if $x \neq 0$ and $p \in
\partial W^\circ(x)$ then $W(p) = 1$ and $x \cdot p = W^\circ(x)$.
Suppose now that $x \neq 0$ and $p \neq 0$. Then
\begin{align*}
\frac x{W^\circ(x)} \in \partial W(p) \quad \Leftrightarrow \quad \frac p{W(p)} \in \partial W^\circ(x).
\end{align*}
\end{lemma}

\subsubsection{Smooth pair approximation}

By the smooth approximation lemma, \cite[Lemma~2.11]{GGP13AMSA}, we can find smooth disjoint open
sets $H_-, H_+$ such that
\begin{align}
\label{A-Hausdorff-approx}
\begin{aligned}
\nbd{\rho/2}(A_-, A_+) \preceq (H_-, H_+) \preceq \nbd{3\rho/4}(A_-, A_+).
\end{aligned}
\end{align}
We note that $(H_-, H_+)$ is an smooth bounded pair.

We claim that we can choose $H_-, H_+$ in such a way that
\begin{align}
\label{nonzero curvature}
\begin{aligned}
\text{the curvature of $\partial H_-$ and
$\partial H_+$ at $x$ is nonzero}\\
\text{whenever $\nu_{\partial H_-}(x) \in \mathcal N$ or $-\nu_{\partial
H_+}(x) \in \mathcal N$, respectively.}
\end{aligned}
\end{align}
Indeed, let $V$ be $H_-$ or $\interior H_+^c$. Since $\partial V$ is smooth and bounded, it is a union of finitely many disjoint closed curves.
Each of these curves is a one-dimensional manifold without boundary and the unit outer normal
vector map $\nu: \partial V
\to S^1$ is smooth. By Sard's theorem we have $\mathcal H^1\pth{\nu\pth{\set{x \in \partial V: d\nu(x) \text{
has rank } < 1}}} = 0$. Note that the curvature $\kappa(x)$ of $\partial V$ at $x \in \partial V$
is zero if and only if the rank of $d\nu(x)$ is zero. Since the set of critical directions
$\mathcal N \subset S^1$ is finite, we can find a rotation $R$ of $\R^2$ by an arbitrary small
angle such that $R(\mathcal N) \cap \nu(\set{x \in \partial V : \kappa(x) = 0}) = \emptyset$. We therefore rotate
the set $V$ by $R^{-1}$ with a sufficiently small such angle so that the rotated set still
approximates the original one. Therefore whenever $x \in R^{-1}(V)$ such that
$\kappa_{R^{-1}(\partial V)}(x) = 0$, we
have $\nu_{R^{-1}(\partial V)}(x) \notin \mathcal N$.
We can therefore replace $H_-$ and $H_+$ with the rotated ones by a sufficiently small angle if
necessary and then $H_\pm$ satisfy \eqref{nonzero curvature}.

\subsubsection{Flattening of $\partial H_\pm$ in the critical directions}

Let $V$ denote either $H_-$ or $\interior H_+^c$ in what follows and let $\nu(x) = \nu_{\partial V}(x)$ be
the unit outer normal to $\partial V$ at $x \in \partial V$. We will modify $V$ in the
neighborhood of the critical points of its boundary $x \in \partial V$ with $\nu(x)
\in \mathcal N$ so that the boundary of the modified set has a flat part of nonzero length with the
same normal. Let us denote the set of these critical points by $S$,
\begin{align*}
S := \set{x \in \partial V: \nu(x) \in \mathcal N}.
\end{align*}
Note that $S$ is compact since $\nu$ is smooth and $\partial V$ is bounded.

We claim that $S$ is finite. Indeed, suppose that $S$ is infinite. Since $S$ is compact, there is
$\hat x \in S$ such that $B_\e(\hat x) \cap S$ is infinite for every $\e > 0$. Since $\mathcal N$
is discrete and $\nu$ is continuous, there exists $\e_0 > 0$ such that $\nu(x) \equiv \nu(\hat x)$
for all $x \in B_{\e_0}(\hat x) \cap S$. But that is a contradiction with $d \nu(\hat x) \neq 0$
from \eqref{nonzero curvature}.

Let us choose $\eta > 0$ such that
\begin{align*}
\eta < \min \set{\frac 1{40}\dist(\partial H_-,\partial H_+), \frac \rho8,
\min_{\substack{x, y \in S\\x \neq y}} |x - y|}.
\end{align*}
Since for any $\hat x \in S$ we have $\kappa(\hat x) \neq 0$, by making $\eta$ smaller if necessary, we may also
assume that $\partial V \cap B_{20\eta}(\hat x)$ is a graph of a convex or a concave function
$g = g_{\hat x}$
 in the sense that
\begin{align*}
V \cap B_{20\eta}(\hat x) = \set{y + \hat x \in B_{20\eta}(\hat x) : y \cdot \nu(\hat x) < g(y
\cdot \tau(\hat x))},
\end{align*}
where $\tau(\hat x) \perp \nu(\hat x)$, $|\tau(\hat x)| = 1$.
Note that $g(0) = g'(0) = 0$.
Since $\kappa(\hat x) \neq 0$, we have $g''(\hat x) \neq 0$ and by Taylor expansion we may also assume that
\begin{align*}
\frac 14 |g''(0)|s^2 \leq |g(s)| \leq |g''(0)| s^2, \quad |s| < 20 \eta.
\end{align*}

With this set-up, we can for every $\hat x \in S$ find $L_{\hat x} > 0$ such that $\set{s:
|g_{\hat x}(s)| < L_{\hat x}} \times [-L_{\hat x}, L_{\hat x}] \subset B_\eta(0)$.
We then define $\hat V$, the set with flattened boundary in the critical directions, as
\begin{align*}
\hat V := &\pth{V \setminus \bigcup_{\hat x \in S} B_{10 \eta}(\hat x)}\\
&\bigcup_{\substack{\hat x \in S\\g_{\hat x}''(0) > 0}} \set{y + \hat x \in B_{10\eta}(\hat x) :
y \cdot \nu(\hat x) < \max(L_{\hat x}, g_{\hat x}(y \cdot \tau(\hat x)))}\\
&\bigcup_{\substack{\hat x \in S\\g_{\hat x}''(0) < 0}} \set{y + \hat x \in B_{10\eta}(\hat x) :
y \cdot \nu(\hat x) < \min(-L_{\hat x}, g_{\hat x}(y \cdot \tau(\hat x)))}.
\end{align*}
Note that $\partial V \subset \nbd\eta(\partial \hat V)$ and $\partial \hat V \subset
\nbd\eta(\partial V)$.

We finish our construction of the admissible pair by defining $G_- = \hat V$ when starting
with $V = H_-$, and $G_+ = \interior \hat V^c$ when starting with $V = \interior H_+^c$.

\subsubsection{Construction of the support function and the Cahn-Hoffman vector field}

In this part we shall finally define a candidate for the admissible function with an appropriate
Cahn-Hoffman vector field in a small neighborhood of the flattened boundary $\partial \hat V$,
where $\hat V = G_-$ or $\hat V = \interior G_+^c$.

Let $\mathcal V$ denote the set of vertices of the Wulff shape $\mathcal W$.
We define $\mathcal C_0$ to be the family of connected components of $\partial \hat V \setminus
\partial V$, and $\mathcal C_r$ to be the family of connected components of $\partial \hat V \cap
\partial V$. We also define $\mathcal C = \mathcal C_0 \cup \mathcal C_r$. Every $\Gamma_0 \in
\mathcal C_0$ is the flattened part of the boundary $\partial \hat V$, the line segment with a
normal vector $\nu_0 \in \mathcal N$. Similarly, every $\Gamma \in \mathcal C_r$ is a connected
piece of the original smooth boundary $\partial V$, and by construction there exists a
unique vertex $v \in \mathcal V$ of the Wulff shape such that $\set{v} = \partial W(\nu(x))$ for $x
\in \Gamma$. We set $\mathcal V(\Gamma) = \set{v}$.

Given $\Gamma \in \mathcal C_0$ with normal $\nu_0$, there exists exactly two distinct vertices $v, w \in \mathcal V$
such that $\set{v, w} \subset \partial W(\nu_0)$. In this case we set $\mathcal V(\Gamma) =
\set{v, w}$.
There exist exactly two sets $\Gamma', \Gamma'' \in \mathcal C_r$ such that $\mathcal V(\Gamma') =
\set{v}$, $\mathcal V(\Gamma'') = \set{w}$, and $\cl \Gamma \cap \Gamma' = \set{x_v}$, $\cl \Gamma
\cap \Gamma'' = \set{x_w}$ for some points $x_v, x_w$;
see Figure~\ref{fig:local-admissible}.
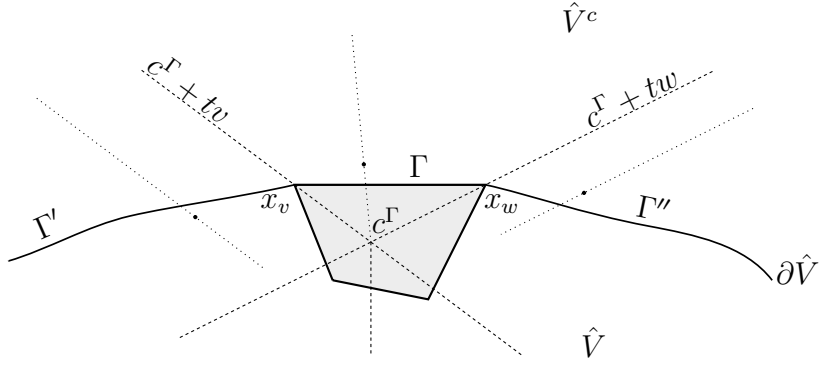
\begin{figure}
\centering
\def\svgwidth{4in}
\begingroup%
  \makeatletter%
  \providecommand\color[2][]{%
    \errmessage{(Inkscape) Color is used for the text in Inkscape, but the package 'color.sty' is not loaded}%
    \renewcommand\color[2][]{}%
  }%
  \providecommand\transparent[1]{%
    \errmessage{(Inkscape) Transparency is used (non-zero) for the text in Inkscape, but the package 'transparent.sty' is not loaded}%
    \renewcommand\transparent[1]{}%
  }%
  \providecommand\rotatebox[2]{#2}%
  \ifx\svgwidth\undefined%
    \setlength{\unitlength}{288bp}%
    \ifx\svgscale\undefined%
      \relax%
    \else%
      \setlength{\unitlength}{\unitlength * \real{\svgscale}}%
    \fi%
  \else%
    \setlength{\unitlength}{\svgwidth}%
  \fi%
  \global\let\svgwidth\undefined%
  \global\let\svgscale\undefined%
  \makeatother%
  \begin{picture}(1,0.5)%
    \put(0,0){\includegraphics[width=\unitlength,page=1]{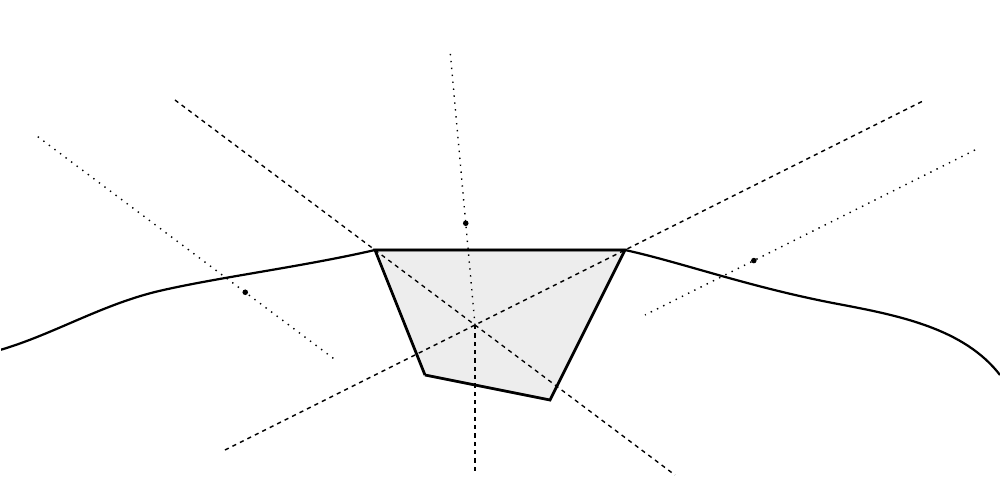}}%
    \put(0.47619277,0.18624524){\color[rgb]{0,0,0}\makebox(0,0)[lb]{\smash{$c^\Gamma$}}}%
    \put(0.525,0.26038208){\color[rgb]{0,0,0}\makebox(0,0)[lb]{\smash{$\Gamma$}}}%
    \put(1.00544951,0.1202816){\color[rgb]{0,0,0}\makebox(0,0)[lb]{\smash{$\partial \hat V$}}}%
    \put(0.15,0.325){\color[rgb]{0,0,0}\makebox(0,0)[lb]{\smash{}}}%
    \put(0.75,0.025){\color[rgb]{0,0,0}\makebox(0,0)[lb]{\smash{$\hat V$}}}%
    \put(0.725,0.45){\color[rgb]{0,0,0}\makebox(0,0)[lb]{\smash{$\hat V^c$}}}%
    \put(0.76245859,0.323688){\color[rgb]{0,0,0}\rotatebox{26.44762343}{\makebox(0,0)[lb]{\smash{$c^\Gamma + t w$}}}}%
    \put(0.17981704,0.40107844){\color[rgb]{0,0,0}\rotatebox{-36.25015355}{\makebox(0,0)[lb]{\smash{$c^\Gamma + t v$}}}}%
    \put(0.04263659,0.1780408){\color[rgb]{0,0,0}\rotatebox{23.80454921}{\makebox(0,0)[lb]{\smash{$\Gamma'$}}}}%
    \put(0.82082068,0.20922166){\color[rgb]{0,0,0}\rotatebox{-11.42487403}{\makebox(0,0)[lb]{\smash{$\Gamma''$}}}}%
    \put(0.33055556,0.21944444){\color[rgb]{0,0,0}\makebox(0,0)[lb]{\smash{$x_v$}}}%
    \put(0.62302195,0.21929311){\color[rgb]{0,0,0}\makebox(0,0)[lb]{\smash{$x_w$}}}%
  \end{picture}%
\endgroup%
\caption{Geometry at a flattened part of the boundary of $\hat V$. The shaded area represents the
rescaled Wulff shape touching the flattened part.}
\label{fig:local-admissible}
\end{figure}
Since $v, w$ are linearly independent, we have a
unique point $c^\Gamma$ at the intersection of $L_v(x_v)$ and $L_w(x_w)$. We have $c^\Gamma + t v = x_v$
and $c^\Gamma + sw = x_w$ for some $t, s \in \R \setminus \set0$.
However, since $(x_w - x_v) \cdot \nu_0 = 0$, we must have $c^\Gamma \cdot \nu_0 + s w \cdot \nu_0
= c^\Gamma \cdot \nu_0 + t v \cdot \nu_0$. As $v \cdot \nu_0 = w \cdot \nu_0 = W(\nu_0)$, it
follows that $t = s$ and we set $\alpha^\Gamma := t$. This induces a coordinate system on $\R^2$
with coordinates $x = \xi_v^\Gamma(x) v + \xi_w^\Gamma(x) w + c^\Gamma$ for every $x \in \R^2$. We note
that
\begin{align}
\label{Gamma alpha level}
\Gamma = \set{x: \xi_v^\Gamma(x) + \xi_w^\Gamma(x) = \alpha^\Gamma,\ \xi_v^\Gamma(x) \xi_w^\Gamma(x)
> 0}.
\end{align}
Clearly $\xi_v^\Gamma(x_v) = \xi_w^\Gamma(x_w) = \alpha^\Gamma$ and $\xi^\Gamma_v(x_w) =
\xi^\Gamma_w(x_v) = 0$.

We define the line through a point $x$ in the direction $v$ as
\begin{align*}
L_v(x) := \set{x + tv: t \in \R},
\end{align*}
and the cylinder through set $\Gamma$
\begin{align*}
L_v(\Gamma) := \set{x + tv: x \in \Gamma,\ t \in \R}.
\end{align*}
The thickness of a cylinder is denoted by
\begin{align*}
\theta(L_v(\Gamma)) := \sup_{x, y \in \Gamma} \dist \pth{L_v(x), L_v(y)}.
\end{align*}

We collect a few basic properties of the relationship between the components $\Gamma \in \mathcal
C$ and the associated cylinders. These results follow from the construction of $\hat V$ in the
previous section.

\begin{lemma}
\label{le:one intersect}
Suppose that $\Gamma \in \mathcal C$ and $x, y \in \Gamma$. Let $v \in \mathcal V(\Gamma)$. Then there exists $p \in \partial
W^\circ(v)$ with $v \cdot p = 1$ such that $(x-y) \cdot p = 0$.
In particular, if $x \in L_v(\Gamma)$
then $L_v(x) \cap \Gamma = \set{y}$ for some $y$, that is, there exists a unique $t \in \R$ such
that $x - tv \in \Gamma$.
\end{lemma}

\begin{proof}
Since $\Gamma$ is a smooth curve, by the mean value theorem there exists $\xi \in \Gamma$ such that
$(x-y) \cdot \nu(\xi) = 0$. But $p:= \frac{\nu(\xi)}{W(\nu(\xi))} \in \partial W^\circ(v)$ by
construction. Then $v \cdot p = 1$ follows from the characterization of the subdifferential of
$W^\circ$ in Lemma~\ref{le:wulff-frank-rel}.

Now let $x \in L_v(\Gamma)$. By definition, there exists $t \in \R$ such that $x - tv \in \Gamma$.
Suppose that $x - sv \in \Gamma$ for $s \in \R$. Then from the above there exists $p$ such that
$v \cdot p = 1$ and $0 = (x - tv - x + sv) \cdot p = (s - t) v \cdot p = s - t$. We have $s = t$.
\end{proof}

\begin{lemma}
\label{le:neighbor cylinders}
Let $\Gamma \in \mathcal C_0$ and $\Gamma' \in \mathcal C_r$ such that $\dist(\Gamma, \Gamma') =
0$. Then there exist $v$ such that $\set{v} = \mathcal V(\Gamma) \cap \mathcal V(\Gamma')$, and $\xi$ such that $\cl \Gamma \cap \Gamma' = \set{\xi}$. Moreover,
$L_v(\Gamma) \cap L_v(\Gamma') = \emptyset$ and
$\cl{L_v(\Gamma)} \cap L_v(\Gamma') = L_v(\xi)$.
\end{lemma}

\begin{proof}
If $\dist(\Gamma, \Gamma') = 0$, then $\Gamma$ must be the flattened part and $\Gamma'$ must be the
adjacent smooth part of $\partial \hat V$.
By construction, $\mathcal V(\Gamma) \cap \mathcal V(\Gamma') = \set{v}$ for some $v \in \mathcal
V$, and $\cl \Gamma \cap \Gamma' = \set{\xi}$ for some $\xi$. In particular, $\cl{L_v(\Gamma)} \cap
L_v(\Gamma') \subset L_v(\xi)$.
Now suppose that there exist distinct points $x \in \cl\Gamma$, $y \in \Gamma'$ such that $L_v(x) = L_v(y)$.
Then by connectedness of $\cl\Gamma$ and $\Gamma'$, we can find such points arbitrarily close to
$\xi$. But this is a contradiction with the fact that $L_v(x)$ can intersect both $\Gamma$ and
$\Gamma'$ at most once by Lemma~\ref{le:one intersect}, and the line $L_v(x)$ in the direction of
$v$ travels from $\hat V$ to $\hat V^c$ at two consecutive points $x, y$, with no transition from $\hat
V^c$ to $\hat V$ in between.
\end{proof}

\begin{corollary}
\label{co:on intersect neighbor}
Suppose that $\Gamma \in \mathcal C_r$, $\Gamma', \Gamma'' \in \mathcal C_0$ are the adjacent flat
parts, $\dist(\Gamma', \Gamma) = \dist(\Gamma'', \Gamma) = 0$, and $x, y \in \Gamma \cup \Gamma'
\cup \Gamma''$. Let $v \in \mathcal V(\Gamma)$. Then there exists $p \in \partial
W^\circ(v)$ with $v \cdot p = 1$ such that $(x-y) \cdot p = 0$.
In particular, if $x \in L_v(\Gamma \cup \Gamma' \cup \Gamma'')$
then $L_v(x) \cap (\Gamma \cup \Gamma' \cup \Gamma'') = \set{y}$ for some $y$, that is, there exists a unique $t \in \R$ such
that $x - tv \in \Gamma \cup \Gamma' \cup \Gamma''$.
\end{corollary}

\begin{proof}
This follows by combining Lemma~\ref{le:neighbor cylinders} and Lemma~\ref{le:one intersect} for the neighboring components $\Gamma,
\Gamma', \Gamma''$, since the flat ones have normals $\nu', \nu'' \in \partial W^\circ(v)$.
\end{proof}

Given $\mu > 0$, we define the sets $U_\Gamma$ for $\Gamma \in \mathcal C$ by
\begin{align*}
U_\Gamma :=
\begin{cases}
\set{x + tv: x \in \Gamma, \ |t| \leq \mu,\ v \in \mathcal V(\Gamma)} & \text{if $\Gamma \in \mathcal
C_r$,}\\
\set{x: |\xi_v^\Gamma(x) + \xi_w^\Gamma(x) - \alpha^\Gamma| \leq \mu, \ \xi_v^\Gamma(x) \xi_w^\Gamma(x) > 0} &
\text{if $\Gamma \in \mathcal C_0$.}
\end{cases}
\end{align*}
We shall show below in \eqref{partialVcover} that $\set{U_\Gamma}_{\Gamma \in \mathcal C}$ cover a neighborhood of $\partial
\hat V$.
Note that if we take $\mu \leq |\alpha^\Gamma|/2$ we must have $\sign \xi_v^\Gamma(x) =
\sign \xi_w^\Gamma(x) = \sign \alpha^\Gamma$ on $U_\Gamma$ for $\Gamma \in \mathcal C_0$.

If we choose $\mu >0$ small enough, the sets $U_\Gamma$ are pair-wise disjoint.
\begin{lemma}
\label{le:UG disjoint}
Suppose that $0 < \mu < \min \set{\mu_1, \mu_2}$, where
\begin{align*}
\mu_1 := \frac{1}{3\max_{W^\circ(v)} |v|}
\min_{\substack{\Gamma, \Gamma' \in \mathcal C\\\dist(\Gamma, \Gamma')> 0
}}\dist(\Gamma, \Gamma'),\qquad
\mu_2 := \min_{\Gamma \in \mathcal C_0}
\frac{|\alpha^\Gamma|}2.
\end{align*}
Then $U_\Gamma \cap U_{\Gamma'} = \emptyset$ for all
$\Gamma, \Gamma' \in \mathcal C$, $\Gamma \neq \Gamma'$.
\end{lemma}

\begin{proof}
Suppose that $\dist(\Gamma, \Gamma') > 0$. Then $U_\Gamma \subset \nbd{t}(\Gamma)$ and
$U_{\Gamma'} \subset \nbd{t}(\Gamma')$ with $t = \mu \max_{W^\circ(v) \leq 1} |v|$.
Hence $U_\Gamma \cap U_{\Gamma'} = \emptyset$ by $\mu < \mu_1$.

On the other hand, if $\dist(\Gamma, \Gamma') = 0$, then one of the sets, say $\Gamma$, belongs to
$\mathcal C_0$, and the other belongs to $\mathcal C_r$.
Suppose that $y \in U_\Gamma \cap U_{\Gamma'}$. We will show that this leads to a contradiction.
Indeed, set $v \in \mathcal V(\Gamma')$ and note that $U_{\Gamma'} \subset L_v(\Gamma')$. We have
$c^\Gamma \in L_v(\Gamma')$. Therefore $y(\lambda) := \lambda y + (1-\lambda) c^\Gamma \in L_v(\Gamma')$ for
every $\lambda \in [0,1]$. By Lemma~\ref{le:neighbor cylinders}, we have $\Gamma \cap
L_v(y(\lambda)) =\emptyset$ for all $\lambda \in [0,1]$.

Let $t:= \xi_v^\Gamma(y) + \xi_w^\Gamma(y) - \alpha^\Gamma$. Since $y \in U_\Gamma$, we have $|t|
\leq \mu < \mu_2 \leq |\alpha^\Gamma|/2$ and $\xi_v^\Gamma(y) \xi_w^\Gamma(y) > 0$. If $t \alpha^\Gamma \leq 0$, we have $y - t
v \in \Gamma$ by \eqref{Gamma alpha level}, and
this is a contradiction with $\Gamma \cap L_v(y) = \emptyset$.
If $t \alpha^\Gamma > 0$, we set $\lambda := \frac{\alpha}{\alpha + t} \in (0,1)$. A simple
computation using \eqref{Gamma alpha level} shows that $y(\lambda) \in \Gamma$, which is a contradiction with $\Gamma \cap
L_v(y(\lambda)) = \emptyset$.
The conclusion $U_\Gamma \cap U_{\Gamma'} = \emptyset$ follows.
\end{proof}

We choose $\mu$ satisfying the assumption in Lemma~\ref{le:UG disjoint}.
Then on the pair-wise disjoint collection of sets $\set{U_\Gamma: \Gamma \in \mathcal C}$, we define functions $\psi$ and $z$ by
\begin{align*}
\psi(x) :=
\begin{cases}
t \text{ such that $x - tv \in \Gamma$, $v \in \mathcal V(\Gamma)$},& x \in U_\Gamma,\ \Gamma \in
\mathcal C_r,\\
\xi_v^\Gamma(x) + \xi_w^\Gamma(x) - \alpha^\Gamma, & x \in U_\Gamma,\ \Gamma \in \mathcal C_0,
\end{cases}
\end{align*}
and
\begin{align*}
z(x) :=
\begin{cases}
v \text{, where $v \in \mathcal V(\Gamma)$}& x \in U_\Gamma,\ \Gamma \in \mathcal C_r,\\
\frac{\xi_v^\Gamma(x) v + \xi_w^\Gamma(x) w}{\xi_v^\Gamma(x) + \xi_w^\Gamma(x)}, \text{ where $v, w
\in \mathcal V(\Gamma)$}, & x \in U_\Gamma,\ \Gamma \in \mathcal C_0.
\end{cases}
\end{align*}

Both $\psi$ and $z$ are well-defined by Lemma~\ref{le:one intersect}. Note that $|\psi| \leq
\mu$ on $U_\Gamma$.
We can easily see that $\psi$ is differentiable in the interior of $U_\Gamma$ for all $\Gamma \in
\mathcal C$ by the inverse function theorem. Moreover, the level set $\set{x: \psi(x) =
\psi(y)}$ in a neighborhood of $y \in \interior U_\Gamma$ is just a translation of $\Gamma$.
Therefore $\nabla \psi(y) = s \nu'$, where $\nu' = \nu^\Gamma$ for $\Gamma \in \mathcal C_0$, or
$\nu'
= \nu(y - \psi(y) v)$ for $\Gamma \in \mathcal C_r$, with $s = v \cdot \nu' > 0$. In particular,
\begin{align*}
z(y) \in \partial W(\nabla \psi(y)) \qquad \text{for $z \in \interior U_\Gamma$, $\Gamma \in
\mathcal C$.}
\end{align*}

We now conclude this part by showing that for small $\delta > 0$, the functions $\psi$ and $z$ are
well-defined, Lipschitz continuous functions on $\nbd\delta(\partial \hat V)$. We shall use the following
two lemmas that we prove first. We set
\begin{align*}
K := \max_{W(p)\leq 1} |p| \qquad \text{and} \qquad
\delta_\theta = \min_{\Gamma \in \mathcal C} \min_{v \in \mathcal V(\Gamma)} \theta(L_v(\Gamma)),
\qquad \text{and} \qquad \delta_\mu := \frac \mu K.
\end{align*}
Finally, we find $\delta_s > 0$ such that for every $\Gamma \in \mathcal C_0$, $\Gamma' \in
\mathcal C_r$ the adjacent component to $\Gamma$, $\dist(\Gamma', \Gamma) = 0$, $v \in \mathcal
V(\Gamma')$, $v \neq w \in \mathcal V(\Gamma)$, we have
\begin{align}
\label{delta sign}
\dist (\Gamma' \cap \nbd{\delta_s}(L_v(\Gamma)), L_w(c^\Gamma)) > \delta_s.
\end{align}
This is possible since $\Gamma' \cap \cl{L_v(\Gamma)} = \set{x_v}$, $\Gamma'$ is smooth (in fact
detaching from $L_v(\Gamma)$ linearly), with $\xi_v^\Gamma(x_v) =
\alpha^\Gamma \neq 0$, and $L_w(c^\Gamma) = \set{x: \xi_v^\Gamma(x) = 0}$.

\begin{lemma}
\label{le:Cr case}
Let $x \in \Gamma \in \mathcal C_r$, $v \in \mathcal V(\Gamma)$, and let $\Gamma', \Gamma'' \in
\mathcal C_0$ be the
neighboring components with $\dist(\Gamma', \Gamma) = \dist(\Gamma'', \Gamma) = 0$. Then for every
$y$, $|y - x| \leq \min\set{\delta_\theta, \delta_\mu, \delta_s}$ there exists a unique $t(y)$ such that $y - t(y) v \in \Gamma
\cup \Gamma' \cup \Gamma''$. Moreover, $|t(y)| \leq \mu$.
Finally, $\sign\xi_v^{\Gamma'}(y) = \sign \alpha^{\Gamma'}$ and $\sign \xi_v^{\Gamma''}(y) = \sign
\alpha^{\Gamma''}$.
\end{lemma}

\begin{proof}
Using Lemma~\ref{le:neighbor cylinders} and $|x - y| \leq \delta_\theta$, that is, that the
distance between $x$ and $y$ is smaller than the width of the cylinders $L_v(\Gamma')$ and $L_v(\Gamma'')$, we
are guaranteed that $y \in L_v(\Gamma \cup \Gamma' \cup \Gamma'')$. Therefore there exists a unique
$t \in \R$ with $y - tv \in \Gamma \cup \Gamma' \cup \Gamma''$.
By Corollary~\ref{co:on intersect neighbor}, there exists $p \in \partial W^\circ(v)$ such that
\begin{align*}
0 = (x - y + tv) \cdot p = (x - y) \cdot p + t.
\end{align*}
By Cauchy-Schwarz $|t| \leq K |x - y| \leq \mu$ and the conclusion follows.
The sign of $\xi_v^{\Gamma'}(y)$ and $\xi_v^{\Gamma''}(y)$ must match the sign at $\Gamma \cap
\cl\Gamma'$, $\Gamma \cap \Gamma''$, which matches that of
$\alpha^{\Gamma'}$, $\alpha^{\Gamma''}$, respectively, since $\delta \leq \delta_s$ and $\delta_s$ satisfies
\eqref{delta sign}.
\end{proof}

\begin{lemma}
\label{le:C0 case}
Let $x \in \Gamma \in \mathcal C_0$, $v \in \mathcal V(\Gamma)$, and let $\Gamma', \Gamma'' \in
\mathcal C_r$ be the
neighboring components with $\dist(\Gamma', \Gamma) = \dist(\Gamma'', \Gamma) = 0$. Let $v \in
\mathcal V(\Gamma')$ and $w \in \mathcal V(\Gamma'')$. Then for every
$y$, $|y - x| \leq \min\set{\delta_\theta, \delta_\mu}$ where $\delta_\mu = \mu/K$. Then exactly
one of the following holds:
\begin{enumerate}
\item $\xi_v^\Gamma(y) \xi_w^\Gamma(y) > 0$, $|\xi_v^\Gamma(y) + \xi_w^\Gamma(y) - \alpha| \leq \mu$, or
\item $y \in L_v(\Gamma')$, there exists $t$ such that $y - t v \in \Gamma'$, and $|t| \leq \mu
$, or
\item $y \in L_w(\Gamma'')$, there exists $t$ such that $y - t w \in \Gamma''$, and $|t| \leq \mu$.
\end{enumerate}
\end{lemma}

\begin{proof}
Let us set $t = \xi_v^\Gamma(y) + \xi_w^\Gamma(y) - \alpha$. Then $\xi_v^\Gamma(y - tv) +
\xi_w^\Gamma(y - tv) - \alpha^\Gamma = 0$ and therefore $(x - y + t v) \cdot \nu_0 =0$, where $\nu_0$ is
the normal of $\Gamma$. In particular, $t = (x - y) \cdot \frac{\nu_0}{W(\nu_0)}$ and hence
$|t| \leq K |x - y| \leq \mu$, which implies the estimate in (a).

Since $|y - x| \leq \delta_\mu$, we have $|t| \leq \mu < \mu_2 \leq |\alpha^\Gamma|/2$ and
therefore $\xi_v^\Gamma(y) + \xi_w^\Gamma(y)$ has the same sign as $\alpha^\Gamma$. We conclude
that at least one of $\xi_v^\Gamma(y)$, $\xi_v^\Gamma(y)$ has the same sign as $\alpha^\Gamma$.
Due to Lemma~\ref{le:neighbor cylinders},
$L_v(\Gamma') \subset \set{\xi_w^\Gamma \alpha^\Gamma \leq 0}$ and $L_w(\Gamma'') \subset \set{\xi_v^\Gamma \alpha^\Gamma \leq 0}$.
Therefore $y \notin L_v(\Gamma') \cap L_w(\Gamma'')$.

If $\xi_v^\Gamma(y) \xi_w^\Gamma(y) > 0$ then we are at case (a).
Otherwise since $|y - x| \leq \delta_\theta$, $y$ must be in exactly one of the cylinders
$L_v(\Gamma')$ or $L_w(\Gamma'')$ due to the discussion above.

Suppose therefore $y \in
L_v(\Gamma')$. Then there exists a unique $t$ such that $y - t v \in \Gamma'$, and Corollary~\ref{co:on intersect neighbor} implies the estimate $|t| \leq K |y -
x| \leq \mu$ as in Lemma~\ref{le:Cr case}.
The case $y \in L_w(\Gamma'')$ can be handled similarly.
\end{proof}

We therefore take
\begin{align*}
0 < \delta < \min\set{\delta_\theta, \delta_\mu, \delta_s}.
\end{align*}
With this choice,
\begin{align}
\label{partialVcover}
\nbd\delta(\partial \hat V) \subset \bigcup_{\Gamma \in \mathcal C} U_\Gamma.
\end{align}
Indeed,
let us fix $y \in \nbd\delta(\partial \hat V)$. Then there exists $x \in \partial \hat V$ with
$|x - y| \leq \delta$.

In the case that $x \in \Gamma \in \mathcal C_r$, we apply Lemma~\ref{le:Cr case} to conclude that there is a
unique $t$, $|t| \leq K |y-x| \leq K\delta \leq \mu$, such that $y - t v \in \Gamma \cup \Gamma'
\cup \Gamma''$ where $v \in \mathcal V(\Gamma)$.
If $y - tv \in \Gamma$, then clearly $y \in U_\Gamma$.
On the other hand, if $y - tv \in \Gamma'$, we have
\begin{align*}
0 = \xi_v^{\Gamma '}(y - tv) + \xi_w^{\Gamma '}(y - tv) - \alpha^{\Gamma'} =
\xi_v^{\Gamma '}(y) + \xi_w^{\Gamma '}(y) - \alpha^{\Gamma'} - t.
\end{align*}
Since also $\sign \xi_v^{\Gamma'}(y) = \sign \alpha^{\Gamma'}$, we conclude that $y \in
U_{\Gamma'}$.
An analogous argument works if $y - tv \in \Gamma''$.

Now if $x \in \Gamma \in \mathcal C_0$, we apply Lemma~\ref{le:C0 case}, and we argue as above to conclude
that $y \in U_{\Gamma}$, $U_{\Gamma'}$ or $U_{\Gamma''}$. Therefore we recover
\eqref{partialVcover}.

\medskip

Now we finally show that $\psi$ and $z$ are Lipschitz on $\nbd\delta(\partial \hat V)$. Since $\psi$ and $z$ are smooth in the interior of $U_\Gamma$, we only need to address the continuity
across the transition between $U_\Gamma$, $U_{\Gamma'}$, $\Gamma \in \mathcal C_0$, $\Gamma' \in
\mathcal C_r$, with $\dist(\Gamma, \Gamma') = 0$.
The function $z$ is clearly Lipschitz across this boundary, since we can alternatively define $z$ in the
neighborhood of this boundary using
\begin{align*}
\zeta(x):=
\begin{cases}
\xi_w^\Gamma(x), & \xi_w^\Gamma(x) \alpha^\Gamma > 0,\\
0, & \text{otherwise}.
\end{cases}
\end{align*}
Then we have in the neighborhood of the boundary between $U_\Gamma$ and $U_{\Gamma'}$ that
\begin{align*}
z(x) = \frac{\xi_v^\Gamma(x) v + \zeta(x) w}{\xi_v^\Gamma(x) + \zeta(x)},
\end{align*}
which is clearly a Lipschitz function when $\abs{\xi_v^\Gamma(x)} > \e > 0$, as is the case near
the boundary.

Similarly, we can alternatively define $\psi$ in the neighborhood of the boundary between
$U_\Gamma$ and $U_{\Gamma'}$ as
\begin{align*}
\psi(x) = t \quad \text{where $t$ is such that $x - t v \in \Gamma \cup \Gamma'$.}
\end{align*}
This function is Lipschitz continuous by Corollary~\ref{co:on intersect neighbor}.

\subsubsection{Completion of the proof of Proposition~\ref{pr:approximate-pair}}

We now have two Lipschitz functions $\psi^-, \psi^+$ and Lipschitz continuous vector fields $z^-,
z^+$ defined in $\nbd{\delta}(\partial G_-)$ and $\nbd{\delta}(\partial G_+)$, respectively, such
that $z^\pm(x) \in \partial W(\nabla \psi^\pm)$. Furthermore, $\partial G_\pm = \set{\psi^\pm =
0}$ almost everywhere.
We now have to connect them to produce an admissible support function of the pair $(G_-,
G_+)$.
We define the constant $\eta = \min (\eta^-, \eta^+) > 0$ by
\begin{align*}
\eta^\pm := \frac 12 \min \set{|\psi_\pm(x)|: \delta/2 \leq \dist(x, \partial G_\pm) \leq \delta}.
\end{align*}
We find smooth cutoff functions $\varphi^\pm \in C^\infty_c$ such that
\begin{align*}
\text{
$0 \leq \varphi^\pm \leq 1$,
$\supp \varphi^\pm \subset \nbd{3\delta/4}(\partial G_\pm)$, $\varphi_\pm = 1$ on
$\nbd{\delta/2}(\partial G_\pm)$.
}
\end{align*}
We define the support function of the pair $(G_-, G_+)$ as
\begin{align*}
\psi(x) :=
\begin{cases}
\eta & x\in G_+ \setminus \nbd\delta(\partial G_+),\\
\min(\eta, \max (\psi^+, 0)) & x\in \nbd\delta(\partial G_+),\\
0 & x \in G_-^c \cap G_+^c \setminus \nbd\delta(\partial G_- \cup \partial
G_+),\\
\max(-\eta, \min (\psi^-, 0)) & x\in \nbd\delta(\partial G_-),\\
-\eta & x\in G_- \setminus \nbd\delta(\partial G_-).
\end{cases}
\end{align*}
It is easy to check that $\psi$ is a Lipschitz support function of $(G_-, G_+)$.
Moreover, it is admissible with the Lipschitz Cahn-Hoffman vector field
\begin{align*}
z(x) := z^-(x) \varphi^-(x) + z^+(x) \varphi^+(x).
\end{align*}
by Lemma~\ref{le:one-homogeneous-subdiff}.

\section{Comparison principle}
\label{sec:comparison principle}

In this section we prove the comparison principle
on a spacetime cylinder $Q := \Rn \times (0,T)$
for some $T > 0$.

\begin{theorem}[Comparison principle]
\label{th:comparison principle}
Let $W: \R^n \to \R$ be a positively one-homogeneous convex polyhedral function such that the
conclusion of
Corollary~\ref{co:approximate pair sliced} holds for $1 \leq k \leq n-1$, and let $F$ be of curvature-free type at
$p_0 = 0$.
Suppose that $u$ and $v$ are a subsolution and a supersolution of
\eqref{P} on $\Rn \times [0, T]$ for some $T > 0$, respectively.
Moreover, suppose that there exist a compact set $K\subset \Rn$ and constants $c_u \leq c_v$
such that
$u \equiv c_u$, $v \equiv c_v$ on $\pth{\Rn \setminus K} \times [0,T]$.
Then $u(\cdot, 0) \leq v(\cdot, 0)$ on $\Rn$
implies $u \leq v$ on $\Rn \times [0, T]$.
\end{theorem}

We will use the standard doubling-of-variables technique with an additional parameter to enforce a
certain facet-like behavior of the functions at a contact point, which will allow us to construct
faceted test functions there.

Let us suppose that the comparison theorem does not hold
for a given subsolution $u$ and supersolution $v$, that is, suppose that
\begin{align}
\label{m_0}
m_0 := \sup_Q [u - v] > 0.
\end{align}

For arbitrary $\zeta \in \Rn$, $\e > 0$ we define
\begin{align*}
\Phi_{\zeta,\e}(x,t,y,s) := u(x,t) - v(y,s) - \frac{\abs{x-y-\zeta}^2}{2\e}-
    S_\e(t,s),
\end{align*}
where
\begin{align}
\label{Se}
S_\e(t,s) := \frac{\abs{t-s}^2}{2\e} + \frac{\e}{T - t}
+\frac\e{T-s}.
\end{align}

As in \cite{GG98ARMA}, we define the maximum of $\Phi_{\zeta,\e}$ as
\begin{align*}
\ell(\zeta, \e) = \max_{\cl Q \times \cl Q} \Phi_{\zeta,\e}
\end{align*}
and
the set of maxima of $\Phi_{\zeta,\e}$
over $\cl Q \times \cl Q$
\begin{align*}
\mathcal A(\zeta,\e) := \argmax_{\cl Q \times \cl Q} \Phi_{\zeta,\e}
    := \set{(x,t,y,s) \in \cl Q \times \cl Q :
    \Phi_{\zeta,\e}(x,t,y,s) = \ell(\zeta,\e)}.
\end{align*}
Moreover, we define
the set of gradients
\begin{align*}
\mathcal B(\zeta,\e) :=
    \set{\frac{x - y -\zeta}{\e} : (x,t,y,s) \in \mathcal A(\zeta,\e)}.
\end{align*}

\begin{proposition}[{cf. \cite{GG98ARMA}}]
\label{pr:maxima-interior}
There exists $\e_0 > 0$ such that
for all $\e \in (0, \e_0)$
we have
\begin{align*}
\mathcal A(\zeta,\e)\subset Q \times Q
\qquad \text{for all $\abs{\zeta} \leq \kappa(\e)$},
\end{align*}
where $\kappa(\e) := \frac 18 (m_0 \e)^{\frac 12}$.
\end{proposition}

From now on, we \textbf{fix} $\e \in (0, \e_0)$ so that Proposition~\ref{pr:maxima-interior} holds and we
write $\kappa = \kappa(\e)$ for simplicity, and drop $\e$ from our notation.

We have the following properties of $\mathcal A(\zeta)$
and $\mathcal B(\zeta)$.

\begin{proposition}
\label{pr:max-graph}
The graphs of $\mathcal A(\zeta)$ and $\mathcal B(\zeta)$
over $\zeta \in \cl B_\kappa(0)$ are compact.
\end{proposition}

\begin{proof}
See \cite[Proposition 7.3]{GG98ARMA}.
Since $\Phi_\zeta - \ell(\zeta) \leq 0$
by definition of $\ell$, we observe that
\begin{align*}
\graph \mathcal A(\zeta) &:=
\{(\zeta, x, t, y, s) \subset \cl B_\kappa(0) \times \cl Q
\times \cl Q:\\
&\qquad\qquad\Phi_\zeta(x,t,y,s) - \ell(\zeta) \geq 0\},
\end{align*}
which is closed since $\Phi_\zeta$
is an upper semi-continuous function of $(\zeta, x, t, y, s)$
and $\ell(\zeta)$ is a lower semi-continuous function.
$\graph \mathcal B(\zeta)$ is a continuous image of
$\graph \mathcal A(\zeta)$ and therefore also compact.
\end{proof}

\begin{proposition}
\label{pr:ball-of-gradients}
With $\kappa = \kappa(\e)$ fixed above, there exists a maximal relatively open convex set $\Xi
\subset \Rn$ on which $\partial W$ is constant, $\zeta_0 \in \Rn$ and $\lambda > 0$
such that $\abs{\zeta_0} + 2\lambda < \kappa$ and
\begin{align*}
\mathcal B(\zeta) \cap \Xi \neq \emptyset \qquad \text{for all } \zeta \in B_{2\lambda}(\zeta_0).
\end{align*}
Moreover, $\aff \Xi \perp \aff \partial W(p)$ for all $p \in \Xi$.

In other words,
for every $\zeta \in B_{2\lambda}(\zeta_0)$ there exists
a point of maximum $(\hat x, \hat t, \hat y, \hat s)
\in \mathcal A(\zeta)$ of $\Phi_\zeta$
such that
\begin{align}
\label{grad-in-Xi}
\frac{\hat x - \hat y - \zeta}{\e} \in \Xi.
\end{align}
\end{proposition}

\begin{proof}
Recall the decomposition of $\Rn$ from Proposition~\ref{pr:feature-decomposition} into relatively open convex sets $\Xi_i$, $i \in \mathcal N$.
Moreover, since $\Xi_i$ is relatively open
we can find an increasing sequence of compact sets $K_{i,j} \subset \Xi_i$
such that
\begin{align*}
\Xi_i = \bigcup_{j\in\N} K_{i,j}.
\end{align*}
Let us now define the sets
\begin{align*}
Z_{i,j} := \set{\zeta \in \cl B_\kappa(0) : K_{i,j} \cap \mathcal B(\zeta)
\neq \emptyset}.
\end{align*}
We observe that $Z_{i,j}$ are compact due to Proposition~\ref{pr:max-graph}.
Since
\begin{align*}
\cl B_\kappa(0) = \bigcup_{i\in\mathcal N} \bigcup_{j\in\N} Z_{i,j},
\end{align*}
the Baire category theorem implies that
there exists $i_0 \in \mathcal N$, $j_0 \in \N$
such that $\interior Z_{i_0, j_0} \neq \emptyset$.
In particular, we can find $\zeta_0$ and $\lambda > 0$
with $B_{2\lambda}(\zeta_0) \subset Z_{i_0,j_0}$, and we set $\Xi = \Xi_{i_0}$.
Note that $\Xi$ is maximal by Proposition~\ref{pr:feature-decomposition}.
\end{proof}

\subsection{Flatness at a contact point}

We will now use the information about the behavior of $u$ and $v$
at the point of maximum to show that there is enough space to construct
faceted test functions.
We shall use the Constancy lemma from \cite{GG98ARMA}.

\begin{lemma}[Constancy lemma]
\label{le:constancy}
Let $1\leq k < N$, $K \subset \R^N$ be compact and $G \subset \R^k$
be a bounded domain.
Denote $P: \R^N \to \R^k$ the natural projection
$w \mapsto (w_1, \ldots, w_k)$.
Assume that $h$ is an upper semi-continuous function and $\phi\in C^2(\R^k)$,
and define for $w \in K$ and $z \in G$
\begin{align*}
h_z(w) &:= h(w) - \phi(Pw - z),\\
H(z) &:= \max_K h_z.
\end{align*}
If for all $z\in G$ there exists $w \in K$ such that
$h_z(w) = H(z)$ and $(\nabla \phi)(P w - z) = 0$
then $H(z)$ is constant on $G$.
\end{lemma}

In what follows, we will decompose $\R^n$ into two orthogonal subspaces $V$ and $U$, as in
Section~\ref{sec:slicing of W}, of dimensions $k = \dim V$ and $n - k = \dim U$.  Therefore we will
use $\TT$, $\TT_U$, $\TT_V$, and the decomposition $x = \TT(x', x'')$ as introduced in
\eqref{rotation}, \eqref{rotation'} and \eqref{rotationUV}.
\begin{lemma}
\label{le:max-constancy}
Suppose that there exist $p_0, \zeta_0 \in \Rn$, a subspace $U \subset \Rn$ and $\lambda > 0$
such that $\abs{\zeta_0} + 2\lambda < \kappa$
and for all $\zeta \in B_{2\lambda}(\zeta_0)$
we have
\begin{align*}
\mathcal B(\zeta) \cap (p_0 + U) \neq \emptyset.
\end{align*}
Then
\begin{align*}
\ell(\zeta) - p_0 \cdot \zeta = const
\quad \text{for } \zeta \in (\zeta_0 + V) \cap B_{2\lambda}(\zeta_0),
\end{align*}
where $V := U^\perp$.
\end{lemma}

\begin{proof}
We apply the constancy lemma, Lemma~\ref{le:constancy}.  Set $k = \dim V$ and $N = 2 (n+1)$.  We
will denote $w = (\xi', \xi'', t, y, s)$ for $\xi, y \in \Rn$, $t, s \in \R$, so that the
natural projection $P:\R^N \to \R^k$ is given as $Pw = \xi'$ (for the notation $'$ and $''$ see
\eqref{rotation'}).  Additionally, let us set
\begin{align*}
K = \set{(\xi', \xi'', t, y, s) : (\xi + y,t) \in \cl Q,
\ (y,s) \in \cl Q}
\quad \text{and} \quad G = B^d_{2\lambda}(0).
\end{align*}
And, finally, let us define the functions
\begin{align*}
h(w) &= u(\xi + y) - v(y) - p_0' \cdot \xi'
- \frac{\abs{\rho'' -  \zeta_0''}^2}{2\e} - S(t,s), &&w \in \R^N,\\
\phi(\eta) &= \frac{\abs{\eta - \zeta_0'}^2}{2\e} - p_0' \cdot \eta, \qquad &&\eta \in \R^k.
\end{align*}
The Pythagorean theorem $|\xi' - \zeta_0' -z|^2 + |\xi'' - \zeta_0''|^2 = |\xi - \zeta_0 - \TT_V
z|^2$, due to the fact that $\TT$ is a rotation, yields
\begin{align*}
h_z(w) := h(w) - \phi(\xi' - z) = \Phi_{\zeta_0 + \TT_V z} (\xi + y, y) - p_0' \cdot z.
\end{align*}
We know, by assumption, that for every $\zeta = \zeta_0 + \TT_V z$,
$z \in G$,
there exists a point of maximum $(x,t,y,s) \in \mathcal A(\zeta)$ of $\Phi_{\zeta}$
such that
$\frac{x - y -\zeta}{\e} \in p_0 + U$. This yields $\frac{x' - y' - z -\zeta_0'}{\e} = p_0'$. In
particular,
\begin{align*}
(\nabla \phi)(x' - y' - z) = 0.
\end{align*}
Thus, by Lemma~\ref{le:constancy}, we infer that
\begin{align*}
H(z) = \max_K h_z = \ell(\zeta_0 + \TT_V z) - p_0' \cdot z
\end{align*}
is constant for $z \in G$, which is what we wanted to prove
since $p_0 \cdot \zeta = p_0' \cdot z + p_0 \cdot \zeta_0$.
\end{proof}

The previous lemma has the following important corollary.

\begin{corollary}
\label{co:contact-ordering}
Suppose that we have $p_0$, $\zeta_0$, $\lambda$, $U$ and $V$
as in Lemma~\ref{le:max-constancy}.
Define
\begin{align*}
\ta(x,t,y,s) := u(x,t) - v(y,s) - \frac{\abs{x'' - y''- \zeta_0''}^2}{2\e}
- p_0' \cdot (x' - y'-\zeta_0') - S(t,s).
\end{align*}
Then for any $(\hat x,\hat t,\hat y,\hat s) \in \mathcal A(\zeta_0)$
such that $\frac{\hat x' - \hat y' - \zeta_0'}{\e} = p_0'$
we have
\begin{align*}
\ta(x,t,y,s) \leq \ta(\hat x, \hat t, \hat y,\hat s) \quad \text{for } (x,t),(y,s) \in \cl Q,
\ \abs{x' - y' - (\hat x' - \hat y')} \leq \lambda.
\end{align*}
\end{corollary}

\begin{proof}
For the sake of clarity, we will drop $t,s, \hat t$ and $\hat s$ from the following formulas.
Let us fix $x, y, \hat x, \hat y$ that satisfy the assumptions
and set
\begin{align}
\label{choice-of-zeta}
\zeta = \zeta_0 + \TT_V(x' - y' - (\hat x' - \hat y')).
\end{align}
Since $|\zeta - \zeta_0| \leq \lambda$ and $\zeta \in \zeta_0 + V$, Lemma~\ref{le:max-constancy}
implies $\ell(\zeta) - p_0 \cdot \zeta = \ell(\zeta_0) - p_0 \cdot \zeta_0$
and we infer from the definition of $\ell$
\begin{align}
\label{Phi-estim}
\Phi_\zeta(x,y)  \leq
\ell(\zeta) =
\ell(\zeta_0) + p_0 \cdot (\zeta - \zeta_0) = \Phi_{\zeta_0} (\hat x, \hat y) + p_0 \cdot (\zeta - \zeta_0).
\end{align}
Note also that $p_0 \cdot (\zeta - \zeta_0) = p_0' \cdot (\zeta' - \zeta_0')$.
We express $\ta$ in terms of $\Phi_\zeta$ and use \eqref{Phi-estim}
to obtain
\begin{align*}
\ta(x,y) &= \Phi_\zeta(x,y) + \frac{\abs{x'-y'-\zeta'}^2}{2\e}
- p_0' \cdot (x' - y' - \zeta_0')\\
&\leq \Phi_{\zeta_0}(\hat x,\hat y) + \frac{\abs{x'-y'-\zeta')}^2}{2\e}
+ p_0' \cdot (\zeta' - \zeta_0' - (x' - y' - \zeta_0'))\\
&= \ta(\hat x, \hat y) + \frac{\abs{x'-y'-\zeta'}^2}{2\e}
-\frac{\abs{\hat x'-\hat y'-\zeta_0'}^2}{2\e}\\
&\quad + p_0' \cdot (-(x' - y' - \zeta') + (\hat x' - \hat y' - \zeta_0'))\\
&= \ta(\hat x, \hat y) + \frac{\abs w^2}{2\e} - \frac{\abs z^2}{2\e}
+ p_0' \cdot (-w + z),
\end{align*}
where we set $w = x'-y'-\zeta'$ and
$z = \hat x' - \hat y' -\zeta_0'$.
We now just have to show that the extra terms cancel out.
First, we see that $w - z = 0$ by \eqref{choice-of-zeta}.
Furthermore, by the choice of $\hat x,\hat y$
we have $z/\e = p_0'$. Therefore we have,
using $\abs{w -z}^2 = \abs w^2 + \abs z^2 - 2w \cdot z$,
\begin{align*}
\frac{\abs{w}^2}{2\e} - \frac{\abs{z}^2}{2\e} =
\frac{\abs{w - z}^2}{2\e} - \frac{\abs{z}^2}{\e} + \frac{w \cdot z}{\e}
= 0 -p_0' \cdot z + p_0' \cdot w.
\end{align*}
Therefore $\ta(x,t,y,s) \leq \ta(\hat x, \hat t, \hat y, \hat s)$, which is what we wanted to
prove.
\end{proof}

\subsection{Construction of faceted test functions}

We shall use Corollary~\ref{co:contact-ordering} to construct
test functions for $u$ and $v$ following the idea
from \cite{GGP13AMSA}.

Let us fix $\Xi, \zeta_0 \in \Rn$ and $\lambda>0$ to be a triplet provided by
Proposition~\ref{pr:ball-of-gradients}.  Then we fix a point of maximum $(\hat x, \hat t, \hat y,
\hat s) \in \mathcal A(\zeta_0)$ that satisfies \eqref{grad-in-Xi} with $\zeta = \zeta_0$.  We set
$p_0 := \frac{\hat x - \hat y - \zeta_0}\e \in \Xi$, $U := \Span(\Xi - p_0)$ and $V \subset \Rn$ be
the subspace parallel to $\aff
\partial W(p_0)$. We have $U = V^\perp$ by Proposition~\ref{pr:feature-decomposition}. It is easy to check that $p_0$, $\zeta_0$, $\lambda$, $U$, $V$ and
$(\hat x, \hat t, \hat y, \hat s)$  satisfy the hypothesis of Corollary~\ref{co:contact-ordering}.
Let us also set $k = \dim V$ as usual.

Depending on the value $k$ and $p_0$, we split the
situation into three cases:

\begin{itemize}
\item[Case I]: $k = 0$;
\item[Case II]: $k = n$, $p_0 = 0$ and $F$ is of curvature-free type at $p_0 = 0$;
\item[Case III]: none of the above.
\end{itemize}

We will deal with each case individually in the following three subsections and show that they all
lead to a contradiction. Therefore \eqref{m_0} cannot occur and the comparison principle holds.

\subsubsection{Case I}
We have $k = 0$.  In this case we use the off-facet test in Definition~\ref{def:visc-solution}(ii).
Corollary~\ref{co:contact-ordering} in this case reduces to
\begin{align}
\label{off-facet max order}
u(x, t) - v(y, s) - \frac{|x - y - \zeta_0|^2}{2\e} - S(t,s)
\leq
u(\hat x, \hat t) - v(\hat y, \hat s) - \frac{|\hat x - \hat y - \zeta_0|^2}{2\e} - S(\hat t,\hat s)
\end{align}
for all $(x,t), (y,s) \in \cl Q$.
We define the test functions
\begin{align*}
\varphi_u(x,t) &:= \frac{|x - \hat y - \zeta_0|^2}{2\e} + S(t,\hat s),\\
\varphi_v(x,t) &:= - \frac{|\hat x - x - \zeta_0|^2}{2\e} - S(\hat t, t).
\end{align*}
From \eqref{off-facet max order} we deduce that $u - \varphi_u$ has a global maximum at $(\hat x,
\hat t)$ and $v - \varphi_v$ has a global minimum at $(\hat y, \hat s)$.
Therefore we must have from the definition of viscosity solutions
\begin{align*}
(\varphi_u)_t(\hat x, \hat t) + F(\nabla \varphi_u(\hat x, \hat t), 0) &\leq 0,\\
(\varphi_v)_t(\hat y, \hat s) + F(\nabla \varphi_v(\hat y, \hat s), 0) &\geq 0.\\
\end{align*}
Since $\nabla \varphi_u(\hat x, \hat t) = \nabla \varphi_v(\hat y, \hat s)$, subtracting the
second inequality from the first and evaluating the time derivatives yields
\begin{align*}
0 \geq (\varphi_u)_t(\hat x, \hat t) - (\varphi_v)_t(\hat y, \hat s) = \frac{\e}{(T - \hat t)^2} + \frac{\e}{(T - \hat s)^2} > 0,
\end{align*}
a contradiction.

\subsubsection{Case II}

Now $k = n$, or, in other words, $V = \aff \partial W(p_0) = \Rn$. Since we now assume that $p_0 =
0$ and that $F$ is of curvature-free type at $p_0 = 0$, we use Definition~\ref{def:level-set-test}.
Then this case is just a minor modification of Case I. Indeed, Corollary~\ref{co:contact-ordering}
now reads
\begin{align}
\label{curvature-free order}
u(x, t) - v(y, s) - S(t,s)
\leq
u(\hat x, \hat t) - v(\hat y, \hat s) - S(\hat t,\hat s)
\end{align}
for all $(x, t), (y, t) \in \cl Q$, $|x - y - (\hat x - \hat y)| \leq \lambda$. Thus if we define
the test functions
\begin{align*}
\varphi_u(x,t) := S(t, \hat s), \quad \text{and} \quad \varphi_v(x,t):= - S(\hat t, t),
\end{align*}
we see from \eqref{curvature-free order} that $u - \varphi_u$ has a local maximum at $(\hat x, \hat
t)$, and $v - \varphi_v$ has a local minimum at $(\hat y, \hat s)$. The definition of viscosity
solution for the curvature-free type case yields
\begin{align*}
(\varphi_u)_t(\hat x, \hat t) + F(0, 0) &\leq 0,\\
(\varphi_v)_t(\hat y, \hat s) + F(0, 0) &\geq 0.\\
\end{align*}
The contradiction then follows as in Case I.

\subsubsection{Case III}

\begin{figure}[t]
\centering
\def\svgwidth{4.5in}
\begingroup%
  \makeatletter%
  \providecommand\color[2][]{%
    \errmessage{(Inkscape) Color is used for the text in Inkscape, but the package 'color.sty' is not loaded}%
    \renewcommand\color[2][]{}%
  }%
  \providecommand\transparent[1]{%
    \errmessage{(Inkscape) Transparency is used (non-zero) for the text in Inkscape, but the package 'transparent.sty' is not loaded}%
    \renewcommand\transparent[1]{}%
  }%
  \providecommand\rotatebox[2]{#2}%
  \ifx\svgwidth\undefined%
    \setlength{\unitlength}{324bp}%
    \ifx\svgscale\undefined%
      \relax%
    \else%
      \setlength{\unitlength}{\unitlength * \real{\svgscale}}%
    \fi%
  \else%
    \setlength{\unitlength}{\svgwidth}%
  \fi%
  \global\let\svgwidth\undefined%
  \global\let\svgscale\undefined%
  \makeatother%
  \begin{picture}(1,0.44444444)%
    \put(0,0){\includegraphics[width=\unitlength,page=1]{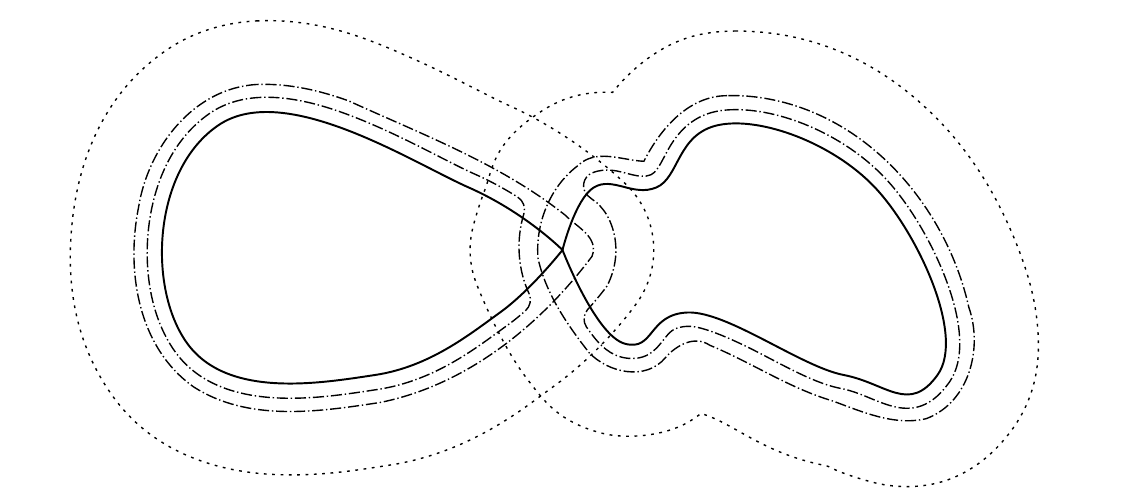}}%
    \put(0.17094889,0.27628104){\color[rgb]{0,0,0}\makebox(0,0)[lb]{\smash{$\hat U$}}}%
    \put(0.7217842,0.26074023){\color[rgb]{0,0,0}\makebox(0,0)[lb]{\smash{$\hat V$}}}%
    \put(0.77445031,0.39370048){\color[rgb]{0,0,0}\makebox(0,0)[lb]{\smash{$\nbd\lambda(\hat V)$}}}%
    \put(0.15775622,0.38050781){\color[rgb]{0,0,0}\makebox(0,0)[lb]{\smash{$\nbd\lambda(\hat U)$}}}%
    \put(0.7217842,0.04082076){\color[rgb]{0,0,0}\makebox(0,0)[lb]{\smash{$(V_-, V_+)$}}}%
    \put(0.18735304,0.05007592){\color[rgb]{0,0,0}\makebox(0,0)[lb]{\smash{$(U_-, U_+)$}}}%
    \put(0.49989601,0.38765683){\color[rgb]{0,0,0}\makebox(0,0)[lb]{\smash{$0$}}}%
    \put(0,0){\includegraphics[width=\unitlength,page=2]{facet-construction.pdf}}%
  \end{picture}%
\endgroup%
\caption{Settings for the facet construction. The dot-dashed lines represent the boundaries of the
constructed pairs $(U_-,U_+)$ and $(V_-, V_+)$.}
\label{fig:facet-construction}
\end{figure}

This is the most involved situation.
Since $W$ is positively one-homogeneous, we have $p_0 \perp V$ by Lemma~\ref{le:aff Xi origin} and the orthogonality from
Proposition~\ref{pr:feature-decomposition}, and therefore $p_0' = 0$ in what follows.
Nevertheless, we keep the terms with $p_0'$ below for completeness, they are necessary when
handling a case of general polyhedral $W$.
We first reduce the problem to the subspace $V$ by introducing
the functions
\begin{align*}
\begin{aligned}
\hat u(w) &:= u(\TT_V w + \hat x, \hat t) - p_0' \cdot w - u(\hat x, \hat t),\\
\hat v(w) &:= v(\TT_V w + \hat y, \hat s) - p_0' \cdot w - v(\hat y, \hat s),
\end{aligned}
&&& w \in \R^k.
\end{align*}
Then we build facets on $\R^k$ using the closed sets
\begin{align*}
\hat U := \set{w \in \R^k : \hat u(w) \geq 0}, &&&
\hat V := \set{w \in \R^k : \hat v(w) \leq 0}.
\end{align*}
as in \cite{GGP13AMSA}; see Figure~\ref{fig:facet-construction}. Note that these sets were
denoted there as $U$ and $V$. This allows us to create test functions for both subsolution and
supersolution and arrive at a contradiction as before.

Let us review the construction. For convenience we set
\begin{align*}
\xi_u(x'', t) &:= \frac{|x'' - \hat y'' - \zeta_0''|^2}{2\e}
- \frac{|\hat x'' - \hat y'' - \zeta_0''|^2}{2\e} + S(t, \hat s) - S(\hat t, \hat s),\\
\xi_v(y'', s) &:= \frac{|\hat x'' - \hat y'' - \zeta_0''|^2}{2\e}
- \frac{|\hat x'' - y'' - \zeta_0''|^2}{2\e} + S(\hat t, \hat s) - S(\hat t, s).
\end{align*}
Then
\begin{align*}
u(x,t) - u(\hat x, \hat t) - p_0' \cdot (x' - \hat x') - \xi_u(x'', t) &\leq 0,
&&\text{for } (x,t) \in \cl Q, x' - \hat x \in \nbd\lambda(\hat V),\\
v(y,s) - v(\hat y, \hat s) - p_0' \cdot (y' - \hat y') - \xi_v(y'', s) &\geq 0,
&&\text{for } (y,s) \in \cl Q, y' - \hat y \in \nbd\lambda(\hat U).
\end{align*}

We set $r := \lambda/10$ and introduce the closed sets
\begin{align*}
X := \cl{(\nbd{r}(\hat U))^c}, \qquad Y:= \cl{(\nbd{r}(\hat V))^c}.
\end{align*}

Since $\dist(\hat U, X) = \dist(\hat V, Y) = r$, the semi-continuity of $u$ and $v$ imply that
there exists $\delta > 0$ such that
\begin{align*}
u(x,t) - u(\hat x, \hat t) - p_0' \cdot (x' - \hat x') - \xi_u(x'', t) &< 0,
&&x' - \hat x' \in X, |x'' - \hat x''| \leq \delta, |t - \hat t| \leq \delta,\\
v(y,s) - v(\hat y, \hat s) - p_0' \cdot (y' - \hat y') - \xi_v(y'', s) &> 0,
&&y' - \hat y' \in Y, |y'' - \hat y''| \leq \delta, |s - \hat s| \leq \delta.
\end{align*}
Note that if $X$ is unbounded, then $u(x, t) = c_u < u(\hat x, \hat t)$ for all $x \notin K$ and
therefore we only need to use semi-continuity of $u$ on a compact subset of $X$ to get the $\delta$
above. We can similarly handle the
case of unbounded $Y$.

Therefore as in \cite{GGP13AMSA}, we define the pairs
\begin{align*}
S_u := (\hat U^c, \hat U \setminus \nbd{\lambda-3r}(\hat V)), \qquad
S_v := (\hat V^c, \hat V \setminus \nbd{\lambda-3r}(\hat U)).
\end{align*}
We note that both $S_u$ and $S_v$ are bounded pairs. Indeed, $S_u$ bounded if $\hat U$ is bounded
or $\hat U^c \cup \hat V$ is bounded.
Since $u(\hat x, \hat t) - v(\hat y, \hat t) \geq m_0$, we deduce that $u(\hat x, \hat t) > v(\hat
y, \hat t)$. Then if $\hat U$ is unbounded, we have $u(\hat x, \hat t) \leq c_u$ and therefore
$v(\hat y, \hat s) < u(\hat x, \hat t) \leq c_u \leq c_v$, and so we conclude that $\hat U^c \cup \hat V$ are both bounded.
We can argue similarly for $S_v$.

Since both $S_u$ and $S_v$ are bounded pairs, Corollary~\ref{co:approximate pair sliced} (currently
only for $k = 1, 2$) implies that there exist $p_0$-admissible pairs $(U_-, U_+)$ and $(V_-, V_+)$ such that
\begin{align*}
\nbd{2r}(S_u) &\preceq (U_-, U_+) \preceq \nbd{3r}(S_u),\\
\nbd{2r}(S_v) &\preceq (V_-, V_+) \preceq \nbd{3r}(S_v).
\end{align*}

We have the following lemma.

\begin{lemma}
\label{le:pair properties}
The pair $(U_-, U_+)$ and the pair $(V_-, V_+)$ have the following properties:
\begin{enumerate}
\item The pairs are strictly ordered in the sense
\begin{align}
\label{pair order}
\nbd r(U_-, U_+) \preceq (V_+, V_-) = - (V_-, V_+).
\end{align}
\item The origin $0$ lies in the interior of the intersection of the facets, that is,
\begin{align*}
\cl B_r(0) \subset U_-^c \cap U_+^c \cap V_-^c \cap V_+^c.
\end{align*}
\item
The pairs are in general position with respect to $R_u$ and $R_v$, that is,
\begin{align*}
\nbd r(R_u) \preceq (U_-, U_+), \qquad \nbd r(R_v) \preceq (V_-, V_+),
\end{align*}
where
\begin{align*}
R_u &:= (X, X^c \setminus \nbd\lambda(\hat V)),\\
R_v &:= (Y, Y^c \setminus \nbd\lambda(\hat U)).
\end{align*}

\end{enumerate}
\end{lemma}

\begin{proof}
See \cite[Lemma~4.6]{GGP13AMSA}.
\end{proof}

Now we have all that we need to reach a contradiction.
Let us define
\begin{align*}
\tilde u(x' - \hat x') &:=\sup_{|x'' - \hat x''| \leq \delta}\sup_{|t - \hat t| \leq \delta} \bra{ u(x,t) - u(\hat
x, \hat t) - p_0' \cdot (x' - \hat x') - \xi_u(x'', t)},\\
\tilde v(y' - \hat y') &:= \inf_{|y'' - \hat y''| \leq \delta}\inf_{|s - \hat s| \leq \delta} \bra{v(y,s) - v(\hat
y, \hat s) - p_0' \cdot (y' - \hat y') - \xi_v(y'', s)}.
\end{align*}
By the construction above, we have $\tilde u < 0$ on $X$ and $\tilde u \leq 0$ on $X \cup
\nbd\lambda(\hat V)$. Similarly, we have $\tilde v > 0$ on $Y$ and $\tilde v \geq 0$ on $Y \cup
\nbd\lambda(\hat U)$.
Lemma~\ref{le:pair properties}(c) implies that $X \supset \nbd{r}(U_-)$ and $X \cup \nbd\lambda(
\hat V) \supset \nbd{r}(U_+^c)$.
Therefore for any support function $\psi$ of pair $(U_-, U_+)$ we can by upper semi-continuity of
$\tilde u$ find two constants $\alpha,
\beta > 0$ so that $\alpha \psi_+ - \beta \psi_- \geq \tilde u(\cdot - w)$ for all $|w| \leq r/2$
in a neighborhood of the facet $U_-^c \cap U_+^c$. An analogous reasoning applies to $\tilde v$.

Since the pairs $(U_-, U_+)$ and $(V_-,
V_+)$ are $p_0$-admissible, there exist faceted functions $\psi_u, \psi_v \in \domain(\Lambda_{p_0})$
that are the support functions of the respective pairs. By applying the observation in the previous
paragraph, we can assume that $\tilde u(\cdot - w) \leq \psi_u$ and $\psi_v \leq \tilde v(\cdot -
w)$ for all $|w| \leq r/2$ in a neighborhood of the respective facets $U_-^c \cap U_+^c$ and $V_-^c
\cap V_+^c$. Therefore the functions $\varphi_u(x,t) := \psi_u(x') + p_0' \cdot (x' - \hat x') + \xi_u(x'',
t)$, $\varphi_v := \psi_v(x') + p_0' \cdot (y' - \hat y')
+ \xi_v(x'', t)$ are test functions for $u$ and $v$, respectively, in the sense of
Definition~\ref{def:visc-solution}(i). Due to Lemma~\ref{le:pair properties}(a--b), the comparison principle for the
crystalline curvature Proposition~\ref{pr:comparison Lambda} yields
\begin{align}
\label{essinf sup order}
\essinf_{B_r(0)} \left[ \Lambda_{p_0}[\psi_u] \right] \leq \esssup_{B_r(0)} \left[
\Lambda_{p_0}[\psi_v] \right].
\end{align}

From the definition of viscosity solutions, namely Definition~\ref{def:visc-solution}(i), we infer
\begin{align*}
(\xi_u)_t(\hat t) + F\pth{p_0, \essinf_{B_r(0)} \left[ \Lambda_{p_0}[\psi_u] \right]} &\leq 0,\\
(\xi_v)_t(\hat s) + F\pth{p_0, \esssup_{B_r(0)} \left[ \Lambda_{p_0}[\psi_v] \right]} &\geq 0.
\end{align*}
Using \eqref{essinf sup order} and the ellipticity of $F$, we get after subtracting the above two inequalities
\begin{align*}
0 < \frac \e{(T - \hat t)^2} + \frac \e{(T - \hat s)^2} + F\pth{p_0, \essinf_{B_r(0)} \left[
\Lambda_{p_0}[\psi_u] \right]} - F\pth{p_0, \esssup_{B_r(0)} \left[ \Lambda_{p_0}[\psi_v] \right]}
\leq 0,
\end{align*}
a contradiction.

This finished the proof of the comparison principle Theorem~\ref{th:comparison principle} since we
have shown that \eqref{m_0} always yields a contradiction.

\section{Stability}
\label{se:stability}

We will show the stability of \eqref{P} under the approximation
by parabolic problems
\begin{align}
\label{regularized-problem}
\begin{cases}
u_t + F(\nabla u, \trace\bra{(\nabla_p^2 W_m)(\nabla u) \nabla^2 u}) = 0,\\
\at{u}{t=0} = u_0,
\end{cases}
\end{align}
where $W_m$ approximate $W$ as in Section~\ref{sec:resolvent-approximation}.
An example of such sequence $\set{W_m}$ is given in Example~\ref{ex:wm-example}.

The main result of this section is the following stability theorem.
We recall the definition of \emph{half-relaxed limits} (semi-continuous limits)
\begin{align*}
\halflimsup_{m\to\infty} u_m(x,t) &:= \lim_{k \to \infty} \sup_{m > k} \sup_{|y - x| < \frac 1k} \sup_{|t-s| <
\frac 1k} u_m(y,s),\\
\halfliminf_{m\to\infty} u_m(x,t) &:= -\halflimsup_{m\to\infty} \big(-u_m(x,t)\big).
\end{align*}

\begin{theorem}[Stability]
\label{th:stability quadratic}
Let $u_m$ be a locally bounded sequence of viscosity solutions of \eqref{regularized-problem}
(without the initial condition).
Then $\halflimsup_{m\to\infty} u_m$ is a viscosity subsolution
of \eqref{P}
and $\halfliminf_{m\to\infty} u_m$ is a viscosity supersolution of \eqref{P}.
\end{theorem}

\medskip

\noindent
\emph{Proof of stability}
We will only show the subsolution part, the proof of the supersolution part is analogous.
Let $u = \halflimsup_m u_m$.
Clearly $u$ is upper semi-continuous.
We want to show that $u$ is a subsolution of \eqref{P}.

We have to verify (i)--(ii) of Definition~\ref{def:visc-solution} and (i-cf) of
Definition~\ref{def:level-set-test} for curvature-free type $F$.

\subsection{Case (i)}
\label{se:stability case (i)}
Suppose that $\vp$ is a stratified faceted test function at $(\hat x, \hat t)$
with gradient $\hat p$ and with  $\bar \psi$, $f$ and $g$ as in
Definition~\ref{def:strat-faceted-test-function},
and suppose that this test function is a test function for $u$
in the sense of Definition~\ref{def:visc-solution}(i), i.e.,
it satisfies \eqref{general-position} with some $\rho > 0$.
Let $(A_-, A_+)$ be the pair supported by $\bar \psi$.
We will set $V \subset \Rn$ to be the subspace parallel to $\aff \partial W(\hat p)$, $U=V^\perp$ and $k = \dim V$.
We recall that we have the rotated coordinate system $x = \TT (x', x'')$ with $\TT = \TT_{\hat p}$ introduced
in \eqref{rotation}.

Let us define the function $\bar u: \R^k \to \R$
\begin{align*}
\bar u(x') := \sup_{\substack{\abs{x''} \leq \rho\\\abs{t - \hat t} \leq \rho}} u(\hat x + x, t)
&- u(\hat x, \hat t) - f(x'') - g(t) + g(\hat t) - \hat p \cdot x
\end{align*}
and the closed subsets of $\R^k$
\begin{align*}
Y &:= \set{x' \in \R^k : \bar u(x') \geq 0}, \qquad\\
Z &:= \set{x'\in O: \bar \psi(x') \leq 0} = A_+^c \cap O,
\end{align*}
where $O = \nbd\rho(\facet A)$,
see Figure~\ref{fig:stability-geometry}.
Note that with this definition of $\bar u$, the condition \eqref{general-position} is equivalent to
\begin{align}
\label{gp-bar}
\bar u(y') \leq \bar\psi(x') \qquad \text{for all $x' \in O$, $\abs{y' - x'} \leq \rho$.}
\end{align}

\begin{figure}
\centering
\def\svgwidth{4.5in}
\begingroup%
  \makeatletter%
  \providecommand\color[2][]{%
    \errmessage{(Inkscape) Color is used for the text in Inkscape, but the package 'color.sty' is not loaded}%
    \renewcommand\color[2][]{}%
  }%
  \providecommand\transparent[1]{%
    \errmessage{(Inkscape) Transparency is used (non-zero) for the text in Inkscape, but the package 'transparent.sty' is not loaded}%
    \renewcommand\transparent[1]{}%
  }%
  \providecommand\rotatebox[2]{#2}%
  \ifx\svgwidth\undefined%
    \setlength{\unitlength}{324bp}%
    \ifx\svgscale\undefined%
      \relax%
    \else%
      \setlength{\unitlength}{\unitlength * \real{\svgscale}}%
    \fi%
  \else%
    \setlength{\unitlength}{\svgwidth}%
  \fi%
  \global\let\svgwidth\undefined%
  \global\let\svgscale\undefined%
  \makeatother%
  \begin{picture}(1,0.44444444)%
    \put(0,0){\includegraphics[width=\unitlength,page=1]{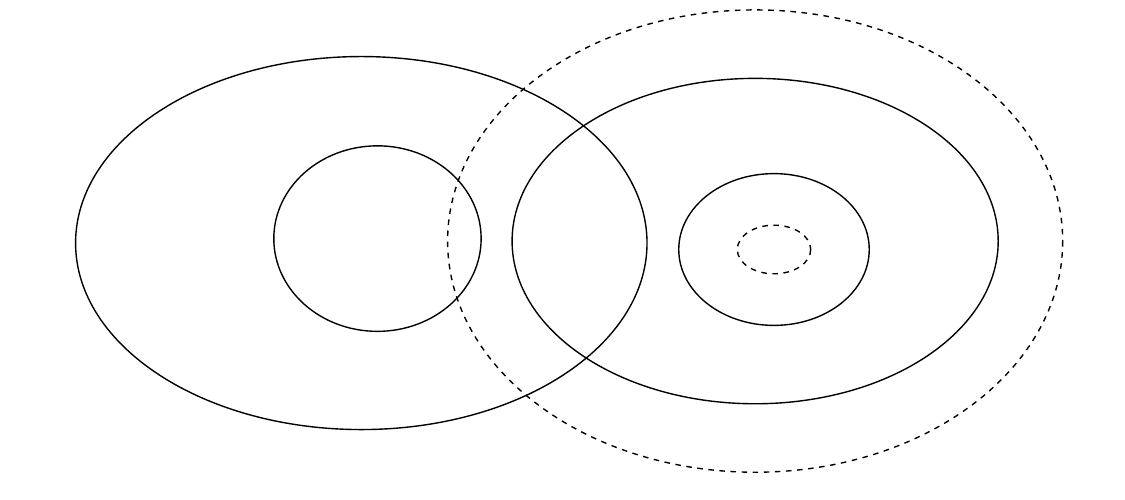}}%
    \put(0.32380951,0.35273369){\color[rgb]{0,0,0}\makebox(0,0)[lb]{\smash{$Y$}}}%
    \put(0.64215171,0.32804233){\color[rgb]{0,0,0}\makebox(0,0)[lb]{\smash{$Z$}}}%
    \put(0,0){\includegraphics[width=\unitlength,page=2]{stability-geometry.pdf}}%
    \put(0.27266314,0.24162258){\color[rgb]{0,0,0}\makebox(0,0)[lb]{\smash{$\bar u > 0$}}}%
    \put(0.64532629,0.2506172){\color[rgb]{0,0,0}\makebox(0,0)[lb]{\smash{$\bar \psi <0$}}}%
    \put(0.51164023,0.20335098){\color[rgb]{0,0,0}\makebox(0,0)[lb]{\smash{$\hat x'$}}}%
    \put(0.81181656,0.36596114){\color[rgb]{0,0,0}\makebox(0,0)[lb]{\smash{$O$}}}%
    \put(0.49647266,0.28042331){\color[rgb]{0,0,0}\makebox(0,0)[lb]{\smash{$N$}}}%
    \put(0,0){\includegraphics[width=\unitlength,page=3]{stability-geometry.pdf}}%
  \end{picture}%
\endgroup%
\caption{Situation at the contact point of $\bar u$ and $\bar \psi$.
The thick line denotes the boundary of $N$.}
\label{fig:stability-geometry}
\end{figure}

We immediately have the following ``geometrical'' lemma.
Intuitively, since $\bar u$ and $\bar\psi$ are ordered even when shifted by a small distance,
we must have that $\bar\psi$ is nonnegative in a neighborhood of the set $Y$ where $\bar u$ is nonnegative,
and, analogously, $\bar u$ is nonpositive in a neighborhood of the set $Z$ where $\bar \psi$ is
nonpositive.

\begin{lemma}[{cf. \cite[Lemma~5.6]{GGP13JMPA}}]
\label{le:u-psi-shift}
Suppose that $u$ and $\vp$ satisfy \eqref{general-position} for some
$\rho > 0$, $(\hat x, \hat t) \in \Rn$.
Then
\begin{align*}
\bar u(x') \leq 0 \qquad \text{for all $x' \in \nbd\rho(Z)$,}
\end{align*}
or, more explicitly,
\begin{align*}
u(x,t) \leq f\pth{x''- \hat x''} + g(t) - g(\hat t) + u(\hat x,\hat t)
+ \hat p \cdot (x - \hat x)
\end{align*}
for
$x' - \hat x \in \nbd\rho(Z)$, $\abs{x'' - \hat x''} \leq \rho$, $\abs{t - \hat t} \leq \rho$.
Furthermore, we have
\begin{align*}
\bar \psi(x') \geq 0 \qquad \text{for } x' \in \nbd\rho(Y) \cap O.
\end{align*}
\end{lemma}

\begin{proof}
Let us prove the first statement.
If $x' \in \nbd\rho(Z)$ then there exists $z' \in Z \subset O$ such that
$\abs{x' - z'} \leq \rho$.
Thus \eqref{gp-bar} and the definition of $Z$ imply
\begin{align*}
\bar u(x') \leq \bar\psi(z') \leq 0,
\end{align*}
and that is what we wanted to prove.

Similarly, if we suppose that
$x' \in \nbd\rho(Y) \cap O$, there exists $y' \in Y$ with $\abs{x' - y'} \leq \rho$.
Then \eqref{gp-bar} and the definition of $Y$
imply
\begin{align*}
\bar\psi(x') \geq \bar u(y') \geq 0.
\end{align*}
The lemma is proved.
\end{proof}

We obtain the following corollary.

\begin{corollary}
Suppose that \eqref{gp-bar} holds with $\rho > 0$.
Then there exists $\delta$, $0 < \delta \leq \rho/5$,
such that $\nbd{4\delta}(N) \subset O$
where
\begin{align*}
N := \nbd\delta(Z) \cap Y \cap O,
\end{align*}
and moreover
\begin{align*}
\cl B^k_\delta(0) \subset A_-^c \cap A_+^c
\end{align*}
and
\begin{align}
\label{u-psi-scaled-comp}
\bar u(x') \leq \alpha \bar\psi(x' + z') \qquad \text{for all $\alpha > 0$, $x' \in \nbd{3\delta}(N)$, $\abs{z'} \leq \delta$,}
\end{align}
with strict inequality for $x' \notin N$.
\end{corollary}

\begin{proof}
By definition, $Z \subset A_+^c$. Moreover, the second result in
Lemma~\ref{le:u-psi-shift} is equivalent to
\begin{align*}
\nbd\rho(Y) \cap O \subset A_-^c.
\end{align*}
We can therefore estimate
\begin{align}
\label{N-bound}
N = \nbd\delta(Z) \cap Y \cap O \subset \nbd\delta(A_+^c) \cap A_-^c
\subset \nbd\delta(\partial A_+) \cup (A_+^c \cap A_-^c).
\end{align}
Since $A_+$ is open, we have $\partial A_+ \subset A_+^c$.
But since $A_- \cap A_+ = \emptyset$ and $A_-$ is also open, we must also have
$\partial A_+ \subset A_-^c$.
Therefore $\partial A_+$ is in the facet, and by assumption on $O$ we have
\begin{align}
\label{boundary-in-facet}
\partial A_+ \subset A_-^c \cap A_+^c \subset O.
\end{align}
Since $O$ is open and $A_-^c \cap A_+^c$ is compact, and
$0\in \interior A_-^c \cap A_+^c$, for $\delta > 0$ small enough we will have
\begin{align*}
\nbd{5\delta}(A_-^c \cap A_+^c) \subset O \qquad \text{and} \qquad
\cl B_\delta^k(0) \subset A_-^c \cap A_+^c.
\end{align*}
Using \eqref{boundary-in-facet} in \eqref{N-bound}, we obtain
\begin{align}
\label{N-4delta}
\nbd{4\delta}(N) \subset \nbd{5\delta}(A_-^c \cap A_+^c) \subset O.
\end{align}
Let us now fix $\alpha > 0$ and $\abs{z'} \leq \delta$.
Using the definition and \eqref{N-4delta}, we can estimate
\begin{align*}
\nbd{3\delta}(N) \subset \nbd{4\delta}(Z) \cap \nbd{3\delta}(Y) \cap
\nbd{-\delta}(O).
\end{align*}
In particular, if $x' \in \nbd{3\delta}(N)$ then $x' + z' \in \nbd\rho(Y) \cap O$
and $x' \in \nbd\rho(Z)$.
Hence Lemma~\ref{le:u-psi-shift} applies, yielding
\begin{align*}
\bar u(x') \leq 0 \leq \bar\psi(x' + z'),
\end{align*}
and therefore \eqref{u-psi-scaled-comp} follows.
If $x' \in \nbd{3\delta}(N) \setminus N$, then we must have $x' + z' \in O$
and at least one of the following:
\begin{itemize}
\item
$x' \notin \nbd\delta(Z)$: Thus $x' + z' \in O \setminus Z$ and therefore $\psi(x'+z') > 0$.
\item
$x' \notin Y$: Thus $u(x') < 0$.
\end{itemize}
We deduce the strict ordering in \eqref{u-psi-scaled-comp} for $x \notin N$.
\end{proof}

The previous corollary has the following important direct consequence.

\begin{lemma}[cf. {\cite[Lemma 5.4]{GGP13JMPA}}]
\label{le:strict_order}
Suppose that \eqref{gp-bar} is satisfied for some $\rho > 0$.
By adding the term $\abs{y}^2$ to $f(y)$ and $\abs{t - \hat t}^2$ to $g(t)$ if necessary,
there exists $0 < \delta < \rho/5$ such that
for all $\abs{z'} \leq \delta$ and $\alpha > 0$
\begin{align*}
u(x,t) - \alpha \bar \psi(x' + z' - \hat x')
- f(x'' - \hat x'') - g(t) - \hat p \cdot (x - \hat x)
\leq u(\hat x, \hat t) - g(\hat t)
\end{align*}
whenever
\begin{align*}
x' - \hat x' \in \nbd{3\delta}(N),\ \abs{x'' - \hat x''} \leq \rho,\ \abs{t - \hat t} \leq \rho,
\end{align*}
with a \emph{strict} inequality outside of  $\set{(x,t): x' - \hat x' \in N,\ x'' = \hat x'',\ t = \hat t}$.
\end{lemma}

We shall now proceed with the proof of stability.
By Proposition~\ref{pr:curvature-as-min-section},
for $L > 0$ sufficiently large and $\Gamma' = \R^k / L \Z^k$,
we can find a function $\xi \in \Lip(\Gamma')$
such that $\xi(x') = \bar\psi(x')$ on a neighborhood of the facet
$\facet A$ such that $\xi \in \domain(\partial \SE_{\hat p}(\cdot; \Gamma'))$
and
$\Lambda_{\hat p}(\bar \psi) = -\partial^0 \SE_{\hat p}(\xi; \Gamma')$
a.e. on $\facet A$.
By making the set $O$ smaller if necessary,
we can assume that $\xi = \bar\psi$ on $O$.
Let $\delta > 0$ be from Lemma~\ref{le:strict_order}.

Fix $\alpha > 0$.
Since $\nabla f(0) = 0$,
we can find $\theta_\alpha > 0$ and $f_\alpha \in \Lip(\Gamma'')$,
$\Gamma'' = \R^{n-k}/L \Z^{n-k}$,
such that $f_\alpha(x'') = f(x'')$ for $\abs{x''} \leq 2\theta$
with $\norm{\nabla f_\alpha}_\infty \leq \alpha \norm{\nabla\xi}_\infty$.
Let us define the function
\begin{align}
\label{def-psi}
\psi(x) = \psi_\alpha(x) = \alpha \xi\pth{x'} + f_\alpha\pth{x''},
\qquad x \in \Gamma = \TT(\Gamma' \times \Gamma'').
\end{align}
We see that $\psi \in \Lip(\Gamma)$ and therefore by
Lemma~\ref{le:subdiff-slicing} $\psi \in \domain(\partial E_{\hat p}'(\cdot; \Gamma))$.
We can estimate
\begin{align}
\label{lip-bound-psi-alpha}
\norm{\nabla \psi_\alpha}_\infty \leq 2\alpha \norm{\nabla \xi}_\infty.
\end{align}
In particular,
if $\alpha$ is sufficiently small,
$\partial E_{\hat p}(\psi_\alpha) = \partial E_{\hat p}'(\psi_\alpha)$
by Lemma~\ref{le:subdiff-homog-relation}.

From now on we fix one such $\alpha$ and we write $\psi = \psi_\alpha$.
For given $a >0$ let $\psi_a$ and $\psi_{a,m}$ be the solutions
of the resolvent problems in Proposition~\ref{pr:resolvent-problems}
for energies $E_W = E_{\hat p}$ and $E_m = E_{W_m(\cdot - \hat p) - W(\hat p)}$, respectively, on $\Gamma$.
Note that these energies satisfy all the assumptions of Proposition~\ref{pr:resolvent-problems}.

For given $a > 0$ and $\abs{z'} \leq \delta$, we
define the set of maxima
\begin{align*}
A_{a, z'} := \argmax_{(x,t) \in M_2} \bra{u(x + \hat x, t + \hat t) -
\psi_a(x + \TT_V z') - \hat p \cdot x - g(t + \hat t)}
\end{align*}
where $M_s := \set{(x,t): x' \in \nbd{s\delta}(N),\ \abs{x''} \leq s \theta,\ \abs t \leq s\delta}$.
Note that $\psi(x + \TT_V z') = \alpha \bar \psi(x' + z') + f(x'')$ for $(x,t) \in M_2$,
$\abs{z'} \leq \delta$.
Due to the uniform convergence $\psi_a \rightrightarrows \psi$
on $\Gamma$ from Proposition~\ref{pr:resolvent-problems},
and the strict ordering of Lemma~\ref{le:strict_order},
we have that there exists $a_0 > 0$, independent of $z'$, such that
\begin{align}
\label{Az-in-M1}
\emptyset \neq A_{a,z'} \subset M_1 \qquad \text{for all $\abs{z'} \leq \delta$, $a < a_0$.}
\end{align}

We now fix one such $a < a_0$ and find $\abs{z'} \leq \delta$ such that
\begin{align}
\label{choice-of-z'}
\psi_a(\TT_V z') - \alpha\bar\psi(z')
= \min_{\abs{w'} \leq \delta} \bra{\psi_a(\TT_V w') - \alpha\bar\psi(w')}.
\end{align}
As in \cite{GGP13JMPA}, $z'$ is chosen in such a way that Lemma~\ref{le:resolvent-order} below holds.

Due to the uniform convergence $\psi_{a,m} \rightrightarrows \psi_a$
as $m\to\infty$ there exists $(x_a, t_a) \in A_{a,z'}$ and
a sequence $(x_{a,m}, t_{a,m})$ (for a subsequence of $m$) of local
maxima of
\begin{align*}
(x,t) \mapsto u_m(x + \hat x,t + \hat t) - \psi_{a,m}(x + \TT_V z') - \hat p \cdot x - g(t + \hat t)
\end{align*}
such that $(x_{a,m}, t_{a,m}) \to (x_a, t_a)$ as $m\to\infty$
(along a subsequence).

Recall the definitions of $h_a$ and $h_{a,m}$ from Proposition~\ref{pr:resolvent-problems}.
Since $\psi_{a,m} \in C^{2,\al}(\Gamma)$
and $u_m$ is a viscosity subsolution of \eqref{regularized-problem}, we must have
\begin{align*}
&g'(t_{a,m} + \hat t) + F(\nabla \psi_{a,m}(x_{a,m} + \TT_V z') + \hat p, h_{a,m}(x_{a,m} + \TT_V z'))\\
&
\begin{aligned}
= g'(t_{a,m} + \hat t) + F\big(&\nabla \psi_{a,m}(x_{a,m} + \TT_V z') + \hat p,\\
&\trace\bra{(\nabla_p^2
W_m)(\nabla \psi_{a,m} + \hat p) \nabla^2 \psi_{a,m}}(x_{a,m} + \TT_V z')\big) \leq 0.
\end{aligned}
\end{align*}

By the uniform Lipschitz bound
$\norm{\nabla \psi_{a,m}}_\infty \leq \norm{\nabla \psi}_\infty \leq C\alpha$ from \eqref{lip-bound-psi-alpha},
and $h_{a,m} \rightrightarrows h_a$ as $m \to\infty$,
we can find a point $p_a \in \Rn$, $\abs{p_a - \hat p} \leq C\alpha$,
and send $m\to\infty$ along a \emph{subsequence} to recover
\begin{align}
\label{subsol ha}
g'(t_a + \hat t) &+ F(p_a, h_a(x_a + \TT_V z')) \leq 0.
\end{align}
To estimate $h_a(x_a + \TT_V z')$, we prove the following lemma.

\begin{lemma}[{cf. \cite[Lemma~5.5]{GGP13JMPA}}]
\label{le:resolvent-order}
We have
\begin{align*}
h_a(x_a + \TT_V z') \leq h_a(\TT_V z') = \min_{\abs{w'}\leq \delta} h_a(\TT_V w').
\end{align*}
\end{lemma}

\begin{proof}
We chose $z'$ so that \eqref{choice-of-z'} holds and therefore the equality above holds as well.
Therefore we only need to show the inequality.
Recalling the definition of $h_a$, we have to show that
\begin{align}
\label{main-ineq-order}
\psi_a(x_a + \TT_Vz') - \psi(x_a + \TT_V z') \leq
\psi_a(\TT_Vz') - \psi(\TT_V z').
\end{align}
We begin by expressing the second term on the left-hand side
using \eqref{def-psi}, which yields
\begin{align*}
- \psi(x_a + \TT_V z') = -\alpha \bar\psi\pth{x_a'+ z'} - f\pth{x_a''}.
\end{align*}
Since $(x_a, t_a) \in A_{a,z'} \in M_1$ by \eqref{Az-in-M1},
clearly
\begin{align}
\label{point-in-neighb}
x_a' + z' \in \nbd{3\delta}(Z) \cap
\nbd{2\delta}(Y) \cap O
\end{align}
and therefore $\bar\psi\pth{x_a' + z'} \geq 0$ by Lemma~\ref{le:u-psi-shift}.
This implies
\begin{align}
\label{main-ineq-1}
- \psi(x_a + \TT_V z') \leq -f\pth{x_a''}.
\end{align}
For the first term in \eqref{main-ineq-order},
we use the fact that $(x_a, t_a)$ is a point of maximum
and therefore
\begin{align*}
u(x_a + \hat x, t_a + \hat t) - \psi_a(x_a + \TT_V z') - \hat p \cdot x_a -g(t_a + \hat t) \geq u(\hat x, \hat t) - \psi_a(\TT_V z') - g(\hat t).
\end{align*}
After rearranging the terms, we obtain
\begin{align*}
\psi_a(x_a + \TT_V z') \leq \bra{u(x_a + \hat x, t_a + \hat t) - u(\hat x, \hat t) - \hat p \cdot x_a - g(t_a + \hat t) + g(\hat t)} + \psi_a(\TT_V z').
\end{align*}
We use \eqref{point-in-neighb} again and therefore the first inequality of Lemma~\ref{le:u-psi-shift}
allows us estimate the term in the bracket from above by $f(x_a'')$,
yielding
\begin{align}
\label{main-ineq-2}
\psi_a(x_a + \TT_V z') \leq \psi_a(\TT_V z')
+ f\pth{x_a''}.
\end{align}
Finally, by the choice of $\de$ we have $z' \in A_-^c \cap A_+^c$ and therefore
\begin{align*}
\psi(\TT_V z') = 0.
\end{align*}
Hence using this observation, and taking the sum of \eqref{main-ineq-1} and \eqref{main-ineq-2}
we arrive at \eqref{main-ineq-order} and the proof of the lemma is finished.
\end{proof}

Then, by the ellipticity of $F$ in \eqref{F ellipticity},
\begin{align*}
g'(t_a + \hat t) + F(p_a + \hat p, \min_{\abs{w'} \leq \delta}  h_a(\TT_V w')) \leq g'(t_a + \hat t) + F(p_a, h_a(x_a + \TT_V z')) \leq 0.
\end{align*}

We send $a \to 0$ along a subsequence $a_l$ such that
$\min h_{a_l} \to \liminf_{a\to 0} \min h_a$
and $p_a \to p_0$ as $l \to\infty$, for some $p_0 \in \R^n$, $\abs{p_0 - \hat p} \leq C\alpha$,
to obtain
\begin{align*}
g'(\hat t) + F(p_0, \liminf_{a\to 0} \min_{\abs{w'} \leq \delta} h_a(T_V w')) \leq 0.
\end{align*}
Now we use Lemma~\ref{le:subdiff-slicing},
in particular the fact that $h_a(x) = \bar h_a(x')$
for some $\bar h_a = (\bar \psi_a - \bar \psi)/a \in \Lip(\Gamma')$
and that $\bar h_a \to -\partial^0 \SE_{\hat p}(\bar\psi; \Gamma')$
in $L^2(\Gamma')$.
Thus, recalling Proposition~\ref{pr:curvature-as-min-section},
\begin{align*}
\liminf_{a\to 0} \min_{\abs{w'} \leq \delta} h_a(\TT_V w')
=\liminf_{a\to 0} \min_{\abs{w'} \leq \delta} \bar h_a(w')
\leq \essinf_{B_\delta(0)} -\partial^0 \SE_{\hat p}(\xi; \Gamma')
= \essinf_{B_\delta(0)} \Lambda_{\hat p}[\bar\psi],
\end{align*}
and ellipticity yields
\begin{align*}
g'(\hat t) + F(p_0,  \essinf_{B_\delta(0)} \Lambda_{\hat p}[\bar\psi]) \leq 0.
\end{align*}
Since this holds for any $\alpha > 0$ small,
and therefore continuity of $F(p,\xi)$ in $p$ and the estimate
$\abs{p_0 - \hat p} \leq C \alpha$ yields
\begin{align*}
g'(\hat t) + F(\hat p,  \essinf_{B_\delta(0)} \Lambda_{\hat p}[\bar\psi]) \leq 0,
\end{align*}
which we needed to prove.

\subsection{Case (ii)}
In this case the test function is also a test function \eqref{regularized-problem} and therefore
the stability follows the standard viscosity solution argument.

\subsection{Case (curvature-free type)}

In this part we will assume that $F$ is of curvature-free type at $p = 0$
in the sense of Definition~\ref{def:level-set-type}.
We need to verify Definition~\ref{def:level-set-test}(i-cf).

Suppose therefore that $\phi(x,t) = g(t)$ on a neighborhood $U$ of a point $(\hat x, \hat t)$
and $u - \phi$ has a local maximum 0 at $(\hat x, \hat t)$.
We want to show that
$g_t (\hat t) + F(0, 0) \leq 0$.

This can be accomplished by perturbing the test function $\phi(x,t)$
and considering the function
\begin{align*}
\phi_{m,q}(x,t) = W^\star_{m; A, q}(x - \hat x) + g(t) +\abs{t - \hat t}^2,
\end{align*}
with $W^\star_{m;A,q}$ given by \cite[Lemma~5.8]{GGP13AMSA},
and with suitable parameters $A, q > 0$.

Let us recall that $W^\star_{m; A,q}$ is the Legendre-Fenchel transform
of
\begin{align*}
W_{m;A,q}(p) := A\pth{W_m(p) + q \psi\pth{\frac pq} - W_m(0)}.
\end{align*}
Here $\psi : \Rn \to [0, \infty]$ is a lower semi-continuous nonnegative convex
function such that $\psi \in C^\infty(B_1(0))$, $\psi(0) = 0$
and $\psi(p) = \infty$ for $\abs p \geq 1$.
The semi-continuity then implies $\psi(p) \to \infty$ as $ \abs p \to 1^-$.

The following lemma was proved in \cite{GGP13AMSA}.

\begin{lemma}[c.f. {\cite[Lemma~5.6]{GGP13AMSA}}]
\label{le:Wstar-props}
For any $m, A, q$ positive, $W^\star_{m;A,q}$ is a strictly convex,
nonnegative, $C^2$ function on $\Rn$ and
\begin{align*}
\abs{\nabla W^\star_{m;A,q}(x)} \leq q, \qquad
0 \leq \mathcal L_m(W^\star_{m;A,q})(x) \leq A^{-1} n,
\qquad x \in \Rn,
\end{align*}
where $\mathcal L_m(u)(x) := \trace \bra{(\nabla^2 W_m) (\nabla u(x)) \nabla^2 u(x)}$ for $u \in C^2(\Rn)$.
\end{lemma}

We will add the following modification of \cite[Lemma~5.8]{GGP13AMSA}.

\begin{lemma}
\label{le:Wstar-growth}
For every $\delta > 0$ there exists $A > 0$
such that for every $q > 0$ there exist $\e > 0$ and $m_0 > 0$
for which
\begin{align*}
W^\star_{m; A, q}(x) > \e, \qquad \text{for all } x,  \abs x \geq \delta,
\text{ and } m \geq m_0.
\end{align*}
\end{lemma}

\begin{proof}
Let us define
\begin{align}
\label{def-mu}
\mu := \sup_{\abs p = 1/2} \bra{W(p) + \psi(p)} \in (0, \infty)
\end{align}
and set for given $\de > 0$
\begin{align*}
A := \frac{\delta}{8\mu}.
\end{align*}
Now we fix $q > 0$ and set
\begin{align*}
\e := \frac{q\delta}{8}.
\end{align*}
By the locally uniform convergence of $W_m \to W$, we can find $m_0 > 0$
such that
\begin{align}
\label{def-m0}
\sup_{\abs p = q/2} \abs{W_m(p) - W_m(0) - W(p)} \leq q \mu
\qquad m \geq m_0.
\end{align}
Now whenever $\abs{x} \geq \delta$ and $m \geq m_0$,
we can take $p = \frac q2 \frac{x}{\abs{x}}$
and estimate,
using \eqref{def-m0}, one-homogeneity of $W$, and \eqref{def-mu},
\begin{align*}
W^\star_{m; A, q} &\geq x \cdot p - W_{m; A, q} (p)\\
&= \frac{q}{2} \abs{x} - A \pth{W_m(p) + q \psi\pth{\frac pq}
- W_m(0)}\\
&\geq \frac{q}{2} \abs{x} - A \pth{W(p) + q \psi\pth{\frac pq} + q\mu}\\
&= \frac{q}{2} \abs{x} - A \pth{qW\pth{\frac pq} + q \psi\pth{\frac pq} + q\mu}\\
&\geq \frac{q}{2} \abs{x} - 2 Aq \mu \geq \frac{q\delta}{4} > \e.
\end{align*}
\end{proof}

\begin{lemma}
\label{le:Wstar-zero}
For any $A, q$ positive
\begin{align*}
W^\star_{m; A, q}(0) \to 0 \qquad \text{as } m \to \infty.
\end{align*}
\end{lemma}
\begin{proof}
Since $W_m$ is a decreasing sequence
converging to $W$ locally uniformly, we have $W_m \geq \min W = 0$
and $W_m(0) \to W(0) = 0$.
As also $\psi \geq 0$, it follows that
\begin{align*}
0 \leq W^\star_{m; A, q}(0) \leq A W_m(0) \to 0.
\end{align*}
\end{proof}

Let us now choose $\de > 0$ small enough
so that $Q := \cl B_\delta(\hat x) \times [\hat t - \delta, \hat t + \delta] \subset U$.
We have $u - \phi \leq 0$ on $Q$ with equality at $(\hat x, \hat t)$.
For this $\delta$ we fix $A > 0$ from Lemma~\ref{le:Wstar-growth}.

Now due to the same lemma for any $q > 0$ we also have $\e, m_0 > 0$
such that
\begin{align*}
u - \phi_{m, q} < -\e \qquad \text{on } \pth{\partial B_\delta(\hat x)}
\times [\hat t - \delta, \hat t + \delta],
\text{ for } m \geq m_0.
\end{align*}
Because $W^\star_{m;A,q} \geq 0$ by Lemma~\ref{le:Wstar-props},
we also have
\begin{align*}
u - \phi_{m,q} \leq  - \delta^2 \qquad \text{on } x \in B_\delta(\hat t),
\ t = \hat t \pm \delta, \text{ for all } m.
\end{align*}
Since $\phi_{m,q}(0) \to 0$ as $m \to \infty$ by Lemma~\ref{le:Wstar-zero}
and since $\phi_{m,q}$ is uniformly Lipschitz in $m$ by Lemma~\ref{le:Wstar-props},
we conclude that there must exist a subsequence $m_j$
and a sequence of points $(x_j, t_j)$
such that $u_{m_j} - \phi_{m_j, q}$ has a local maximum at $(x_j, t_j)$,
$x_j \in B_\delta(\hat x)$,
and, moreover, $t_j \to \hat t$.

Let us now choose $q_k = 1/k$. By the standard diagonalization
argument we can find a subsequence $m_k$ such that
$u_{m_k} - \phi_{m_k, q_k}$ has a local maximum at a point $(x_k, t_k)$,
$x_k \in B_\delta(\hat x)$,
and $\abs{t_k - \hat t} \leq 1/k$.
Thus we introduce
\begin{align*}
p_k &:= \nabla \phi_{m_k, q_k}(x_k, t_k) = \nabla W^\star_{m_k; A, q_k}(x_k - \hat x), \text{ and}\\
\xi_k &:= \mathcal L_{m_k}\pth{\phi_{m_k, q_k}(\cdot, t_k)}(x_k) =
\mathcal L_{m_k}\pth{W^\star_{m_k; A, q_k}}(x_k - \hat x).
\end{align*}
By the assumption that $u_{m_k}$ is a subsolution of \eqref{regularized-problem},
we have
\begin{align*}
g'(t_k) + 2 (t_k - \hat t) + F(p_k, \xi_k) \leq 0.
\end{align*}
Furthermore, from Lemma~\ref{le:Wstar-props} and the choice of $q_k$
we have the bounds
\begin{align*}
\abs{p_k} \leq 1/k, \qquad \abs{\xi_k} \leq A^{-1} n \qquad \text{for all } k,
\end{align*}
where $A$ is independent of $k$.

Since $F$ is of curvature-free type at $p = 0$, Definition~\ref{def:level-set-type},
we finally obtain
\begin{align*}
g'(\hat t) + F(0, 0) &=
g'(\hat t) + \liminf_{p \to 0} \inf_{\abs{\xi} \leq A^{-1} n} F(p, \xi)
\\
&\leq \liminf_{k \to \infty} \bra{g'(t_k) + 2(t_k - \hat t) + F(p_k, \xi_k)} \leq 0.
\end{align*}

The supersolution case can be handled similarly with a test function
\begin{align*}
\phi_{m,q}(x,t) = -W^\star_{m; A, q}(-x + \hat x) + g(t) +\abs{t - \hat t}^2.
\end{align*}
This finishes the proof of stability for the curvature-free test function case.

The proof of Theorem~\ref{th:stability quadratic} is complete.

\subsection{Approximation by linear growth functionals}

In this section we prove the following approximation result:

\begin{theorem}
\label{th:linear growth stability}
Suppose that $F$ is of curvature-free type at $p_0 = 0$ and that $\set{W_m}_{n\in \N} \subset C(\R^n) \cap
C^2(\R^n\setminus\set0)$ are positively one-homogeneous functions with bounded, strictly convex sub-level
sets $\set{W_m
\leq 1}$ such that $W_m \rightrightarrows W$ uniformly on $\cl B_1(0)$.
Let $u_m$ be the unique viscosity solutions of
\begin{align*}
\left\{
\begin{aligned}
u_t + F(\nabla u, \divo \nabla_p W_m(\nabla u)) &= 0, && \text{in $\R^n \times (0,\infty)$,}\\
u(\cdot, 0) &= u_{0, m}, && \text{in $\R^n$},
\end{aligned}
\right.
\end{align*}
where $u_{0, m} \in C(\R^n)$ are uniformly bounded.
Then
\begin{align*}
\overline{u} &:= \halflimsup_{m\to\infty} u_m, &
\underline{u} &:= \halfliminf_{m\to\infty} u_m
\end{align*}
are a viscosity subsolution and a viscosity supersolution of \eqref{P}.
\end{theorem}

\begin{proof}
We will follow the proof of Theorem~\ref{th:stability quadratic} with an additional approximation
because the solutions $\psi_{a,m}$ of the resolvent problem for the linear growth energy $E_m$
might not be smooth.
Let us set for $\delta > 0$
\begin{align*}
W_m^\delta(p) := (W_m * \eta_\delta)(p) + \delta |p|^2,
\end{align*}
where $\eta_\delta$ is the standard mollifier with radius $\delta$,
and let $u_m^\delta$ be the unique viscosity solution of
\begin{align}
\label{quad approximation}
\left\{
\begin{aligned}
u_t + F(\nabla u, \divo \nabla_p W_m^\delta(\nabla u_m)) &= 0, && \text{in $\R^n \times (0,\infty)$,}\\
u_m(\cdot, 0) &= u_{0, m}, && \text{in $\R^n$}.
\end{aligned}
\right.
\end{align}
From the standard theory we have that $u_m^\delta \rightrightarrows u_m$ as $\delta \to 0$ locally uniformly on $\R^n
\times [0, \infty)$.

Suppose now that $\varphi$ is a stratified test function at $(\hat x, \hat t)$ with gradient $\hat
p$, as in the proof in Section~\ref{se:stability case (i)}, Case (i) above, for a subsolution.
We proceed as in that proof, but we use an additional perturbation of the test function by solving
the resolvent problem for the energy $E_m^\delta := E_{W_m^\delta(\cdot - \hat p) - W(\hat
p)}$:
we define the unique solution $\psi_{a,m}^\delta \in L^2(\Gamma)$ of
\begin{align*}
\psi_{a,m}^\delta + a \partial E_m^\delta(\psi_{a,m}^\delta) \ni \psi,
\end{align*}
where $\psi$ and $\Gamma$ were given in \eqref{def-psi}. Recall that $\psi_{a,m}^\delta \in
C^{2,\gamma}(\Gamma)$ by the elliptic regularity.

We can apply Proposition~\ref{pr:resolvent-problems} to $E_m^\delta$ and $E_m$ for fixed $m$ in the
limit $\delta \to 0$. We in particular have $\psi_{a,m}^\delta \rightrightarrows \psi_{a,m}$ and
$h_{a,m}^\delta \rightrightarrows h_{a,m}$ as $\delta \to 0$ for fixed $a, m$.

Due to the Mosco convergence of $E_m$ to $E_{\hat p}$ in Lemma~\ref{le:lingrowthapproximation}, we also can apply
Proposition~\ref{pr:resolvent-problems} to $E_m$ and $E_{\hat p}$ in the limit $m \to \infty$.

We now fix $a$ and $z'$ as in \eqref{choice-of-z'}.
Due to the uniform convergence $\psi_{a,m}^\delta \rightrightarrows \psi_{a,m}$ as $\delta \to 0$
and $\psi_{a,m} \rightrightarrows \psi_a$
as $m\to\infty$, there exists $(x_a, t_a) \in A_{a,z'}$ and
a sequence $(x_{a,m}, t_{a,m})$ (for a subsequence of $m$) of local
maxima of
\begin{align*}
(x,t) \mapsto u_m(x + \hat x,t + \hat t) - \psi_{a,m}(x + \TT_V z') - \hat p \cdot x - g(t + \hat t)
\end{align*}
such that $(x_{a,m}, t_{a,m}) \to (x_a, t_a)$ as $m\to\infty$
(along a subsequence), and for each $m$ in this subsequence there exist a sequence
$(x_{a,m}^\delta, t_{a,m}^\delta)$ (for a subsequence of $\delta$ as $\delta \to 0$) of local
maxima of
\begin{align*}
(x,t) \mapsto u_m^\delta(x + \hat x,t + \hat t) - \psi^\delta_{a,m}(x + \TT_V z') - \hat p \cdot x
- g(t + \hat t),
\end{align*}
such that $(x_{a,m}^\delta, t_{a,m}^\delta) \to (x_{a,m}, t_{a,m})$ as $\delta \to 0$ (along a
subsequence).

Since $u_m^\delta$ is a viscosity solution of \eqref{quad approximation}, we have
\begin{align*}
&g'(t^\delta_{a,m} + \hat t) + F(\nabla \psi^\delta_{a,m}(x^\delta_{a,m} + \TT_V z') + \hat p,
h^\delta_{a,m}(x^\delta_{a,m} + \TT_V z'))\\
&
\begin{aligned}
= g'(t^\delta_{a,m} + \hat t) + F\Big(&\nabla \psi^\delta_{a,m}(x^\delta_{a,m} + \TT_V z') + \hat p,\\
&\trace\bra{(\nabla_p^2
W^\delta_m)(\nabla \psi^\delta_{a,m} + \hat p) \nabla^2 \psi^\delta_{a,m}}(x^\delta_{a,m} + \TT_V
z')\Big) \leq 0.
\end{aligned}
\end{align*}
Sending $\delta \to 0$ along a \emph{subsequence} and using the uniform convergence of $h^\delta_{a,m}
\rightrightarrows h_{a,m}$, we can find $p_{a,m}$ with $|p_{a,m} - \hat p| \leq C \alpha$ such that
\begin{align}
\label{ham subsol}
&g'(t_{a,m} + \hat t) + F(p_{a,m},
h_{a,m}(x_{a,m} + \TT_V z')) \leq 0.
\end{align}
Sending $m \to \infty$ along a \emph{subsequence}, we obtain $p_a$ and \eqref{subsol ha}.
Then we finish the proof as in the proof of Theorem~\ref{th:stability quadratic} for Case(i).

Case (ii) as well as the curvature-free case are both straightforward.
\end{proof}

\section{Well-posedness}
\label{sec:well-posedness}

Once the stability with respect to the approximation of the energy density $W$ is established,
we get existence of solutions as in \cite{GGP13JMPA}.

\begin{theorem}[Well-posedness]
\label{th:well-posedness}
Let $W: \R^n \to \R$ be a positively one-homogeneous convex polyhedral function such that the
conclusion of
Corollary~\ref{co:approximate pair sliced} holds for $1 \leq k \leq n-1$, and let $F$ be of curvature-free type at
$p_0 = 0$.
Then for given $u_0 \in C(\R^n)$ such that $u \equiv
c$ on $\R^n \setminus K$ for some compact $K \subset \R^n$ and $c \in \R$ there exists a unique viscosity solution
of
\begin{align}
\label{limit problem}
\left\{
\begin{aligned}
u_t + F(\nabla u, \divo \partial W(\nabla u)) &= 0, && \text{in $\R^n \times (0, \infty)$},\\
u(\cdot, 0) &= u_0, && \text{in $\R^n$.}
\end{aligned}
\right.
\end{align}
Moreover, if $u_0$ is Lipschitz, then
\begin{align*}
\norm{\nabla u(\cdot, t)}_\infty \leq \norm{\nabla u_0}_\infty, \qquad t \geq 0.
\end{align*}
\end{theorem}

\begin{proof}
We follow a standard approximation argument using the stability result from
Section~\ref{se:stability}.
Let $W_m \in C(\R^n) \cap C^2(\R^n \setminus \set0)$ be a sequence of convex positively
one-homogeneous functions with $\set{W_m \leq 1}$ strictly convex, such that $W_m
\rightrightarrows W$ on $\cl B_1(0)$.
We can find the unique viscosity solutions $u_m$ of the problem
\begin{align*}
\left\{
\begin{aligned}
u_t + F(\nabla u, \divo \nabla_p W_m(\nabla u)) &= 0, && \text{in $\R^n \times (0, \infty)$},\\
u(\cdot, 0) &= u_0, && \text{in $\R^n$.}
\end{aligned}
\right.
\end{align*}
We define the limits
\begin{align*}
\overline{u} &:= \halflimsup_{m\to\infty} u_m, &
\underline{u} &:= \halfliminf_{m\to\infty} u_m.
\end{align*}
These limits are well-defined since $u_m$ are uniformly bounded.
By the stability result Theorem~\ref{th:linear growth stability}, we see that $\overline u$ is a viscosity subsolution and
$\underline u$ is a viscosity supersolution of \eqref{limit problem}.

We need to prove that $\overline u$ and $\underline u$ have the correct initial data. We can
compare $u_m$ with translations of barriers
\begin{align*}
\psi^+_{m; a, b} &:= a (W_m^\circ(x - x_0) - bt)_+, &
\psi^-_{m; a, b} &:= -a (-W_m^\circ(-x + x_0) + bt)_-,
\end{align*}
where $W_m^\circ$ is the polar of $W_m$.
The comparison with such barriers shows that $\overline u(\cdot, 0) = \underline u(\cdot, 0) =
u_0$,
and for every $T > 0$ there exists a compact set $K_T \subset \R^n$ such that $\overline u =
\underline u = c$ on $(\R^n \setminus K_T) \times [0, T]$.

Then the comparison principle Theorem~\ref{th:comparison principle} yields $\overline u \leq
\underline u$ and thus $u := \overline u = \underline u$ is the unique solution of \eqref{limit
problem}.

The Lipschitz continuity follows from the comparison principle.
\end{proof}

We now present the proofs of Theorems~\ref{th:unique existence} and \ref{th:convergence}.

\begin{proof}[Proof of Theorems~\ref{th:unique existence} and \ref{th:convergence}]
Find $R > 0$ such that $D_0 \subset B_{R/2}(0)$. Let $F$ and $W$ be as in \eqref{geometric F}.
Let $u_0$ be a continuous function with $D_0 = \set{u_0 > 0}$ such that $u_0 = -c$ for some $c > 0$
for $\abs x \geq R$. For instance, take a cutoff of the signed distance function to $\Gamma_0$,
$u_0(x) :=
-\min(\dist(x, D_0), 1) + \dist(x, D_0^c)$.
Then there is a unique solution $u$ of \eqref{P} with initial data $u_0$ by
Theorem~\ref{th:well-posedness}. This establishes the existence of a level set flow
$\set{\Gamma_t}_{t\geq0}$ as $\Gamma_t := \set{x: u(x, t) = 0}$.

We therefore only need to show that the zero level set of $u$ does not depend on $u_0$. For this we
simply argue as in \cite[Section~4.1.1]{G06} to show that $\theta \circ u := \theta(u)$ is also a
viscosity solution of \eqref{P} for any continuous, nondecreasing $\theta$.  Then for any given two
continuous level set functions $u_0$, $\tilde u_0$ of $\Gamma_0$ we can find $\theta_1$, $\theta_2
\in C(\R)$, strictly increasing, such that $\theta_1 \circ u_0 \leq \tilde u_0$ and $\theta_2 \circ
\tilde u_0 \leq u_0$. Let $u$, $\tilde u$ be the two unique viscosity solutions of \eqref{P} with
initial data $u_0$, $\tilde u_0$, respectively. By the comparison principle
Theorem~\ref{th:comparison principle} we get $\theta_1 \circ u \leq
\tilde u$ and $\theta_2 \circ \tilde u \leq u$. Since $\theta_1 \circ u$ and $\theta_2 \circ \tilde
u$ have the same zero level sets as $u$ and $\tilde u$, respectively, we conclude that the
level set flow $\set{\Gamma_t}_{t \geq 0}$ is unique.

The stability result of Theorem~\ref{th:convergence} follows from Theorem~\ref{th:linear growth
stability} and the comparison principle Theorem~\ref{th:comparison principle}.
\end{proof}

\subsection*{Acknowledgments}

Y. G. is partially supported by Grants-in-Aid for
Scientific Research No. 26220702 (Kiban S), No. 23244015 (Kiban A) and
No. 25610025 (Houga) of Japan Society for the Promotion of Science (JSPS).
N. P. is partially supported by JSPS KAKENHI Grant Number 26800068 (Wakate B).

\end{document}